\documentclass[12pt, twoside]{article}
\usepackage{amsmath,amsthm,amssymb}
\usepackage{times}
\usepackage{enumerate}

\usepackage{amscd,mathtools,stmaryrd,mathscinet}
\usepackage{tabularx}

\def\titlerunning#1{\gdef\titrun{#1}}
\makeatletter
\def\author#1{\gdef\autrun{\def\and{\unskip, }#1}\gdef\@author{#1}}
\def\address#1{{\def\and{\\\hspace*{18pt}}\renewcommand{\thefootnote}{}%
\footnote {#1}}%
\markboth{\autrun}{\titrun}}
\makeatother
\def\email#1{e-mail: #1}
\def\subjclass#1{{\renewcommand{\thefootnote}{}%
\footnote{\emph{Mathematics Subject Classification (2010):} #1}}}
\def\keywords#1{\par\medskip
\noindent\textbf{Keywords.} #1}

\newcommand{\resunip}{\operatorname{resunip}}
\newcommand{\lie}[1]{\mathfrak{#1}}
\newcommand{\clos}{\overline{\F}_p}
\newcommand{\Liec}{\mathbf{L}}
\newcommand{\grpc}{\mathbf{G}}
\newcommand{\SGRU}{\mathcal{SG}_{\operatorname{rug}}(G(\Z_p))}
\newcommand{\LSAN}{\mathcal{SA}_{\operatorname{rng}}(\lie{g}_{\Z_p})}
\newcommand{\nd}{non-degenerate}
\newcommand{\Res}{\operatorname{Res}}
\newcommand{\der}{\operatorname{der}}
\newcommand{\unip}{\operatorname{unip}}
\newcommand{\nilp}{\operatorname{nilp}}
\newcommand{\Exp}{\operatorname{Exp}}
\newcommand{\resuni}{\operatorname{resunip}}
\newcommand{\resnilp}{\operatorname{resnilp}}
\newcommand{\BG}{(BG)}
\newcommand{\gl}{\mathfrak{gl}}
\newcommand{\End}{\operatorname{End}}
\newcommand{\chr}{\operatorname{char}}
\newcommand{\bjct}[2]{\xrightleftharpoons[#2]{#1}}
\newcommand{\FC}{(FC)}
\newcommand{\maxfam}{\mathcal{MAX}}
\newcommand{\nmbr}{n}
\newcommand{\aff}{\mathbf A}
\newcommand{\Spec}{\operatorname{Spec}}
\newcommand{\C}{{\mathbb C}}

\newcommand{\R}{{\mathbb R}}
\newcommand{\Z}{{\mathbb Z}}
\newcommand{\Q}{{\mathbb Q}}
\newcommand{\A}{{\mathbb A}}
\newcommand{\K}{\mathbf{K}}
\newcommand{\bs}{\backslash}
\newcommand{\reg}{\operatorname{reg}}
\newcommand{\sprod}[2]{\left\langle#1,#2\right\rangle}
\newcommand{\QF}{\mathfrak{Q}}
\newcommand{\tr}{\operatorname{tr}}
\newcommand{\Ad}{\operatorname{Ad}}
\newcommand{\Lie}{\operatorname{Lie}}
\newcommand{\Hom}{\operatorname{Hom}}
\newcommand{\Ker}{\operatorname{Ker}}
\newcommand{\vol}{\operatorname{vol}}
\newcommand{\SL}{\operatorname{SL}}
\newcommand{\GL}{\operatorname{GL}}
\newcommand{\altp}{p'}
\newcommand{\card}[1]{\lvert#1\rvert}
\newcommand{\abs}[1]{\lvert#1\rvert}
\newcommand{\norm}[1]{\lVert#1\rVert}
\newcommand{\ad}{\operatorname{ad}}
\newcommand{\rest}{\big|}
\newcommand{\fin}{{\operatorname{fin}}}
\newcommand{\level}{\operatorname{lev}}
\newcommand{\sgrmx}{\mathcal{MSGR}}
\newcommand{\EXPG}{exp.~gen.}
\newcommand{\NILG}{nilp.~gen.}
\newcommand{\algnori}[1]{\tilde{#1}}
\newcommand{\SC}{\operatorname{sc}}
\newcommand{\sm}[4]{\left(\begin{smallmatrix}{#1}&{#2}\\{#3}&{#4}\end{smallmatrix}\right)}
\newcommand{\F}{\mathbb{F}}
\newcommand{\Pro}{\operatorname{Pr}}

\newtheorem{theorem}{Theorem}[section]
\newtheorem{lemma}[theorem]{Lemma}
\newtheorem{proposition}[theorem]{Proposition}
\newtheorem{remark}[theorem]{Remark}

\newtheorem{definition}[theorem]{Definition}
\newtheorem{corollary}[theorem]{Corollary}

\newenvironment{thmbis}[1]
  {\renewcommand{\thetheorem}{\ref{#1}$'$}%
   \addtocounter{theorem}{-1}%
   \begin{theorem}}
  {\end{theorem}}

\begin{document}


\baselineskip=17pt


\titlerunning{Approximation principle}

\title{An approximation principle for congruence subgroups}

\author{Tobias Finis
\and
Erez Lapid}

\date{\today}

\maketitle

\address{Tobias Finis: Universit\"at Leipzig, Mathematisches Institut, Postfach 100920, D-04009 Leipzig, Germany; \email{finis@math.uni-leipzig.de}
\and
Erez Lapid: Department of Mathematics, The Weizmann Institute of Science, Rehovot 7610001, Israel; \email{erez.m.lapid@gmail.com}}



\begin{abstract}
The motivating question of this paper is roughly the following: given a flat group scheme $G$ over $\Z_p$,
$p$ prime, with semisimple generic fiber $G_{\Q_p}$, how far are open subgroups of $G(\Z_p)$ from subgroups of the form
$X(\Z_p)\K_p(p^n)$, where $X$ is a subgroup scheme of $G$ and $\K_p(p^n)$ is the principal
congruence subgroup $\Ker(G(\Z_p)\rightarrow G(\Z/p^n\Z))$?
More precisely, we will show that for $G_{\Q_p}$ simply connected there exist constants $J\ge1$ and $\varepsilon>0$, depending only on $G$,
such that any open subgroup of $G (\Z_p)$ of level $p^n$ admits an open subgroup of index $\le J$ which is contained
in $X(\Z_p)\K_p(p^{\lceil \varepsilon n\rceil})$ for some proper, connected algebraic subgroup $X$ of $G$ defined over $\Q_p$.
Moreover, if $G$ is defined over $\Z$, then $\varepsilon$ and $J$ can be taken independently of $p$.

We also give a correspondence between natural classes of $\Z_p$-Lie subalgebras of $\lie{g}_{\Z_p}$ and of closed subgroups of $G(\Z_p)$ that
can be regarded as a variant over $\Z_p$ of Nori's results on the structure of finite subgroups of $\GL(N_0,\F_p)$ for large $p$ \cite{MR880952}.

As an application we give a bound for the volume of the intersection of a
conjugacy class in the group $G (\hat{\Z}) = \prod_p G (\Z_p)$, for $G$ defined over $\Z$, with an arbitrary open subgroup.
In a companion paper, we apply this result to the limit multiplicity problem
for arbitrary congruence subgroups of the arithmetic lattice $G (\Z)$.

\end{abstract}

\numberwithin{equation}{section}

\setcounter{tocdepth}{1}
\tableofcontents

\section{Introduction}

In this paper we study closed subgroups of the profinite groups $G (\Z_p)$, where $G$ is a semisimple algebraic group defined over $\Q$ with
a fixed $\Z$-model, and $p$ is a prime. The first main topic, which is treated in \S\S \ref{sectionapprox} and \ref{SectionApproxLevel0},
is the comparison between arbitrary open subgroups $H$ of $G (\Z_p)$ and
special open subgroups of the form $(X(\Q_p) \cap G (\Z_p)) \K_p (p^n)$, where $X$ is a proper, connected algebraic subgroup of $G$ defined over $\Q_p$, and
$\K_p(p^n) = \Ker(G(\Z_p)\rightarrow G(\Z/p^n\Z))$ is the principal congruence subgroup of $G (\Z_p)$ of level $p^n$.
These special open subgroups are ubiquitous in the literature, especially in the case where $X$ is a parabolic subgroup of $G$,

To motivate our theorem, recall that commensurability classes of \emph{closed} subgroups of $G (\Z_p)$ are in one-to-one correspondence with $\Q_p$-Lie subalgebras of
$\lie{g}_{\Q_p} = \Lie_{\Q_p} G$ (cf. \cite[Theorem 9.14]{MR1720368}),
and that by Chevalley's theorem \cite[Corollary 7.9]{MR1102012} any proper, $\Q_p$-Lie subalgebra of $\lie{g}_{\Q_p}$ is contained in the Lie algebra of a proper,
connected algebraic subgroup $X$ of $G$ defined over $\Q_p$.
These statements do not say anything non-trivial about open subgroups.
Recall that the level of an open subgroup $H$ of $G (\Z_p)$ is defined as $\min\{l=p^n\ge1:\K_p (l) \subset H\}$.
Our first main result (Theorem \ref{TheoremAlgebraicLevel0}) says that if $G$ is simply connected then there exist constants $J\ge1$ and $\varepsilon>0$, depending only on $G$,
such that any open subgroup of $G (\Z_p)$ of level $p^n$ admits an open subgroup of index $\le J$ which is contained
in $(X(\Q_p) \cap G (\Z_p)) \K_p(p^{\lceil \varepsilon n\rceil})$ for some proper, connected algebraic subgroup $X$ of $G$ defined over $\Q_p$.

Simple considerations show that it is not possible to take $\varepsilon=1$ here, and that in fact necessarily $\varepsilon \le \frac12$ for every $G$
(cf. Remark \ref{RemarkAlgebraic} and Lemma \ref{RemarkEpsilonG}).
The question of the optimal value of $\varepsilon$ for a given $G$ (or of the optimal asymptotic value
as $n \to \infty$) remains unresolved except in the case of $G = \SL (2)$ (cf. Lemma \ref{lem: SL2}).

We also have a variant of Theorem \ref{TheoremAlgebraicLevel0} (Theorem \ref{TheoremAlgebraic}) for subgroups of the pro-$p$ groups $\K_p (\altp)$,
where $\altp=p$ if $p$ is odd and $\altp=4$ if $p=2$. Here $G$ does not need to be simply connected, and we do not need to pass to a finite index subgroup.
In fact, we first prove Theorem \ref{TheoremAlgebraic} in \S \ref{sectionapprox} and then modify our arguments to deal with arbitrary open subgroups in \S \ref{SectionApproxLevel0}.
Our proofs are based on the general correspondence between (a large class of) subgroups of $\K_p(\altp)$ and
$\Z_p$-Lie subalgebras of $\altp \lie{g}_{\Z_p}$.
This correspondence reduces Theorem \ref{TheoremAlgebraic} to a $\Z_p$-Lie algebra analog (Theorem \ref{thm: Liealgebra}).
The latter is deduced from a general lifting result of M.~Greenberg \cite{MR0351994}, which in the case at hand allows one (under natural conditions) to lift Lie subalgebras modulo $p^\nu$
to Lie subalgebras in characteristic zero.
To infer Theorem \ref{TheoremAlgebraicLevel0} from Theorem \ref{TheoremAlgebraic} we use in addition the results of Nori \cite{MR880952} on the structure of finite subgroups of $G (\F_p)$ for large $p$.

The discussion above spurs the next result (\S \ref{sectionliealgebra}) which is a natural correspondence (valid for all $p$ large enough with respect to $G$) between the set of
$\Z_p$-Lie subalgebras of $\lie{g}_{\Z_p}$ that are generated by residually nilpotent elements and the set of closed subgroups of $G(\Z_p)$ that are generated by residually unipotent elements.
Note that the analogous correspondence over $\F_p$ is a consequence of Nori's results, which not surprisingly, play an important role in the proof of the $\Z_p$-analog.
In addition, our theorem generalizes (and relies on) the previously known correspondence between closed pro-$p$ subgroups and Lie subalgebras
consisting of residually nilpotent elements, which is given by the application of the logarithm and exponential maps \cite{MR1345293,MR2126210,MR2554763}.
Apart from these results, our proof of the correspondence uses Steinberg's algebraicity theorem \cite[Theorem 13.3]{MR0230728} and Jantzen's semisimplicity theorem \cite{MR1635685}
for characteristic $p$ representations of the groups $X (\F_p)$, where $X$ is semisimple and simply connected over $\F_p$,
as well as the Curtis--Steinberg--Tits presentations of these groups \cite{MR0153677,MR0188299,MR630615,MR637803}.
In a somewhat different direction, results on closed subgroups of $G (\Z_p)$ connected to Nori's theorems have been obtained by Larsen \cite{MR2832632}.

As an application of the approximation theorems, we give in \S \ref{MainApplication} a bound for the volume of the intersection of a
conjugacy class in the profinite group $G (\hat{\Z}) = \prod_p G (\Z_p)$ with an arbitrary open subgroup.
(The results of \S\ref{sectionliealgebra} are not logically necessary for either Theorem \ref{TheoremAlgebraicLevel0} or this application.)
The main technical result is Theorem \ref{thm: mainbnd} and the bound in question is given in
Corollary \ref{CorConjugacyClass} (which is stated for a slightly more general class of groups $G$). In Corollary \ref{CorLattices}, we
give a variant of this result which deals with arithmetic lattices in semisimple Lie groups.
We note that essentially the same result has been obtained independently in \cite[\S 5]{ABBGNRS} by a different method.
A rough form of the resulting bound can be stated as follows.
For an arbitrary group $\Gamma$, a finite index subgroup $\Delta$, and an element $\gamma \in \Gamma$ set
\[
c_\Delta (\gamma) =  \abs{\{ \delta \in \Gamma / \Delta \, : \, \delta^{-1} \gamma \delta \in \Delta\}},
\]
which is also the number of fixed points of $\gamma$ in the permutation representation of $\Gamma$ on the finite set $\Gamma / \Delta$.
Let $G$ be semisimple and simply connected and assume that for no $\Q$-simple factor $H$ of $G$ the group $H (\R)$ is compact.
Let $\K = G (\hat{\Z}) \subset G (\A_{\fin})$
and let $\Gamma = G (\Z) = G (\Q) \cap \K$, which is an arithmetic lattice in the connected semisimple Lie group $G (\R)$.
For any open subgroup $K \subset \K$ let $\Delta = G (\Q) \cap K$ be the associated finite index subgroup of $\Gamma$.
By definition, these subgroups are the congruence subgroups of $\Gamma$. Then
there exists a constant $\varepsilon > 0$, depending only on $\Gamma$, such that
for all congruence subgroups $\Delta$ of $\Gamma$ and all $\gamma \in \Gamma$ that are not contained in any proper normal subgroup of $G$ (which we may assume to be defined over $\Q$) we have
\[
\frac{c_\Delta (\gamma)}{[\Gamma : \Delta]} \ll_{\Gamma,\gamma} [\Gamma : \Delta]^{-\varepsilon}.
\]
Moreover, we can also control the dependence of the implied constant on $\gamma$ very explicitly.
In the companion paper \cite{1504.04795} we will use this result to extend the (partly conditional) results of \cite{MR3352530} about limit multiplicities
for subgroups of the lattice $\Gamma$ from the case of principal congruence subgroups to arbitrary congruence subgroups.

The appendix contains some elementary bounds on the number of solutions of polynomial congruences that are needed in \S \ref{MainApplication}.
These results are probably well known but we were unable to retrieve them in the literature in a form convenient for us.

It would be of interest to study the case of groups defined over function fields over finite fields and over the corresponding local fields (cf. \cite{MR1264349, MR1422889, MR1971298}).
We do not say anything about this here.

We thank the Centre Interfacultaire Bernoulli, Lausanne,
and the Max Planck Institute for Mathematics, Bonn, where a part of this paper was worked out.
We thank Yiftach Barnea, Vladimir Berkovich, Tsachik Gelander, Ehud Hrushovski, Michael Larsen, Alex Lubotzky and Aner Shalev for useful discussions and interest in the subject matter of this paper.

\section{The pro-$p$ approximation theorem} \label{sectionapprox}

\subsection{The general setup} \label{subspropsetup}
Throughout \S\S \ref{sectionapprox} -- \ref{sectionliealgebra} (except for \S \ref{subsectionfiniteness}, where we consider a more general situation)
let $G = G_{\Q}$ be a connected semisimple algebraic group defined over $\Q$, together with a fixed embedding
$\rho_0: G \to \GL (N_0)$ defined over $\Q$.
(Actually, only the image of $\rho_0$ inside $\GL (N_0)$ is relevant for our purposes.)
The choice of $\rho_0$ endows $G$ with the structure of a flat group scheme over $\Z$, which we denote by $G_{\Z}$.
Concretely, it is given as $\Spec \Z[\GL(N_0)]/(I\cap\Z[\GL(N_0)])$, where
$R[\GL(N_0)]=R[y,x_{ij},i,j=1,\dots,N_0]/(1-y\det x)$ for any ring $R$
and $I$ is the ideal of $\Q[\GL(N_0)]$ defining the image of $\rho_0$.
If $R$ is a ring which is flat over $\Z$, i.e., if the total quotient ring $Q(R)$ is a $\Q$-algebra, then
$G(R)=\{g\in G(Q(R)):\rho_0(g)\in\GL(N_0,R)\}$.
For any commutative ring $R$ let $G_R = G_{\Z} \times \Spec R$ be the base extension to $R$
of the group scheme $G_{\Z}$. If $R$ is a field
of characteristic zero, then $G_R$ is a connected semisimple group over $R$.
For almost all primes $p$, the group scheme $G_{\Z_p}$ is smooth over $\Z_p$ and $G_{\F_p}$ is a connected semisimple
group defined over $\F_p$ \cite[\S 3.8.1, \S 3.9.1]{MR546588}.

Note that we can always choose $\rho_0$ such that $G$ is smooth over $\Z$, and therefore smooth over $\Z_p$ for all $p$.
(Indeed, by [ibid., \S 3.4.1], for each $p$ there exists a smooth group scheme over $\Z_p$ with generic fiber $G_{\Q_p}$, and we can glue
these as in \cite{MR1369418} to obtain a smooth $\Z$-model for $G$.)
However, only under strong restrictions on $G$ it is possible to have $G$ smooth over $\Z$
and $G_{\F_p}$ connected semisimple for all $p$ \cite{MR1369418}.

For each prime $p$ we set
\[
\K_p := G (\Z_p) = \{g \in G (\Q_p) \, : \, \rho_0 (g) \in \GL (N_0, \Z_p) \},
\]
a compact open subgroup of $G (\Q_p)$.
The principal congruence subgroups of $\K_p$ are defined by
\[
\K_p (p^n) := \Ker \left( G(\Z_p) \rightarrow G(\Z/p^n\Z) \right) = \{g \in \K_p \, : \, \rho_0 (g) \equiv 1 \pmod{p^n} \}, \quad n \ge 1.
\]
They are clearly pro-$p$ normal subgroups of $\K_p$, and form a neighborhood base of the identity element. For $n=0$ we set $\K_p (1) := \K_p$.
\begin{definition}
Let $H$ be an open subgroup of $\K_p$. The level of $H$ is
\[
\min\{q=p^n:\K_p(q)\subset H\}.
\]
\end{definition}

The main result of this section is the following theorem on open subgroups of the pro-$p$ groups $\K_p (\altp)$
where $\altp=p$ if $p$ is odd and $\altp=4$ if $p=2$.
As usual, we denote by $\lfloor x \rfloor$ (resp., $\lceil x \rceil$) the largest integer $\le x$ (resp., the smallest integer $\ge x$).

\begin{theorem} \label{TheoremAlgebraic}
There exists a constant $\varepsilon>0$, depending only on $G$ and $\rho_0$, such that for any prime $p$ and any open subgroup
$H \subset \K_p (\altp)$ of level $p^n$ there exists a proper, connected algebraic subgroup $X$ of $G$ defined over $\Q_p$
such that $H\subset (X (\Q_p)\cap \K_p(p')) \K_p (p^{\lceil \varepsilon n \rceil})$.
\end{theorem}

Alternatively, we can reformulate the theorem as a statement on the projective system consisting of the finite groups $G (\Z / p^n \Z)$ and the
canonical reduction homomorphisms $\pi_{n,m}: G (\Z / p^n \Z) \to G (\Z / p^m \Z)$, $n \ge m$.
Note that $\K_p=\varprojlim G(\Z/p^n\Z)$.
We write $\pi_{\infty,n}$ for the reduction homomorphism $\K_p \to G (\Z / p^n \Z)$. A subgroup $\bar{H}$ of
$G (\Z / p^n \Z)$ is called \emph{essential} if it is not the pullback via $\pi_{n,n-1}$ of a subgroup of
$G (\Z / p^{n-1} \Z)$, i.e. if $\bar{H}$ does not contain the kernel of $\pi_{n,n-1}$.
Then an equivalent form of Theorem \ref{TheoremAlgebraic} is the following

\begin{thmbis}{TheoremAlgebraic}
There exists a constant $\varepsilon>0$, depending only on $G$ and $\rho_0$, such that for any prime $p$, an integer $n\ge1$ and an essential subgroup
$\bar{H}$ of $\pi_{\infty, n} (\K_p (\altp)) \subset G (\Z / p^n \Z)$ there exists a
proper, connected algebraic subgroup $X$ of $G$ defined over $\Q_p$ such that
$\pi_{n,\lceil \varepsilon n \rceil} (\bar{H}) \subset \pi_{\infty,\lceil \varepsilon n \rceil} (X (\Q_p) \cap \K_p(p'))$.
\end{thmbis}

Note here that if $G_{\Z_p}$ is smooth (which is the case for almost all $p$) then the
$p$-group $\pi_{\infty, n} (\K_p (\altp))$ coincides with the kernel of
the homomorphism $\pi_{n,\epsilon_p}: G (\Z / p^n \Z) \to G (\Z / p' \Z)$
where $\epsilon_p=1$ if $p$ is odd and $\epsilon_2=2$.

\begin{remark} \label{RemarkAlgebraic}
In the statement of Theorem \ref{TheoremAlgebraic} there is no harm in assuming that $n\ge n_0$ for some fixed positive integer $n_0$ depending only on $G$.
In fact, for any integer $n \ge 1$ let $\nu (n)$ be the largest integer $\nu$ such that for any $p$ and any open subgroup
$H \subset \K_p (\altp)$ of level $p^n$ there exists a proper, connected algebraic subgroup $X$ of $G$ defined over $\Q_p$ such that
$H \subset (X (\Q_p)\cap \K_p (p')) \K_p (p^{\nu})$.
Then the content of Theorem \ref{TheoremAlgebraic} is that $\varepsilon_G := \liminf_{n \to \infty} \frac{\nu(n)}{n} > 0$.
It is natural to ask what further can be said about $\varepsilon_G$.
It is easy to see that $\varepsilon_G \le \frac12$ for any $G$ (see Lemma \ref{RemarkEpsilonG} below).
Thus, the introduction of the factor $\K_p (p^{\lceil \varepsilon n \rceil})$ instead of $\K_p (p^n)$ is essential
for Theorem \ref{TheoremAlgebraic} to hold.
For $G = \SL (2)$ we have $\varepsilon_G = \frac12$
(see Lemma \ref{lem: SL2} below).
With our current knowledge we cannot rule out the possibility that $\varepsilon_G=\frac12$ for all $G$.
However, our method of proof falls short of even proving that $\varepsilon_G$ is bounded away from $0$ independently of $G$.
\end{remark}

Theorem \ref{TheoremAlgebraic} will be proved in \S\ref{subsliealgebra} below.

\subsection{Uniform pro-$p$ groups} \label{subsuniform}
A key ingredient in the proof of Theorem \ref{TheoremAlgebraic} is a linearization argument.
In order to carry it out, we need to recall some further structural properties of the congruence subgroups $\K_p (p^n)$.
We first recall some standard terminology and basic results from \cite{MR1720368}
which have their root in the fundamental work of Lazard \cite{MR0209286}.

For any group $H$ and any integer $n$ we write $H^{\{n\}}=\{x^n:x\in H\}$ and denote by $H^n$ the group generated by $H^{\{n\}}$.
We will use the standard notation $[x,y]=xyx^{-1}y^{-1}$ for the commutator.
Also, for any subgroups $H_1$, $H_2$ of $H$ we write $[H_1,H_2]$ for the subgroup generated by the commutators
$[h_1,h_2]$, $h_1\in H_1$, $h_2\in H_2$.

In the following, all pro-$p$ groups that we consider will be implicitly assumed to be topologically finitely generated.

\begin{definition}
Let $L$ be a pro-$p$ group.
\begin{enumerate}
\item $L$ is called \emph{powerful} if $[L,L] \subset L^{\altp}$.
\item $L$ is called \emph{uniformly powerful} (or simply \emph{uniform}) if it is powerful and torsion-free (cf. [ibid., Theorem 4.5]).
\item The \emph{rank} of $L$ is the smallest integer $n$ such that every closed (or alternatively by [ibid., Proposition 3.11], every open) subgroup
of $L$ can be topologically generated by $n$ elements.\footnote{We say that $L$ is of infinite rank if no such integer exists. In this paper, we will only work with pro-$p$ groups of finite rank.}
\item A subgroup $S$ of $L$ is called \emph{isolated}
if the only elements $x \in L$ with $x^p \in S$ are the elements of $S$.
\end{enumerate}
Similarly, a free $\Z_p$-Lie algebra $M$ of finite rank is called powerful if $[M,M]\subset \altp M$
and a subalgebra $S$ of $M$ is called isolated if $x\in M$ and $px\in S$ implies $x\in S$.
\end{definition}

The prime examples of uniform pro-$p$ groups [ibid., Theorem 5.2] are the groups
\[
\Gamma (N_0, p^{n}) = \{g\in\GL (N_0, \Z_p) \, : \, g \equiv 1 \pmod{p^{n}}\}, \quad n \ge \epsilon_p.
\]

We list the following properties of powerful and uniform pro-$p$ groups. (The unreferenced parts are clear.)

\begin{lemma} \label{lem: basunif}
\begin{enumerate}
\item ([ibid., Theorem 3.6 and Proposition 1.16]) If $L$ is a powerful pro-$p$ group then $L^{p^i}=L^{\{p^i\}}$, $i\ge0$.
These groups form a neighborhood base of the identity.
\item ([ibid., Theorem 3.8]) The rank of a uniform pro-$p$ group $L$ is the minimum number of generators of $L$,
which is also the dimension of the $\mathbb{F}_p$-vector space $L/L^p$.
\item Let $L$ be a free $\Z_p$-module of finite rank.
Then the map $U\mapsto U\cap L$ defines a bijection between the set of vector subspaces of the $\Q_p$-vector space $L \otimes \Q_p$
and the set of isolated $\Z_p$-submodules of $L$.
Moreover, if $L$ is a $\Z_p$-Lie algebra then this map induces a bijection between the $\Q_p$-Lie subalgebras of $L\otimes\Q_p$
and the isolated closed Lie subalgebras of $L$.
\item An isolated subgroup of a uniform group is uniform.
Similarly, an isolated subalgebra of a powerful Lie algebra is powerful.
\item The only isolated open subgroup of a pro-$p$ group $L$ is the group $L$ itself.
\item ([ibid., Theorem 3.10]) \label{LemmaUniform}
For any positive integer $d$ there exists a non-negative integer $N=N(d)$ (in fact, we may take $N (d) = d (\lceil \log_2 d \rceil + \epsilon_p - 1)$) such that any open subgroup
$H$ of a uniform pro-$p$ group $L$ of rank $d$ contains a uniform characteristic open subgroup $V$ of index $\le p^{N(d)}$
(which implies in particular that $H^{p^N} \subset V$).
\end{enumerate}
\end{lemma}

For any commutative ring $R$ let $\gl(N_0,R)$ be the Lie algebra of $N_0\times N_0$ matrices over $R$ with the usual Lie bracket.
Let $\lie{g}$ be the Lie algebra of $G$ over $\Q$, regarded as a subalgebra of $\gl (N_0, \Q)$, and let
$\lie{g}_\Z = \lie{g} \cap \gl(N_0, \Z)$ be its $\Z$-form, which is
a free $\Z$-module of rank $d = \dim G$. Let $\lie{g}_{R} = \lie{g}_\Z \otimes R$ for any ring $R$.

Let $\exp$ and $\log$ be the power series
$\exp x=\sum_{n=0}^\infty x^n/n!$ and $\log x=-\sum_{n=1}^\infty (1-x)^n/n$, whenever convergent.

If $F$ is a field of characteristic zero and $\xi \in \gl (N_0,F)$, then for
any regular function $f$ on $\GL (N_0)$ over $F$ we can form the formal power series
$f (\exp (t \xi)) \in F [[t]]$.
The following lemma is probably well known. For convenience we provide a proof.

\begin{lemma} \label{lem: formalps}
Let $X$ be an algebraic subgroup of $\GL(N_0)$ defined over a field $F$ of characteristic zero and $I_F(X)$ the ideal
of all regular functions on $\GL (N_0)$ over $F$ vanishing on $X$. Let $\xi \in \gl (N_0,F)$.
Then $\xi \in \Lie_F X$ if and only if the formal power series $f (\exp (t \xi)) \in F [[t]]$ vanishes for all $f \in I_F (X)$.

In particular, if $F=\Q_p$ and $\xi\in\gl(N_0,\Q_p)$ is such that the power series $\exp (t \xi)$ converges for all $t\in\Z_p$, then
$\xi \in \Lie_{\Q_p} X$ is equivalent to $\exp \xi\in X (\Q_p)$.
\end{lemma}

\begin{proof}
Using the notation of \cite[\S3.7]{MR1102012},
we have the differentiation formula
\[
\frac{d}{dt} g (\exp (t \xi)) = (g*\xi) (\exp (t\xi))
\]
for any regular function $g$ on $\GL (N_0)$ over $F$. To see this, write $g (x \cdot y) = \sum_i u_i (x) v_i (y)$
with regular functions $u_i$ and $v_i$ as in [loc.~cit.], which gives
the relation $g (\exp ((t_1+t_2)\xi)) = \sum_i u_i (\exp (t_1 \xi)) v_i (\exp (t_2 \xi))$ in $F[[t_1,t_2]]$.
Taking the derivative with respect to $t_2$ at $t_2 = 0$ yields
$\frac{d}{dt} g (\exp (t \xi))\rest_{t=t_1}=\sum_i u_i (\exp (t_1 \xi)) (\xi v_i) = (g*\xi) (\exp (t_1 \xi))$ (by formula (2) in [loc. cit.]), as claimed.

It follows that the coefficient of $t^i$ in the power series of $f (\exp (t\xi))$ is $\frac{1}{i!} f_i (1)$ where
$f_0=f$ and $f_i=f_{i-1}*\xi$, $i>0$.
By [ibid., Proposition 3.8 (ii)], for $f \in I_F (X)$ and $\xi \in \Lie_F X$ the functions $f_i$ vanish on $X$ for all $i$,
whence the `only if' part of the lemma.
For the `if' part, note that the vanishing of $f (\exp (t\xi))$ for all $f \in I_F (X)$ implies the vanishing of the linear terms
of these series, i.e., of $\xi f = (f*\xi) (1)$, for all such $f$. By [ibid., Proposition 3.8 (i)],
this gives $\xi \in \Lie_F X$.

The last assertion follows since a convergent power series on $\Z_p$ which vanishes on the rational integers is identically zero.
\end{proof}

If $A$ and $B$ are sets and $f:A\rightarrow B$ and $g:B\rightarrow A$ are functions, we write
\[
A\bjct{f}{g}B
\]
to mean that $f$ and $g$ are mutually inverse bijections, i.e., $g\circ f=1_A$ and $f\circ g=1_B$.

The following proposition summarizes the correspondence between subgroups of $\Gamma(N_0,\altp)$ and
Lie subalgebras of $\altp\gl(N_0,\Z_p)$. It is discussed in \cite{MR1720368}.

\begin{proposition} \label{prop: unif}
\begin{enumerate}
\item ([ibid., Proposition 6.22 and Corollary 6.25]) \label{explogconvergence}
The power series $\exp$ (resp. $\log$) converge
on the $\Z_p$-Lie algebra $\altp \gl(N_0, \Z_p)$ (resp., the uniform group $\Gamma (N_0, \altp)$)
and define mutually inverse bijections
\[
\altp \gl (N_0, \Z_p)\bjct{\exp}{\log}\Gamma (N_0, \altp).
\]
\item \label{part: expcongruence1}
The maps $\exp$ and $\log$ induce mutually inverse bijections
\[
\altp \gl (N_0, \Z_p)/p^n\gl (N_0, \Z_p)\bjct{\exp}{\log}\Gamma(N_0,\altp)/\Gamma(N_0,p^n),
\quad n\ge\epsilon_p.
\]
\item ([ibid., Corollary 7.14 and Theorem 9.10])
Applying the maps $\exp$ and $\log$ to subsets gives rise to bijections
\begin{gather*}
\{\text{powerful closed Lie subalgebras of }\altp \gl(N_0,\Z_p)\}\\
\bjct{\exp}{\log}\\
\{\text{uniform closed subgroups of }\Gamma(N_0,\altp)\}.
\end{gather*}
\item ([ibid., Theorem 4.17 and scholium to Theorem 9.10]) Under these bijections the rank of a uniform subgroup
$H$ of $\Gamma(N_0,\altp)$ is $\dim_{\Q_p}\log H\otimes\Q_p$ and
\begin{gather*}
\{\text{isolated closed Lie subalgebras of }\altp \gl(N_0,\Z_p)\}\\
\bjct{\exp}{\log}\\
\{\text{isolated closed subgroups of }\Gamma(N_0,\altp)\}.
\end{gather*}
\item \label{eqnalgebraicsubgroup} For any algebraic subgroup $X$ of $\GL(N_0)$ defined over $\Q_p$ we have
\[
\exp (\Lie_{\Q_p} X \cap p^n\gl(N_0,\Z_p)) = X (\Q_p) \cap \Gamma(N_0,p^n), \quad n \ge \epsilon_p.
\]
In particular, for all $n\ge\epsilon_p$ we have $\exp (p^n \lie{g}_{\Z_p}) = \K_p (p^n)$
and $\K_p(p^n)$ is a uniform group.
\end{enumerate}
\end{proposition}

\begin{remark}
In \S \ref{SectionApproxLevel0} below we will consider the exponential and logarithm maps on
larger domains, provided that $p$ is sufficiently large with respect to $N_0$.
\end{remark}

\begin{proof}
To prove part \ref{part: expcongruence1},
we need to show that
\[
\exp(\altp x+p^ny) \equiv \exp(\altp x)\pmod {p^n\gl(N_0,\Z_p)}
\]
for any $x,y\in\gl(N_0,\Z_p)$ and $n\ge\epsilon_p$. Expanding the power series as an infinite linear
combination of products of $x$ and $y$ (which in general do not commute), it is enough to show that
\[
-v_p(k!)+\epsilon_p(k-i)+ni\ge n
\]
for any $0<i\le k$. Equivalently,
\[
v_p(k!)\le n(i-1)+\epsilon_p(k-i),
\]
which holds since $v_p(k!)\le (k-1)/(p-1) = (i-1)/(p-1) + (k-i)/(p-1)$ and $1/(p-1) \le n$, $\epsilon_p$.
Similarly, the weaker inequality
\[
-v_p(k)+\epsilon_p(k-i)+ni\ge n
\]
shows that
\[
\log((1+\altp x)(1+p^ny))\equiv\log(1+ \altp x)\pmod {p^n\gl(N_0,\Z_p)}\\
\]
for any $x,y\in\gl(N_0,\Z_p)$ and $n\ge\epsilon_p$.
Part \ref{part: expcongruence1} follows.
(Alternatively, this part also follows from [ibid., Theorem 5.2 and Corollary 7.14] together with the first part of [ibid., Corollary 4.15].)

Given part \ref{explogconvergence},
part \ref{eqnalgebraicsubgroup}
follows from Lemma \ref{lem: formalps} applied to the elements of $p^n \gl (N_0, \Z_p)$.
\end{proof}

\begin{remark}
In fact, one can intrinsically endow any uniform pro-$p$ group with the structure of a Lie algebra over $\Z_p$
which is free of finite rank as a $\Z_p$-module [ibid., \S4.5] and from which we can recover the group structure.
We will not recall this construction here, since for our purposes it is advantageous to use the realization of the Lie algebra
inside a matrix space via $\rho_0$ (cf. [ibid., \S7.2]).
\end{remark}

We can immediately deduce analogous results for subgroups of $\K_p(\altp)$.

\begin{proposition} \label{prop: unifg}
\begin{enumerate}
\item \label{part: expcongruence}
The maps $\exp$ and $\log$ induce mutually inverse bijections
\[
\altp\lie{g}_{\Z_p}/p^n\lie{g}_{\Z_p}\bjct{\exp}{\log}\K_p(\altp)/\K_p(p^n),\quad n\ge\epsilon_p.
\]
\item
The application of $\exp$ and $\log$ to subsets gives rise to bijections
\[
\{\text{powerful closed Lie subalgebras of }\altp \lie{g}_{\Z_p}\}
\bjct{\exp}{\log}
\{\text{uniform closed subgroups of }\K_p (\altp)\}
\]
and
\[
\{\text{isolated closed Lie subalgebras of }\altp \lie{g}_{\Z_p}\}\bjct{\exp}{\log}
\{\text{isolated closed subgroups of }\K_p (\altp)\}.
\]
\item \label{eqnalgebraicsubgroup2} For any algebraic subgroup $X$ of $G$ defined over $\Q_p$ we have
\[
\exp (\Lie_{\Q_p} X \cap p^n\lie{g}_{\Z_p}) = X (\Q_p) \cap \K_p (p^n), \quad n \ge \epsilon_p.
\]
\end{enumerate}
\end{proposition}

\begin{remark} \label{rem: Ilani}
In general, a closed subgroup of $\K_p (\altp)$ is not necessarily uniform.
Nevertheless, by a result of Ilani \cite{MR1345293} (combined with \cite[Corollary 7.14]{MR1720368}) for $p \ge \dim G$
the application of $\exp$ and $\log$ to subsets gives rise to bijections
\[
\{\text{closed Lie subalgebras of }\altp \lie{g}_{\Z_p}\}
\bjct{\exp}{\log}
\{\text{closed subgroups of }\K_p (\altp)\}.
\]
If moreover $p>N_0+1$, then this bijection can be extended to all closed pro-$p$ subgroups of $\K_p$ (\cite{MR2126210, MR2554763}).
See \S \ref{sectionliealgebra}, where we also give a further generalization that is valid for all $p$ large enough with respect to $G$.

Note that Ilani's theorem implies (together with Proposition \ref{prop: unif}, part \ref{part: expcongruence1})
that for $p \ge \dim G$ we have $[\K_p (p) : H] \ge p^{n-1}$ for all open subgroups $H$ of $\K_p (p)$ of level $p^n$.
If $G$ is simply connected and $p$ is sufficiently large, then by \cite[Window 9, Lemma 5]{MR1978431}
\begin{multline} \label{eq: fratinni}
\text{for any proper open subgroup $H$ of $\K_p$ the image $\bar H$ of $H$ in $G (\F_p) \simeq \K_p / \K_p (p)$} \\ \text{is a proper subgroup of $G (\F_p)$.}
\end{multline}
Since $G (\F_p)$ is generated by elements of $p$-power order,
we have $[G (\F_p) : \bar H] \ge p$, and therefore if $H$ is of level $p^n$, $n \ge 1$, then $[\K_p : H] = [G (\F_p) : \bar H] [\K_p (p) : H \cap \K_p (p)] \ge p^n$, which is
a well-known result of Lubotzky \cite[Lemma 1.6]{MR1312501}.
By Lemma \ref{lem: basunif}, part \ref{LemmaUniform}, and Proposition \ref{prop: unifg}, for every $p$ we have at least $[\K_p : H] \ge c p^n$
for all open subgroups $H$ of $\K_p$ of level $p^n$, where $c>0$ depends only on $G$.
\end{remark}

\subsection{Proof of Theorem \ref{TheoremAlgebraic}} \label{subsliealgebra}
The essential step in the proof of Theorem \ref{TheoremAlgebraic} is to establish the following Lie algebra analog.
If $L_p$ is a free $\Z_p$-module of finite rank, we say that
an open submodule $M \subset L_p$ has level $p^n$ if $n$ is the minimal integer such that $p^n L_p \subset M$.

\begin{theorem} \label{thm: Liealgebra}
Let $L$ be a Lie algebra over $\Z$ which is free of finite rank as a $\Z$-module.
Then there exist an integer $D > 0$ and a constant $0 < \varepsilon \le 1$ with the following property.
For any prime $p$ and an open Lie subalgebra $M \subset L_p = L \otimes \Z_p$ of level $p^n$, where $n > v_p (D)$,
there exists a proper, isolated closed subalgebra $I$ of $L_p$ such that
$M \subset I + p^{\lceil \varepsilon n \rceil} L_p$.
\end{theorem}

Before proving Theorem \ref{thm: Liealgebra} we show how it implies Theorem \ref{TheoremAlgebraic}.
\begin{proof}[Proof of Theorem \ref{TheoremAlgebraic}]
As in Remark \ref{RemarkAlgebraic}, we are free to assume throughout that $n \ge n_0$, where $n_0$ depends only on $G$
(and will be determined below).

Let $H \subset \K_p (\altp)$ be an arbitrary subgroup of level $p^n$ and $V \subset H$ a uniform subgroup
as in part \ref{LemmaUniform} of Lemma \ref{lem: basunif}. Clearly, $V$ has level at least $p^n$.
Since $V$ is uniform, $\log V$ is by part 2 of
Proposition \ref{prop: unifg} a Lie subalgebra of $\altp \lie{g}_{\Z_p}$ of level at least $p^n$ (inside $\lie{g}_{\Z_p}$).
We now apply Theorem \ref{thm: Liealgebra} to $L = \lie{g}_{\Z}$ and $M = \log V$,
and obtain under the assumption $n>v_p(D)$
the existence of a proper, isolated closed subalgebra $I$ of $\lie{g}_{\Z_p}$ with
\[
\log V \subset I + p^{\lceil \varepsilon n \rceil} \lie{g}_{\Z_p}.
\]
Since $V\supset H^{\{p^N\}}$, where $N = N (\dim G)$ depends only on $G$, we get
\[
p^N\log H=\log H^{\{p^N\}}\subset I + p^{\lceil \varepsilon n \rceil} \lie{g}_{\Z_p},
\]
and hence, if $\lceil \varepsilon n \rceil \ge N+\epsilon_p$, the set $\log H$ is contained in
\[
(\Q_p I+p^{\lceil \varepsilon n \rceil-N} \lie{g}_{\Z_p})\cap \altp \lie{g}_{\Z_p}=
\altp I+p^{\lceil \varepsilon n \rceil-N} \lie{g}_{\Z_p},
\]
where the last equality holds since $I$ is isolated in $\lie{g}_{\Z_p}$.
It follows from part \ref{part: expcongruence} of Proposition \ref{prop: unifg} that
\begin{equation} \label{eq: hspep}
H=\exp\log H \subset \exp (\altp I) \K_p (p^{\lceil \varepsilon n \rceil-N}).
\end{equation}

Using the notation of \cite[\S 7.1]{MR1102012}, let $X = \mathcal{A} (\Q_p I)$ be the
smallest algebraic subgroup of $G$ defined over $\Q_p$ such that $\Lie_{\Q_p} X \supset \Q_p I$. Note
that $X$ is connected and that its Lie algebra $\Lie_{\Q_p} X = \mathfrak{a} (\Q_p I)$
is the smallest algebraic subalgebra of $\lie{g}_{\Q_p}$ containing the proper subalgebra $\Q_p I$.
Since $\lie{g}_{\Q_p}$ is semisimple,\footnote{This is the only place in the proof where we use the semisimplicity of $G$.}
it follows from [ibid., Corollary 7.9] that $\Lie_{\Q_p} X$ is still a \emph{proper} subalgebra of $\lie{g}_{\Q_p}$, and we may therefore
replace $I$ by $\Lie_{\Q_p} X \cap \lie{g}_{\Z_p}$.
Combining part \ref{eqnalgebraicsubgroup2} of Proposition \ref{prop: unifg} with \eqref{eq: hspep}, we obtain
\[
H\subset (X (\Q_p) \cap \K_p (\altp)) \K_p (p^{\lceil \varepsilon n \rceil-N}).
\]
Since we are free to replace $\varepsilon$ by a smaller positive constant, we conclude Theorem \ref{TheoremAlgebraic}.
\end{proof}

\begin{remark} \label{rem: sameps}
If $p\ge\dim G$ then using Ilani's theorem (see Remark \ref{rem: Ilani}) we can avoid using $V$ in the argument above and keep the same value
of $\varepsilon$ as in Theorem \ref{thm: Liealgebra}.
\end{remark}

It remains to prove Theorem \ref{thm: Liealgebra}. We first need an abstract lifting result for isolated subalgebras.

\begin{lemma} \label{LemmaLifting}
Let $L$ be a Lie algebra over $\Z$ which is free of rank $d$ as a $\Z$-module.
Then there exist an integer $D_1 > 0$ and a constant $0 < \varepsilon_1 \le 1$ with the following property.
Let $p$ be a prime and $r=1,\dots,d-1$.
Assume that $x_1, \ldots, x_r$ are elements of $L_p = L \otimes \Z_p$ which are linearly independent modulo $p$ and satisfy
\[
[x_i,x_j] \in \sum_{k=1}^r \Z_p x_k + p^\nu L_p, \quad 1 \le i, j \le r,
\]
with $\nu > v_p (D_1)$.
Then there exist $y_1, \ldots, y_r \in L_p$ such that $y_i \equiv x_i \pmod{p^{\lceil \varepsilon_1 \nu \rceil}}$
and $I = \Z_p y_1 + \dots + \Z_p y_r$ is an isolated closed subalgebra of $L_p$ of rank $r$.
\end{lemma}

\begin{remark}
An analogous result in the archimedean case was proved in \cite{MR2507639} using the Lojasiewicz inequality.
\end{remark}

The proof is an application of the following general lifting theorem of M.~Greenberg \cite{MR0351994}.

\begin{theorem}[Greenberg] \label{thm: Greenberg}
Let $f_1,\dots,f_m$ be polynomials in $N$ variables over $\Z$. Then there exist positive integers $C$ and $D$ with the following property.
Suppose that $x_1, \ldots, x_N \in\Z_p$ and $\nu > v_p (D)$ are such that $f_i (x_1, \dots, x_N) \equiv 0 \pmod{p^\nu}$ for all $i$.
Then there exist $y_1, \ldots, y_N \in\Z_p$ such that $y_i \equiv x_i \pmod{p^{\lceil \frac{\nu-v_p (D)}{C} \rceil}}$
and $f_i(y_1,\dots,y_N)=0$ for all $i$.
\end{theorem}

\begin{remark}
If the affine variety over $\Q$ defined by $f_1,\dots,f_m$ is smooth, then the proof in \cite{MR0351994} shows that we can in fact take $C=1$
(cf. \cite{MR0207700}, where the analogous result over $\Z_p$ is credited to N\'{e}ron).
The point of Theorem \ref{thm: Greenberg} is that it holds without any restriction on $f_1,\dots,f_m$.
\end{remark}

\begin{proof}[Proof of Lemma \ref{LemmaLifting}]
By fixing a $\Z$-basis for $L$, we can identify $L$ with $\Z^d$. The space of $k$-tuples $(x_1, \dots, x_k)$ of
vectors in $L \otimes \Q$ can then be identified with an affine space $X_k$ of dimension $kd$.
For any subset $S \subset \{1, \ldots, d\}$ of cardinality $k$ let
$D_S (y_1, \dots, y_k)$ be the corresponding $k\times k$-minor of the $k\times d$-matrix
formed by the vectors $y_1, \dots, y_k \in L$.

In the statement of the lemma we can clearly fix $1\le r<d$.
Let $k=r+1$ and consider the polynomials
$f_{ij,S} (x_1, \dots, x_r) := D_S (x_1, \dots, x_r, [x_i,x_j])$ on $X_r$ for all $1 \le i,j \le r$ and
subsets $S \subset \{1, \ldots, d\}$ of cardinality $k$. These polynomials have coefficients in $\Z$ and
they define a closed subvariety $V_r$ of $X_r$.

Whenever $x_1, \ldots, x_r \in L_p$ are linearly independent, and $(x_1, \dots, x_r) \in V_r (\Q_p)$, the
associated free $\Z_p$-submodule
$M = \Z_p x_1 \oplus \dots \oplus\Z_p x_r$ of $L_p$ (of rank $r$) is a closed Lie subalgebra. If furthermore
the reductions modulo $p$ of the vectors $x_1, \ldots, x_r$ are already linearly independent, then
$M=\Z_p x_1 \oplus \dots \oplus\Z_p x_r$ is an isolated closed Lie subalgebra of $L_p$.

Let now $x_1, \ldots, x_r \in L_p$ be as in the statement of Lemma \ref{LemmaLifting}.
Then $f_{ij,S} (x_1, \dots, x_r) \equiv 0 \pmod{p^\nu}$ for all $i$, $j$ and $S$.
Applying Theorem \ref{thm: Greenberg} to the polynomials $f_{ij,S}$, if
$\nu > v_p (D)$ there exists $(y_1, \ldots, y_r) \in V_r (\Q_p)$ with $y_i \equiv x_i \pmod{p^\mu}$, where
\[
\mu = \left\lceil \frac{\nu-v_p (D)}{C} \right\rceil \ge \left\lceil \frac{\nu}{C+v_p(D)} \right\rceil.
\]
In particular, the vectors $y_1, \dots, y_r$ are congruent modulo $p$ to the vectors $x_1,\dots,x_r$, which implies that they are
linearly independent modulo $p$.
Therefore $I = \Z_p y_1 \oplus \dots \oplus\Z_p y_r$ is an isolated closed Lie subalgebra of $L_p$.
We obtain the assertion with $\varepsilon_1 = (C + \max_p v_p (D))^{-1}$.
\end{proof}

\begin{proof}[Proof of Theorem \ref{thm: Liealgebra}]
Let $M$ be an open Lie subalgebra of $L_p$ of level $p^n$. In particular, $M$ is a $\Z_p$-submodule of finite index.
By the elementary divisor theorem, there exist a $\Z_p$-basis
$x_1, \ldots, x_d$ of $L_p$ and integers $0 \le \alpha_1 \le \dots \le \alpha_d$ such that
$(p^{\alpha_i} x_i)_{1 \le i \le d}$ is a basis of $M$.
The fact that $M$ is of level $p^n$ means that $\alpha_d=n$.

Fix an arbitrary constant $0<c<\frac12$. Let $0<r \le d-1$ be the maximal index such that $\alpha_r<c\alpha_{r+1}$, if such an index exists;
otherwise set $r=0$. Note that
\[
\alpha_{r+1} \ge c\alpha_{r+2}\ge\dots\ge c^{d-r-1}\alpha_d=c^{d-r-1} n.
\]
Write $[x_i,x_j] = \sum_k c_{ijk} x_k$ with $c_{ijk} \in \Z_p$. Since $M$ is a Lie subalgebra of $L_p$,
we have $p^{\alpha_i+\alpha_j} c_{ijk} \in p^{\alpha_k} \Z_p$. For $i,j \le r$ and $k > r$ we obtain
$c_{ijk} \in p^{\alpha_k-\alpha_i-\alpha_j} \Z_p \subset p^{\alpha_{r+1} - 2 \alpha_r} \Z_p$.
Here,
\[
\alpha_{r+1}-2\alpha_r > (1-2c) \alpha_{r+1} \ge (1-2c) c^{d-r-1} n.
\]
Summing up, we obtain the existence of $0 \le r \le d-1$ such that $[x_i,x_j] \in \sum_{k=1}^r \Z_p x_k + p^\nu L$
for $1 \le i, j \le r$, where
\begin{equation} \label{eqndefinitionnu}
\nu = \left\lceil \left( 1-2c\right) c^{d-r-1} n \right\rceil.
\end{equation}
We may apply Lemma \ref{LemmaLifting} to the elements $x_1, \ldots, x_r$
and obtain an isolated Lie subalgebra $I = \Z_p y_1 \oplus \dots \oplus\Z_p y_r$ of $L_p$ with $x_i \equiv y_i \pmod{p^{\lceil \varepsilon_1 \nu \rceil}}$
whenever $\nu > v_p (D_1)$. The last condition is evidently satisfied when $n > v_p (D)$, where $D = D_1^e$ with $e = \lceil (1-2c)^{-1} c^{1-d} \rceil$.
Under this condition we get
\[
M \subset \Z_p x_1 \oplus \dots \oplus \Z_p x_r + p^{\alpha_{r+1}} L \subset I + p^{\lceil \varepsilon n \rceil} L
\]
for $\varepsilon = \varepsilon_1 (1-2c) c^{d-1}$,
since $\alpha_{r+1}\ge c^{d-r-1} n\ge\varepsilon n$.
\end{proof}

We remark that Theorem \ref{thm: Liealgebra} admits a local counterpart:
\begin{theorem} \label{thm: Liealgebrap}
Let $p$ be a prime number and let $L$ be a Lie algebra over $\Z_p$ which is free of finite rank as a $\Z_p$-module.
Then there exist an integer $n_0\ge 0$ and a constant $0 < \varepsilon \le 1$ with the following property.
For any open Lie subalgebra $M \subset L$ of level $p^n$, where $n > n_0$,
there exists a proper, isolated closed subalgebra $I$ of $L$ such that
$M \subset I + p^{\lceil \varepsilon n \rceil} L$.
\end{theorem}
The proof is the same except that we use the natural variant of Theorem \ref{thm: Greenberg} where the
ring $\Z$ is replaced by $\Z_p$ (which is contained in \cite{MR0351994}, but also already in \cite[Theorem 1]{MR0207700}).
For our purposes, the additional uniformity in $p$ of $\varepsilon$ and $n_0$ in Theorem \ref{thm: Liealgebra} is important.

Similarly, we have the following analog of Theorem \ref{TheoremAlgebraic} for arbitrary uniform pro-$p$ groups.

\begin{theorem}
Let $L$ be a uniform pro-$p$ group.
Then there exist an integer $n_0\ge 0$ and a constant $0 < \varepsilon \le 1$ such that
for any open subgroup $H$ of $L$ with $n=\min\{m\ge0:L^{p^m}\subset H\}>n_0$ there exists a proper, isolated closed subgroup $I$ of $L$ with
$H \subset I L^{p^{\lceil \varepsilon n \rceil}}$.
\end{theorem}

\begin{proof}
Let $H$ be a subgroup of $L$ and $n$ as above.
We use the Lie algebra structure on $L$ (and in particular the additive structure) defined in \cite[\S4]{MR1720368}.
In particular, $x^m=mx$ for any $x\in L$ and $m\in\Z$ [ibid., Lemma 4.14].
Let $V \subset H$ be as in part \ref{LemmaUniform} of Lemma \ref{lem: basunif}.
Since $V$ is uniform, it is a Lie subalgebra of $L$ of level $\ge p^n$ [ibid., Theorem 9.10].
Applying Theorem \ref{thm: Liealgebrap}, there exists a proper, isolated closed subalgebra $I$ of $L$ such that
\[
V \subset I + p^{\lceil \varepsilon n \rceil} L.
\]
Since $V\supset H^{\{p^N\}}=p^NH$ we get
\[
p^NH\subset I + p^{\lceil \varepsilon n \rceil}L,
\]
and hence, if $\lceil\varepsilon n\rceil\ge N$,
\[
H\subset I+p^{\lceil \varepsilon n \rceil-N} L,
\]
since $I$ is isolated in $L$. For any integer $m\ge0$ the identity map induces a bijection between the
quotient group $L/L^{p^m}$ and the quotient $L/p^mL$ with respect to the additive structure on $L$
[ibid., Corollary 4.15]. We obtain
\[
H\subset IL^{p^{\lceil \varepsilon n \rceil-N}},
\]
from which the theorem follows. (Recall that by [ibid., Scholium to Theorem 9.10], $I$ is an isolated subgroup of $L$.)
\end{proof}

We end up the discussion of Theorem \ref{TheoremAlgebraic} with a couple of comments about
the function $\nu (n)$ and the number $\varepsilon_G$ of Remark \ref{RemarkAlgebraic} above.

\begin{lemma} \label{RemarkEpsilonG}
We have $\nu (n) \le \lceil \frac{n}2 \rceil$ and hence $\varepsilon_G \le \frac12$.
\end{lemma}

\begin{proof}
For $n \ge 2\epsilon_p-1$ the logarithm map induces an isomorphism
\[
\K_p (p^{\lceil \frac{n}2 \rceil}) / \K_p (p^n) \simeq p^{\lceil \frac{n}2 \rceil} \lie{g}_{\Z_p} / p^n \lie{g}_{\Z_p}
\simeq \lie{g}_{\Z_p} / p^{\lfloor \frac{n}2 \rfloor} \lie{g}_{\Z_p}
\]
of abelian groups.
(To see this, note that we have $\log x \equiv x - 1 - \frac{(x-1)^2}2 \pmod{p^n}$
and therefore $\log (xy) \equiv \log x + \log y \pmod{p^n}$ for $x,y \in \Gamma (N_0, p^{\lceil \frac{n}2 \rceil})$.
The bijectivity follows from part \ref{part: expcongruence} of Proposition \ref{prop: unifg}.)

Therefore, the map $H\mapsto p^{-\lceil \frac{n}2 \rceil}\log H$ gives rise to a bijection between
the open subgroups of $\K_p (p^{\lceil \frac{n}2 \rceil}) \subset \K_p (\altp)$ of level $p^n$ and the
essential subgroups of $\lie{g}_{\Z_p} / p^{\lfloor \frac{n}2 \rfloor} \lie{g}_{\Z_p}$,
namely those which do not contain $p^{\lfloor \frac{n}2 \rfloor - 1} \lie{g}_{\Z_p} / p^{\lfloor \frac{n}2 \rfloor} \lie{g}_{\Z_p}$.

Assume on the contrary that $\nu = \nu (n) > \lceil \frac{n}2 \rceil$.
Then it follows from this discussion and Proposition \ref{prop: unifg} that
every essential subgroup ${\bar H}$ of $\lie{g}_{\Z_p} / p^{\lfloor \frac{n}2 \rfloor} \lie{g}_{\Z_p}$ is contained in
$(\Lie_{\Q_p} X \cap \lie{g}_{\Z_p}) + p^{\nu - \lceil \frac{n}2 \rceil} \lie{g}_{\Z_p} / p^{\lfloor \frac{n}2 \rfloor} \lie{g}_{\Z_p}$
for some proper, connected algebraic subgroup $X$ of $G$.
Equivalently, the projection of ${\bar H}$ to $\lie{g}_{\Z_p} / p^{\nu - \lceil \frac{n}2 \rceil} \lie{g}_{\Z_p}$
is contained in the reduction modulo $p^{\nu - \lceil \frac{n}2 \rceil}$ of
$\Lie_{\Q_p} X \cap \lie{g}_{\Z_p}$, which is a proper Lie subalgebra of $\lie{g}_{\Z_p} / p^{\nu - \lceil \frac{n}2  \rceil} \lie{g}_{\Z_p}$.
Considering the subgroups $\bar H$ obtained as kernels of arbitrary surjective group homomorphisms
$\lie{g}_{\Z_p} / p^{\lfloor \frac{n}2  \rfloor} \lie{g}_{\Z_p} \to \Z / p^{\lfloor \frac{n}2  \rfloor} \Z$,
one sees that at least for almost all $p$ this is impossible.

We conclude that $\nu (n) \le \lceil \frac{n}2  \rceil$.
\end{proof}

In general, the proof of Theorem \ref{thm: Liealgebra} does not give a good value for $\varepsilon_G$, since we used Greenberg's theorem to lift a Lie algebra modulo $p^\nu$ to characteristic zero.
We note however that for $L = \mathfrak{sl}_2 (\Z)$ and $p$ odd one gets the optimal value (and that even for $p=2$ we get an almost optimal result).
\begin{lemma} \label{lem: SL2}
Consider the Lie algebra $L = \mathfrak{sl}_2 (\Z)$.
Then for $p$ odd, Theorem \ref{thm: Liealgebra} is true for all $n > 0$ with $\varepsilon = \frac12$, which is optimal.
Similarly, Theorem \ref{TheoremAlgebraic} holds for $G = \SL (2)$ over $\Z$ and all odd primes $p$ with $\varepsilon = \frac12$.

For $p = 2$, Theorem \ref{thm: Liealgebra} holds for all $n \ge 3$ with $\lceil \varepsilon n \rceil$ replaced by
$\lceil \frac{n}2 \rceil - 1$. We have $\nu (n) \ge \lceil \frac{n}2 \rceil - 10$ for $G = \SL (2)$ over $\Z$,
and
in particular $\varepsilon_G = \frac12$.\footnote{Using \cite[Proposition 2.9, Theorem 2.10]{MR749678} instead of Lemma \ref{lem: basunif}, one may improve this bound to $\nu (n) \ge \lceil \frac{n-3}{2} \rceil$.
We omit the details.}
\end{lemma}

\begin{proof}
Consider an open Lie subalgebra $M = p^{\alpha_1} \Z_p x_1 \oplus p^{\alpha_2} \Z_p x_2 \oplus p^n \Z_p x_3$ of $L_p$, where
$x_1$, $x_2$, $x_3$ form a $\Z_p$-basis of $L_p$ and $0 \le \alpha_1 \le \alpha_2 \le n$.
Assume first that $p$ is odd. If $\alpha_2 \ge \frac{n}2$ we set $I = \Z_p x_1$ and obtain $M \subset I + p^{\alpha_2} L_p$.
Otherwise, the rank two isolated submodule $J = \Z_p x_1 \oplus \Z_p x_2$ of $L_p$ maps to a Lie subalgebra of $L / p^{n-\alpha_1-\alpha_2} L$, where
$n-\alpha_1-\alpha_2>0$. We claim that we can lift $J$ to a rank two isolated subalgebra $I \subset L_p$ with
$J + p^{n-\alpha_1-\alpha_2} L_p = I + p^{n-\alpha_1-\alpha_2} L_p$.
Granted this claim, we get $M \subset I + p^{n-\alpha_2} L_p$.
Thus, in both cases we have $M \subset I + p^{\lceil \frac{n}2 \rceil} L_p$.
The statement for $\SL(2)$ follows then from Remark \ref{rem: sameps}.

For $p = 2$ we modify the argument slightly and distinguish the cases $\alpha_2 \ge \frac{n}2 - 1$ and $\alpha_2 < \frac{n}2 - 1$. In the second case
$J$ maps to a Lie subalgebra of $L / 2^m L$ with $m = n - \alpha_1 - \alpha_2 \ge n - 2 \alpha_2 > 2$, and we claim that there exists a
rank two isolated subalgebra $I \subset L_2$ with
$J + 2^{m-2} L_2 = I + 2^{m-2} L_2$. We obtain $M \subset I + 2^{\lceil \frac{n}2 \rceil - 1} L_2$. The statement for $\SL (2)$
follows from the proof of Theorem \ref{TheoremAlgebraic}, noting that we can take $N = N (3) = 9$
in part \ref{LemmaUniform} of Lemma \ref{lem: basunif}.

It remains to establish our claim on the lifting of subalgebras (for arbitrary $p$).
We may parametrize rank two isolated submodules $J$ of $L_p$ by elements $c = \sm{c_1}{c_2}{c_3}{-c_1} \in \mathfrak{sl}_2 (\Q_p)$ with
$\min  ( v_p (2c_1), v_p (c_2), v_p (c_3) ) = 0$, considered up to multiplication by scalars in $\Z_p^\times$, by setting
$J = J (c) = \{ x \in L_p \, : \, \tr (cx) = 0 \}$. It is easy to check that the submodule $J (c)$ is a Lie subalgebra of $L_p$ if and only if $2 \tr (c^2) = (2c_1)^2 + 4c_2 c_3 = 0$,
and that for $m \ge 1$ the module $J (c)$ maps to a Lie subalgebra of $L_p /p^m L_p$ if and only if $2 \tr (c^2) = (2c_1)^2 + 4c_2 c_3 \equiv 0 \pmod{p^m}$.
For $p > 2$ the quadric defined by this equation is smooth over $\Z_p$ and points modulo $p^m$ lift to $\Z_p$-points.
For $p = 2$ it is easy to see that for $m \ge 3$, points modulo $2^m$ are congruent modulo $2^{m-2}$ to $\Z_2$-points.
Therefore, whenever $J(c)$ maps to a Lie subalgebra modulo $p^m$ there exists
a rank two isolated Lie subalgebra $I \subset L_p$ with $J (c) + p^m L_p = I + p^m L_p$ for $p$ odd (resp., $J (c) + 2^{m-2} L_2 = I + 2^{m-2} L_2$ for $p=2$).
\end{proof}

As one sees from the proof, it is not difficult to obtain an explicit parametrization of open Lie subalgebras of $\mathfrak{sl}_2 (\Z_p)$ and of
open subgroups of the principal congruence subgroups of $\SL (2, \Z_p)$, at least for odd $p$. For more details and the explicit computation of the corresponding counting functions
see \cite{MR1679595,MR1740889,MR1866324,MR2529493}.

\subsection{A finiteness property for maximal connected algebraic subgroups} \label{subsectionfiniteness}
For a linear algebraic group $G$ defined over a field $F$ and a field extension $E \supset F$,
we denote by $\sgrmx_E(G)$ the set of maximal (proper) connected algebraic subgroups of $G_E$ defined over $E$.
Note that in the statement of Theorem \ref{TheoremAlgebraic} we can clearly take $X$ to be in $\sgrmx_{\Q_p}(G)$.
As a supplement to this theorem, we now consider a certain simple finiteness property of the family or algebraic groups $\sgrmx_{\Q_p}(G)$, where $p$ ranges over all prime numbers.
This property is crucial for the application of our results in \S \ref{MainApplication} below.

We first quote a basic finiteness property of affine varieties \cite[\S 65]{MR0349648}.
Denote by $\aff_R^n=\Spec R[x_1,\dots,x_n]$ the $n$-dimensional affine space over a ring $R$.

\begin{lemma}[Seidenberg] \label{LemmaGeometricFiniteness}
Let $n$ and $d$ be positive integers. Then there exists a constant $C=C(n,d)$ with the following property.
Let $F$ be an algebraically closed field and let $V$ be a closed subvariety of $\aff_F^n$
whose defining ideal $I(V)$ is generated in degrees $\le d$. Then
the number of irreducible components of the variety $V$ is at most $C$ and the defining ideal of any
irreducible component $W$ is generated in degrees $\le C$.
Moreover, whenever $I$ is an arbitrary ideal of $F[X_1,\ldots,X_n]$ generated in degrees $\le d$, the radical $\sqrt{I}$, which is the
defining ideal of the zero set of $I$, is generated in degrees $\le C$.
\end{lemma}

\begin{remark} \label{FinitenessSubfield}
Let $F$ be an arbitrary field and $V$ be a closed subvariety of $\aff_{\bar F}^n$ defined over $F$
with defining ideal $I (V) \subset {\bar F} [X_1, \ldots, X_n]$. Then the following
are equivalent:
\begin{itemize}
\item The ideal $I_F (V) = I (V) \cap F[X_1,\ldots,X_n] \subset F[X_1,\ldots,X_n]$ is generated in degrees $\le d$.
\item $I (V)$ is generated in degrees $\le d$.
\end{itemize}
\end{remark}

This is simply because $I(V)$ is generated as an $\bar F$-vector space by $I_F (V)$.

\begin{definition}
Let $Y = \Spec A$ be an affine scheme over a ring $R$.
An \emph{admissible family} $\mathcal{X}$ is a family of pairs $(E,X)$, where $E$ is a field together with a ring homomorphism $R \to E$
and $X \subset Y_E = Y \times_{\Spec R} \Spec E$ is a closed subscheme.
We say that $\mathcal{X}$ is \BG\ (boundedly generated)
if there exists a finitely generated $R$-submodule $M \subset A$ such that
for all $(E,X) \in \mathcal{X}$
the defining ideal $I_E (X)$ of the subscheme $X$ is generated by elements of $M \otimes_R E$.
\end{definition}

\begin{remark}
If $Y$ is of finite type over $R$, then we can reformulate this definition as follows.
Fix a closed embedding of $Y$ into $\aff_R^n$ (equivalently, a choice of generators $a_1, \ldots, a_n$ of $A$ as an $R$-algebra).
We then have corresponding closed embeddings of $Y_E$ into $\aff_E^n$ for any $E$, and can define the notion of the degree of a regular function on $Y_E$.
An admissible family $\mathcal{X}$ is \BG\ if and only if there exists an integer $N$ such that for all $(E, X) \in \mathcal{X}$ the defining ideal $I_E (X)$ of the subscheme $X$
is generated by regular functions defined over $E$ of degree at most $N$. Note that this property is actually independent of the choice of embedding.
\end{remark}

We will mostly use this definition in the situation where we are given a linear algebraic group $G$ defined over a field $F$
and a family of algebraic subgroups defined over field extensions $E \supset F$.
We will also consider the case of a flat affine group scheme over $\Z$ and of the fields $E = \F_p$ as in \S \ref{subspropsetup}.

\begin{lemma} \label{LemmaFiniteness}
Let $G$ be an affine group scheme of finite type over a ring $R$
and let $\mathcal{X}$ be an admissible family such that for all
$(E,X) \in \mathcal{X}$ the base extension $G_E$ is geometrically reduced (i.e., a linear algebraic group defined over $E$) and
$X$ is an algebraic subgroup of $G_E$ defined over $E$. Fix a closed embedding of $G$ into $\aff_R^n$ for some $n$.

Then $\mathcal{X}$ is \BG\ if and only if
there exist an integer $N$ and for every $(E,X) \in \mathcal{X}$ an $E$-rational representation
$(\rho, V)$ of $G_E$, where $V$ is an $E$-vector space, and a line $D\subset V$, such that the matrix coefficients $v^\vee (\rho (g) v)$, $v \in V$, $v^\vee \in V^\vee$, have
degree $\le N$ and $X=\{g\in G_E:\rho(g)D=D\}$.
Moreover, if this is the case, then we may require in addition that $\rho$ is injective,
that $\dim V \le N$
and that the associated orbit map $G \to \mathbb{P} (V)$, $g \mapsto \rho (g) D$, is separable,
i.e., that $\Lie_E X=\{ v \in \lie{g}_E \, : \, d\rho (v) D \subset D \}$ \cite[Proposition 6.7]{MR1102012}.
\end{lemma}

\begin{proof}
The `only if' part (with the additional injectivity and separability conditions and the boundedness of $\dim V$) follows from the proof of \cite[Theorem 5.1]{MR1102012}.
The main point is the existence of an integer $N'$ such that for all $(E,X) \in \mathcal{X}$ there exists a right $G_E$-invariant $E$-subspace $U$
of the space of regular functions on $G_E$ of degree $\le N'$, which generates the ideal of $X$.
This claim follows from the proof of [ibid., Proposition 1.9].
The representation $\rho$ is then given as the direct sum of a suitable exterior power of $U$ and some injective representation $\rho_0$.
If we obtain the representation $\rho_0$ by the construction of [ibid., Proposition 1.10]
from the coordinate functions provided by the fixed closed embedding of $G$ into $\aff_R^n$, then the degrees of the matrix coefficients of $\rho_0$ and
its dimension are bounded in terms of $G$.
Therefore the degrees of the matrix coefficients of $\rho$ and its dimension are bounded in terms of $G$ and $N'$.

To see the `if' part, note that the defining ideal $I(X)$ of $X$ over $\bar E$
is the radical of the ideal generated by the matrix coefficients $v^\vee (\rho (g) v)$, $v \in D$, $v^\vee \in D^\perp$, and hence
by Lemma \ref{LemmaGeometricFiniteness} $I(X)$ is generated in degrees bounded in terms of $G$ and $N$ only. On the other hand,
$I(X)$ is by assumption generated by regular functions defined over $E$. In view of Remark \ref{FinitenessSubfield}, this shows that $I_E(X)$ is generated
in degrees bounded in terms of $G$ and $N$.
\end{proof}

For the following lemma, recall that the connected component of the identity $X^\circ$ of an algebraic group $X$ defined over $E$
is again defined over $E$ \cite[Proposition 1.2]{MR1102012}.

\begin{lemma} \label{LemmaComponent}
Let $G$ and $\mathcal{X}$ be as in Lemma \ref{LemmaFiniteness}.
Suppose that the family $\mathcal{X}$ is \BG. Then the family of its identity components $\{ (E,X^\circ):(E,X)\in\mathcal{X} \}$
is also \BG\ and the sizes of the component groups $X/X^\circ$, $(E,X)\in\mathcal{X}$, are bounded.
\end{lemma}

\begin{proof}
This follows immediately from Lemma \ref{LemmaGeometricFiniteness}.
\end{proof}

\begin{lemma} \label{LemmaNormalizerFiniteness}
Let $G$ be an affine group scheme of finite type over a ring $R$ and $\mathcal{E}$ a collection of perfect fields $E$ together with homomorphisms $R \to E$
for which the schemes $G_E$ are geometrically reduced.
Then the family of all pairs $(E,N_{G_E} (U))$, where $E \in \mathcal{E}$ and $U$ ranges over the subspaces of
the $E$-vector space $\lie{g}_E = \Lie_R G \otimes E$, is \BG.
The family of all pairs $(E,N_{G_E} (U)^\circ)$ is also \BG.
\end{lemma}

Note here that $N_{G_E} (U)$ is defined over $E$ because $E$ is perfect.

\begin{proof}
Let $d = \dim U$ and consider the $d$-th exterior power $V_d = \wedge^d \Ad$ of the adjoint representation of $G_E$.
Then the group $N_{G_E} (U)$ is the stabilizer of the line $\wedge^d U \subset V_d \otimes E$ (cf. \cite[\S 5.1, Lemma]{MR1102012}).
The degrees of the matrix coefficients of the representations $V_d$ are clearly bounded in terms of $G$ only.
Hence we can use Lemma \ref{LemmaFiniteness}.
The last assertion follows immediately from Lemma \ref{LemmaComponent}.
\end{proof}

\begin{corollary} \label{CorollaryMaximalFiniteness}
Let $G$ be a semisimple algebraic group defined over a field $F$ of characteristic zero
and $\mathcal{E}$ a collection of field extensions of $F$.
Then the family
\[
\maxfam_{\mathcal{E}}=\{(E,X):E \in \mathcal{E}, X\in\sgrmx_E (G)\}
\]
is \BG.
\end{corollary}

By Remark \ref{RemarkPositiveChar} below, the analogous statement for $G$ over $R = \Z$ (as in \S \ref{subspropsetup})
and the family $\{(\F_p,X) : X \in \sgrmx_{\F_p} (G) \}$ is not true in general.

\begin{proof}
Let $E \in \mathcal{E}$ and $X \in \sgrmx_E (G)$.
Consider the normalizer $N = N_G (\mathfrak{x})$ of the Lie algebra $\mathfrak{x}$ of $X$.
Since we are in characteristic zero, we have $N=N_G(X) \supset X$. If $N = G$, then $X$ is a normal subgroup of $G$,
which cannot be maximal, since $G$ is semisimple.
Therefore necessarily $N^\circ = X$. It remains to apply Lemma \ref{LemmaNormalizerFiniteness}.
\end{proof}

In fact, in the setting of Corollary \ref{CorollaryMaximalFiniteness} one can show a stronger finiteness property for the family $\maxfam_{\mathcal{E}}$.
(This will not be used in the sequel.)
Let us say that a family $\mathcal{X}$ satisfies \FC\ if
there exists a finite set $S$ of subgroups of $G_{\bar F}$ such that for
any $(E,X)\in\mathcal{X}$ the group $X_{\bar E}$ is $G({\bar E})$-conjugate to $Y_{\bar E}$ for some $Y \in S$.
(By Remark \ref{FinitenessSubfield}, this clearly implies that $\mathcal{X}$ is \BG.)
We claim that the following families satisfy \FC:
\begin{enumerate}
\item $\{(E,X):E\in\mathcal{E}, X\text{ contains a maximal torus of $G$ defined over $E$}\}$,
\item $\{(E,X):E\in\mathcal{E}, X\text{ is a semisimple connected subgroup of $G$ defined over $E$}\}$,
\item $\maxfam_{\mathcal{E}}$.
\end{enumerate}
The first case follows from \cite[Corollary 11.3 and Proposition 13.20]{MR1102012} (for any field $E$).
The second case follows from Richardson's rigidity theorem \cite{MR0206170},
or alternatively from Dynkin's description of the connected semisimple subgroups over an algebraically closed field, which is based on the classification of semisimple groups
\cite{MR0047629, MR0049903}.
The third case follows from the first two together with the fact (valid for any field $E$) that
for $X\in\sgrmx_E(G)$ there are three possibilities:
\begin{enumerate}
\item $X$ is not reductive,
\item $X$ is the centralizer of a torus defined over $E$,
\item $X$ is semisimple.
\end{enumerate}
(To see this, note that for $X$ reductive but not semisimple, $Z(X)^\circ$ is a non-trivial torus defined over $E$ \cite[Proposition 11.21, Theorem 18.2 (ii)]{MR1102012},
and that $X \subset C_G(Z(X)^\circ)$, which is
defined over $E$ [ibid., Corollary 9.2] and connected [ibid., Corollary 11.12]. Therefore $X = C(Z(X)^\circ)$ and we are in the second case.)
In the first case $X$ is a maximal parabolic subgroup of $G_E$ defined over $E$
(See \cite{MR0088488}; in fact, by \cite[Corollaire 3.3]{MR0294349} this is true for any perfect field $E$.)

\begin{remark}
Corollary \ref{CorollaryMaximalFiniteness} does not hold in general if $G$ is not semisimple.
For example, if $G$ a torus of dimension $\ge 2$ which is split over $E$
then $\sgrmx_E (G)$ is the set of all codimension one subtori, and it is easy to see that this set is not \BG.
\end{remark}

\begin{remark} \label{RemarkPositiveChar}
Corollary \ref{CorollaryMaximalFiniteness} does not hold in the positive characteristic case.
Indeed, even if $F$ is algebraically closed the set $\sgrmx_F (G)$ may not be \BG\ (see \cite[Lemma 2.4]{MR2043006}
for a description of $\sgrmx_F (G)$ in this case).
On the other hand, if $G$ is a \emph{simple} algebraic group over (an algebraically closed field) $F$ then
the finiteness of the number of $G(F)$-conjugacy classes in $\sgrmx_F(G)$
(or in fact, the finiteness of the number of $G(F)$-conjugacy classes of maximal subgroups of $G$ of positive dimension)
has been established in \cite{MR2044850}. However, even in that case the set of all
connected semisimple subgroups is usually not \BG\ (cf. \cite[esp. Corollary 4.5]{MR2043006}).
\end{remark}

\begin{remark}
In the situation where $F = \Q$ and $\mathcal{E}$ is the family of the $p$-adic fields $E = \Q_p$, finiteness
of Galois cohomology \cite[Theorem 6.14]{MR1278263} implies that each of the sets
$\sgrmx_{\Q_p}(G)$ consists of finitely many classes under conjugation by $G(\Q_p)$ (instead of $G(\overline{\Q_p})$).
\end{remark}

\section{The approximation theorem for open subgroups of $G (\Z_p)$} \label{SectionApproxLevel0}

\subsection{Statement of the result}
Let $G$ and $\rho_0$ be as in \S \ref{subspropsetup} above.
We will now prove an extension of Theorem \ref{TheoremAlgebraic}
to arbitrary open subgroups of $\K_p$ in the case where $G$ is simply connected.

\begin{theorem} \label{TheoremAlgebraicLevel0}
Suppose that $G$ is simply connected.
Then there exist constants $J \ge 1$ and $\varepsilon > 0$, depending only on $G$ and $\rho_0$, with the property that
for every open subgroup $H \subset \K_p$ of level $p^n$ there exists
a proper, connected algebraic subgroup $X \subset G$
defined over $\Q_p$ such that $[H:H\cap (X (\Q_p) \cap \K_p) \K_p (p^{\lceil \varepsilon n \rceil})]\le J$.
\end{theorem}

\begin{remark}
\begin{enumerate}
\item Note that Theorem \ref{TheoremAlgebraicLevel0} implies the existence of an open \emph{normal} subgroup $H'$ of $H$ of index
$\le J!$ which is contained in the group $(X (\Q_p) \cap \K_p) \K_p (p^{\lceil \varepsilon n \rceil})$.

\item It is clear that just like Theorem \ref{TheoremAlgebraic}, we can rephrase Theorem \ref{TheoremAlgebraicLevel0} in terms of essential
subgroups of the finite groups $G (\Z / p^n \Z)$. Note here that when $G_{\Z_p}$ is smooth over $\Z_p$ (which is the case for almost all $p$), the reduction maps
$\pi_{\infty,n}: \K_p \to G (\Z / p^n \Z)$ are surjective for all $n$.

\item
As will be explained in \S \ref{subsprooflevel0} below, the case $n = 1$ (or $n$ bounded) follows rather directly from results of Nori on subgroups of $G (\F_p)$ \cite{MR880952}.
(A variant of these results has been obtained independently by Gabber \cite[Ch. 12]{MR955052}.)
\end{enumerate}
\end{remark}

\subsection{Consequences of Nori's Theorem}
We first prove a corollary of an algebraization theorem of Nori for subgroups of $G (\F_p)$ \cite{MR880952}.
While this corollary is probably known (cf.~\cite[Proposition 8.1]{MR1329903} for a less explicit result), we
include the deduction from Nori's result (and Jordan's classic theorem on finite subgroups of $\GL (N, F)$ of order prime to the
characteristic of $F$), since we will use the essential part of the argument also in \S \ref{subsprooflevel0} below.
One may in fact consider the following as an introduction to the proof of Theorem \ref{TheoremAlgebraicLevel0}.
Also, for our application in \S \ref{MainApplication} below, Proposition \ref{PropositionApproxLevelp} is in fact sufficient.
Note that a more general algebraization theorem for finite subgroups of $\GL (N, F)$, $F$ an arbitrary field, has been obtained by
Larsen--Pink \cite{MR2813339}. Of course, one can also deduce such theorems (with optimal constants) from the classification of
finite simple groups \cite{MR2334748,MR2456628}.
However, the proofs of Nori and Larsen--Pink are independent of the classification.

We write $\nmbr(G)$ for a constant (which may change from one occurrence to the other) depending only on $G$ and $\rho_0$.

Note first that for $p>\nmbr(G)$
the group scheme $G_{\Z_p}$ is smooth over $\Z_p$ and $G_{\F_p}$ is semisimple and simply connected.
In particular, for those $p$ the reduction map $\K_p \to G (\F_p)$ is surjective.

\begin{proposition} \label{PropositionApproxLevelp}
Assume that $G$ is simply connected.
Then there exists a \BG\ family $\mathcal{X}$
such that for any prime $p>\nmbr(G)$ and any proper subgroup $H$ of $G (\F_p)$ there exists $(\F_p,X)\in\mathcal{X}$
such that $X$ is a proper, connected algebraic subgroup of $G_{\F_p}$ defined over $\F_p$ and $[H:H\cap X(\F_p)]\le\nmbr(G)$.
\end{proposition}

\begin{remark}
As above, we obtain from Proposition \ref{PropositionApproxLevelp} the existence of a \emph{normal} subgroup $H'$ of $H$ of index
$\le\nmbr(G)$ that is contained in the group $X (\F_p)$.
\end{remark}

To prove this proposition, we need to recall the results of Nori and Jordan.
We will use the following variant of Jordan's theorem to deal with subgroups of order prime to $p$.

\begin{proposition} \label{PropositionJordan}
For any integer $n$ there exists a constant $J=J(n)$ with the following property.
Let $G$ be a reductive group defined over a field $F$ with a faithful representation of degree $n$.
Then for all finite subgroups $H$ of $G(F)$ with $\chr F\not|\abs{H}$ there exists a maximal torus $T$ of $G$
defined over $F$ such that $[H:H\cap T(F)]\le J$.
\end{proposition}

\begin{proof}
Jordan's theorem implies the existence of an abelian subgroup of $H$ of index at most $ J' (n)$.
On the other hand, by \cite{MR0054612, MR0068531, MR0268192} we can embed any supersolvable subgroup of $G(F)$ into $N_G (T) (F)$ for some maximal torus $T$ of $G$ defined
over $F$. It remains to note that the index $[N_G(T):T]$ is bounded in terms of the $\bar{F}$-rank of $G$, i.e., in terms of $n$.
\end{proof}

We now state Nori's main result, which describes (for $p>\nmbr(G)$) the subgroups of $G (\F_p)$
generated by their elements of order $p$ in terms of certain connected algebraic subgroups of $G$ defined over $\F_p$,
or equivalently by certain Lie subalgebras of $\lie{g}_{\F_p}$.
It is based on the following construction. Recall the truncated logarithm and exponential functions
\[
\log^{(p)} x = - \sum_{i=1}^{p-1} \frac{(1-x)^i}{i}, \quad \exp^{(p)} y = \sum_{i=0}^{p-1} \frac{y^i}{i!},
\]
which are defined over any ring in which the primes $< p$ are invertible.
Fix $n$ and let $F$ be a field of characteristic $p\ge n$.
Denote by $\gl(n,F)_{\nilp}$ (resp., $\GL(n,F)_{\unip}$) the set of nilpotent (resp., unipotent) elements of $\gl (n, F)$
(resp., $\GL (n,F)$).
Note that $\gl(n,F)_{\nilp}=\{x\in\gl(n,F):x^p=0\}$ and $\GL(n,F)_{\unip}=\{x\in\GL(n,F):x^p=1\}$ since $p\ge n$.
Then
\[
\gl(n,F)_{\nilp}\bjct{\exp^{(p)}}{\log^{(p)}}\GL(n,F)_{\unip}.
\]

We say that an algebraic subgroup of $\GL (n)$ over $F$ is \emph{exponentially generated} (\EXPG), if it is generated by
one-dimensional unipotent subgroups of the form $t \mapsto \exp^{(p)} (t y)$, where $y \in \gl (n, F)_{\nilp}$.
By \cite[Proposition 2.2]{MR1102012} such a group is automatically connected.
For any subset $S \subset \gl (n, F)$ write $\Exp S$
for the group generated by $t \mapsto \exp^{(p)} (t y)$, $y\in S_{\nilp}=S\cap \gl (n, F)_{\nilp}$.
We call a Lie subalgebra of $\gl (n,F)$ \emph{nilpotently generated} (\NILG), if it is generated as a vector space by nilpotent elements.

We now quote the main part of \cite[Theorem A]{MR880952}.

\begin{theorem}[Nori] \label{TheoremNoriA}
Assume that $p>\nmbr(G)$. Then
\[
\{\text{\NILG\ Lie subalgebras of }\lie{g}_{\F_p}\}
\bjct{L\mapsto\Exp L}{\Lie_{\F_p}X\mapsfrom X}
\{\text{\EXPG\ subgroups of }G_{\F_p}\}.
\]
Moreover, if $X$ is an \EXPG\ subgroup of $G$ over $\F_p$ then
$\Lie_{\F_p} X$ is spanned by the elements $\log^{(p)} x$, where $x \in X (\F_p)_{\unip}=X(\F_p)\cap \GL(n,\F_p)_{\unip}$.
\end{theorem}

Following Nori, for any subgroup $H$ of $G (\F_p)$ we define the algebraic subgroup
$\algnori{H} = \Exp \, \{ \log^{(p)} x \, : \, x \in H_{\unip} \}$ of $G_{\F_p}$.
Also, we denote by $H^+$ the subgroup of $H$ generated by $H_{\unip}$, i.e., by the $p$-Sylow subgroups of $H$.
Thus, $H^+$ is a characteristic subgroup of $H$ of index prime to $p$.
The following result is \cite[Theorem B]{MR880952}.

\begin{theorem}[Nori] \label{TheoremNori}
For $p>\nmbr(G)$ and any subgroup $H$ of $G (\F_p)$ we have $\algnori{H}(\F_p)^+=H^+$
and $\Lie_{\F_p}\algnori{H}$ is spanned by $\log^{(p)}x$, $x\in H_{\unip}$.
\end{theorem}

\begin{proof}[Proof of Proposition \ref{PropositionApproxLevelp}]
We first remark that using Lemma \ref{LemmaComponent} we may drop the connectedness statement in the conclusion of the proposition, since we may
pass to the family of identity components.
We will take $\mathcal{X}$ to be the family
\[
\{(\F_p,N_G(I)):\text{$I$ is a subspace of $\Lie_{\F_p}G$ with $N_G (I) \neq G$}\}.
\]
By Lemma \ref{LemmaNormalizerFiniteness}, $\mathcal{X}$ is \BG.

Let $H$ be a proper subgroup of $G (\F_p)$ and let $\algnori{H}$ be as before.
Note that $\algnori{H}$ is a proper (connected) algebraic subgroup of $G_{\F_p}$ since otherwise
we would have $H^+ = \algnori{H}(\F_p)^+ = G(\F_p)^+ = G (\F_p)$, because $G_{\F_p}$ is simply connected.
Clearly, $H$ is contained in the group of $\F_p$-points of $N_G (\algnori{H})$, which coincides
with $N_G (\Lie\algnori{H})$ by Theorem \ref{TheoremNoriA}.
If $\algnori{H}$ is not normal in $G$ then $N_G(\Lie\algnori{H}) = N_G (\algnori{H}) \in\mathcal{X}$ and we are done.
Otherwise, as in the proof of \cite[Theorem C]{MR880952},
using the Frattini argument we can write $H = H_1 H^+$ where $H_1$ is a subgroup of $H$ of order prime to $p$.
(Namely, we choose a $p$-Sylow subgroup $P$ of $H$ and take $H_1$ to be a complement of $P$ in $N_H(P)$.)
By Proposition \ref{PropositionJordan}, there exist a subgroup $A$ of $H_1$ of bounded index
and a maximal torus $T$ of $G$ defined over $\F_p$ such that $A \subset T (\F_p)$.
Then we take $X =N_G(\Lie(T\algnori{H}))$. For $p > \nmbr (G)$ the $G$-invariant subspaces of $\lie{g}_{\F_p}$ are precisely the
Lie algebras of the connected normal subgroups of $G_{\F_p}$.
Therefore, $\Lie (T \algnori{H}) \subset \lie{g}_{\F_p}$ is not $G$-invariant, and $X \in\mathcal{X}$.
\end{proof}

\subsection{Proof of Theorem \ref{TheoremAlgebraicLevel0}} \label{subsprooflevel0}
We now prove Theorem \ref{TheoremAlgebraicLevel0}.
We first need to modify our previous intermediate results
Theorem \ref{thm: Liealgebra} and Lemma \ref{LemmaLifting}
by incorporating the action of a finite abelian subgroup of $\K_p$ of
order prime to $p$. The following is the analog of Lemma \ref{LemmaLifting} in this setting.

\begin{lemma} \label{LemmaLiftingWithai}
Let $G$ be as above and $s$ a positive integer.
Then there exist an integer $D_1 > 0$ and a constant $0 < \varepsilon_1 \le 1$, depending only on $G$, $\rho_0$, and $s$, with the following property.
Suppose that we are given $x_1, \ldots, x_r \in \lie{g}_{\Z_p}$, $1\le r<d$, $a_1,\dots,a_s\in\K_p$ and $\nu > v_p (D_1)$ such that
\begin{enumerate}
\item $x_1,\dots,x_r$ are linearly independent modulo $p$,
\item $[x_i,x_j] \in \sum_{k=1}^r \Z_p x_k + p^\nu \lie{g}_{\Z_p}$, $1 \le i, j \le r$,
\item $\Ad (a_i) x_j \in \sum_{k=1}^r \Z_p x_k + p^\nu \lie{g}_{\Z_p}$, $1 \le i \le s$, $1 \le j \le r$.
\end{enumerate}
Then there exist $y_i\in x_i+p^{\lceil \varepsilon_1 \nu \rceil}\lie{g}_{\Z_p}$, $i=1,\dots,r$,
and $b_j\in a_j \K_p (p^{\lceil \varepsilon_1 \nu \rceil})$, $j=1,\dots,s$, such that
\begin{enumerate}
\item $I = \Z_p y_1 + \dots + \Z_p y_r$ is an isolated closed subalgebra of $\lie{g}_{\Z_p}$ of rank $r$,
\item $\Ad (b_j) I = I$, $j = 1, \ldots, s$.
\end{enumerate}
\end{lemma}

\begin{proof}
Modify the proof of Lemma \ref{LemmaLifting} by considering instead of $X_r$ the affine variety
$X_r \times G^s$, with the $\Z$-structure induced from the $\Z$-structure of the affine variety $G$,
and its subvariety $V_{r,s} \subset V_r \times G^s$ defined by the polynomials
$f_{ij,S} (x_1, \dots, x_r)$ introduced above together with the regular functions
$D_S (x_1, \dots, x_r, \Ad (y_i) x_j)$, $i = 1, \ldots, s$, $j=1,\ldots,r$,
of the variables $(x_1, \ldots, x_r) \in X_r$ and $y_1,\ldots,y_s \in G$.
An application of Theorem \ref{thm: Greenberg} yields the assertion.
\end{proof}

The following is the analog of Theorem \ref{thm: Liealgebra}.

\begin{lemma} \label{LemmaLieAlgebraWithA}
Let $G$ be as above and $s$ a positive integer.
Then there exist an integer $D > 0$ and a constant $0 < \varepsilon \le 1$ with the following property.
For any open Lie subalgebra $M \subset \lie{g}_{\Z_p}$ of level $p^n$, where $n > v_p (D)$,
and elements $a_1, \ldots, a_s \in \K_p$
with $\Ad (a_i) M = M$, $i = 1, \ldots, s$,
there exist a proper, isolated closed subalgebra $I$ of $\lie{g}_{\Z_p}$ such that
$M \subset I + p^{\lceil \varepsilon n \rceil} \lie{g}_{\Z_p}$, and elements
$b_j \in a_j \K_p (p^{\lceil \varepsilon n \rceil})$, $j=1,\dots,s$
such that $\Ad (b_j) I = I$.
\end{lemma}

\begin{proof}
This is proved exactly like Theorem \ref{thm: Liealgebra} except that we use Lemma \ref{LemmaLiftingWithai} instead of
Lemma \ref{LemmaLifting}.
We only need to check that the extra condition of Lemma \ref{LemmaLiftingWithai} is satisfied for the elements
$x_1,\dots,x_r$ constructed in the proof above.
Write
\[
\Ad(a_i)x_j=\sum_{k=1}^d d_{ijk} x_k, \quad d_{ijk}\in\Z_p,
\]
for $i = 1,\dots,s$ and $j=1,\dots,d$.
Then $p^{\alpha_j}\Ad(a_i)x_j\in M$ and therefore $d_{ijk} \in p^{\alpha_k - \alpha_j} \Z_p$.
For $1 \le i \le s$ and $1 \le j \le r < k \le d$ we have
\[
d_{ijk}\in p^{\alpha_k-\alpha_j}\Z_p\subset p^{\alpha_{r+1}-\alpha_r}\Z_p\subset p^{\lceil (1-c)\alpha_{r+1} \rceil}\Z_p\subset p^{\lceil (1-c)c^{d-r-1}n \rceil}\Z_p
\subset p^\nu \Z_p,
\]
where $\nu$ is as in \eqref{eqndefinitionnu}. Therefore, $\Ad (a_i) x_j \in
\sum_{k=1}^r \Z_p x_k + p^\nu \lie{g}_{\Z_p}$, $1 \le i \le s$, $1 \le j \le r$, as required.
\end{proof}

We also need to know that for $p$ large we can extend the domain of definition of the exponential and logarithm functions.
Let
\begin{eqnarray*}
\GL (N_0, \Z_p)_{\resuni} & = &
\{ x \in \GL (N_0, \Z_p) \, : \, x^{p^n} \equiv 1 \pmod{p} \text{ for some $n$} \} \\
& = &
\{ x \in \GL (N_0, \Z_p) \, : \, x^{p^n} \to 1, \quad n \to \infty \}
\end{eqnarray*}
be the open set of residually unipotent elements of $\GL (N_0, \Z_p)$. Thus, for $p\ge N_0$ we have
\[
\GL (N_0, \Z_p)_{\resuni}=\{ x \in \GL (N_0, \Z_p) \, : \, x^p \equiv 1 \pmod{p} \}.
\]
Similarly, let
\begin{eqnarray*}
\gl (N_0, \Z_p)_{\resnilp}
& = &
\{ y \in \gl (N_0, \Z_p) \, : \, y^{n} \equiv 0 \pmod{p} \text{ for some $n$} \} \\
& = &
\{ y \in \gl (N_0, \Z_p) \, : \, y^{n} \to 0, \quad n \to \infty \}
\end{eqnarray*}
be the set of residually nilpotent elements of $\gl (N_0, \Z_p)$,
so that for $p\ge N_0$ we have
\[
\gl (N_0, \Z_p)_{\resnilp}=\{ y \in \gl (N_0, \Z_p) \, : \, y^p \equiv 0 \pmod{p} \}.
\]

\begin{lemma} \label{lem: explog}
The power series $\exp$ (resp., $\log$) converges in the domain
$\gl (N_0, \Z_p)_{\resnilp}$ (resp., $\GL (N_0, \Z_p)_{\resuni}$)
provided that $p>N_0+1$.
Assume that $p>2N_0$. Then
\begin{enumerate}
\item We have
\[
\gl (N_0, \Z_p)_{\resnilp}\bjct{\exp}{\log}\GL (N_0, \Z_p)_{\resuni}.
\]
\item The diagrams
\begin{equation} \label{eq: expcomm}
\begin{CD}
\gl (N_0, \Z_p)_{\resnilp} @>\exp>> \GL (N_0, \Z_p)_{\resuni} \\
@VVV                                   @VVV \\
\gl (N_0, \F_p)_{\nilp} @>\exp^{(p)}>>  \GL(N_0,\F_p)_{\unip}
\end{CD}
\end{equation}
\begin{equation} \label{eq: logcomm}
\begin{CD}
\GL (N_0, \Z_p)_{\resuni} @>\log>> \gl (N_0, \Z_p)_{\resnilp}\\
@VVV @VVV\\
\GL(N_0,\F_p)_{\unip} @>\log^{(p)}>> \gl (N_0, \F_p)_{\nilp}
\end{CD}
\end{equation}
commute, where the vertical arrows denote reduction modulo $p$.
\item For any $n\ge1$ we get induced maps
\begin{equation} \label{eq: explogmodpn}
\gl (N_0, \Z_p)_{\resnilp}/p^n\gl (N_0, \Z_p)\bjct{\exp}{\log}
\GL (N_0, \Z_p)_{\resuni}/\Gamma (N_0, p^n).
\end{equation}
\end{enumerate}
\end{lemma}

\begin{proof}
We will prove the first part assuming only $p > N_0 + 1$.
Consider $\exp x=\sum_{n=0}^\infty x^k/k!$. If $x \in\gl (N_0, \Z_p)_{\resnilp}$ then $x^{N_0}\in p\gl(N_0,\Z_p)$.
It follows that $\norm{x^k}\le p^{-\lfloor k/N_0\rfloor}$ for all $k$.
Since $v_p(k!)\le (k-1)/(p-1)$ for $k>0$, we get
\[
\norm{x^k/k!}\le p^{-\lfloor k/N_0\rfloor+\lfloor (k-1)/(p-1)\rfloor}\rightarrow0\ \ \ \text{as }k\rightarrow\infty,
\]
provided that $p-1>N_0$, which gives the convergence of $\exp x$. Moreover, under the same restriction on $p$
we have $\norm{x^k/k!} \le 1$ for all $k$, which shows that the matrices $x_k = x^k/k!$ and their sum $\exp x$ have integral entries.
Since the elements $x^k$, $k>0$ are residually nilpotent and pairwise commute, we conclude that $\exp x \in \GL (N_0, \Z_p)_{\resuni}$.
A similar reasoning applies to $\log$. That the maps $\exp: \gl (N_0, \Z_p)_{\resnilp} \to \GL (N_0, \Z_p)_{\resuni}$ and
$\log: \GL (N_0, \Z_p)_{\resuni} \to \gl (N_0, \Z_p)_{\resnilp}$ so obtained are mutually inverse follows from \cite[\S II.8.4, Proposition 4]{MR1728312} (applied to
the algebra of all square matrices of size $N_0$ with entries in $\C_p$ equipped with the operator norm).\footnote{In fact, [ibid.] also gives the convergence of $\exp$ and $\log$ on
$\gl (N_0, \Z_p)_{\resnilp}$ and $\GL (N_0, \Z_p)_{\resuni}$, respectively.}

If moreover $p\ge 2N_0$ then we have $\lfloor k/N_0\rfloor\ge 1+\lfloor (k-1)/(p-1)\rfloor$
for any $k\ge p$. Thus, if $p>2N_0$ then the diagram \eqref{eq: expcomm} commutes, and similarly we obtain \eqref{eq: logcomm}.

The existence of the maps in \eqref{eq: explogmodpn} is proved like
part \ref{part: expcongruence1} of Proposition \ref{prop: unif}. We need to show that
\begin{equation} \label{eq: expcongruencesection3}
\exp(x+p^ny) \equiv \exp x \pmod {p^n\gl(N_0,\Z_p)}
\end{equation}
for any $x \in \gl (N_0, \Z_p)_{\resnilp}$, $y\in\gl(N_0,\Z_p)$ and $n\ge1$.
Expanding the power series as an infinite linear combination of products of $x$ and $y$ (which do not commute in general),
we may first observe that we only need to consider terms of total degree
\begin{equation} \label{eq: kbound}
1 \le k < \frac{n}{N_0^{-1} - (p-1)^{-1}},
\end{equation}
since for all larger $k$ we have
$\norm{x^k/k!}$, $\norm{(x+p^n y)^k / k!} \le p^{-\lfloor k/N_0\rfloor+\lfloor (k-1)/(p-1)\rfloor} \le p^{-n}$.
We will show that all summands in the range \eqref{eq: kbound} involving $y$ are actually $\equiv 0 \pmod{p^n}$.
If $y$ occurs at least twice, then the corresponding summand has norm $\le p^{-2n + v_p (k!)} = p^{-n} p^{-n + v_p (k!)} < p^{-n}$,
since $v_p (k!) < k / (p-1) < n$ for $k$ as in \eqref{eq: kbound} and $p > 2 N_0$. It remains to consider the terms
$x^i y x^{k-1-i} / k!$ for $0 \le i \le k-1$. We can bound the norm of this term by
$p^{-n - \lfloor i / N_0 \rfloor - \lfloor (k-1-i)/N_0 \rfloor + v_p (k!)}$. For $k < p$ this is clearly $\le p^{-n}$. For $k \ge p$ we have
\[
\left\lfloor \frac{i}{N_0} \right\rfloor + \left\lfloor \frac{k-1-i}{N_0} \right\rfloor - v_p (k!) > \frac{k-1}{N_0} - 2 - \frac{k-1}{p-1}
\ge \frac{p-1}{N_0} - 3 \ge -1,
\]
and therefore $-n - \lfloor i / N_0 \rfloor - \lfloor (k-1-i)/N_0 \rfloor + v_p (k!) \le -n$,
which establishes the congruence \eqref{eq: expcongruencesection3}.
The analogous congruence for $\log$ can be proven similarly.
\end{proof}

Using these extensions of $\log$ and $\exp$, we have the following consequence.

\begin{lemma} \label{LiftingHPlus}
Assume $p > 2N_0$. Let $X \subset G$ be an algebraic subgroup defined over $\Q_p$, and $h \in \K_p$
with $h^p \in (X (\Q_p) \cap \K_p (p)) \K_p (p^m)$ for some $m \ge 1$.
Then $h \in (X (\Q_p) \cap \K_p) \K_p (p^{m-1})$.
\end{lemma}

\begin{proof}
The lemma is trivially true for $m=1$, so we can assume that $m>1$.
Note that $h^p \in \K_p (p)$, so that $h\in \GL (N_0, \Z_p)_{\resuni}$.
By the first part of Lemma \ref{lem: explog} we can write
$h = \exp y$ where $y = \log h\in\gl (N_0, \Z_p)_{\resnilp}$.
Then $py = \log (h^p) \in \log ((X (\Q_p) \cap \K_p (p)) \K_p (p^m))
= p (\Lie_{\Q_p} X \cap \lie{g}_{\Z_p}) + p^m \lie{g}_{\Z_p}$
by parts \ref{part: expcongruence} and \ref{eqnalgebraicsubgroup2}
of Proposition \ref{prop: unifg}, and therefore $y \in (\Lie_{\Q_p} X \cap \lie{g}_{\Z_p}) + p^{m-1} \lie{g}_{\Z_p}$.
Let $y' \in \Lie_{\Q_p} X \cap \lie{g}_{\Z_p}$ be such that $y-y' \in p^{m-1} \lie{g}_{\Z_p}$.
Clearly, $y'\in\gl (N_0, \Z_p)_{\resnilp}$.
By the last part of Lemma \ref{lem: explog} we have $\exp y \in (\exp y') \K_p (p^{m-1})$.
On the other hand, by Lemma \ref{lem: formalps} we have $\exp (y') \in X (\Q_p)$.
The lemma follows.
\end{proof}

In the following, for any closed subgroup $H \subset \K_p$ we will denote by $H^+$ the open normal subgroup of $H$ generated by the set $H_{\resuni}$.

\begin{proof}[Proof of Theorem \ref{TheoremAlgebraicLevel0}]
First note that we may assume without loss of generality that $p>\nmbr(G)$
since otherwise we can apply Theorem \ref{TheoremAlgebraic} to $H' = H \cap \K_p (\altp)$ as long as
$J \ge [\K_p : \K_p (\altp)]$. In particular, we may assume in the following that $p \ge \max(2 N_0+1,\dim G)$.

Let $H \subset \K_p$ be an arbitrary open subgroup.
The quotient $H/H^+$ is finite of order prime to $p$.
Repeating the argument of the proof of \cite[Theorem C]{MR880952} in the setting of profinite groups,
we see that $H = H_1 H^+$ for a finite subgroup $H_1$ of $H$ of order prime to $p$.
Applying Proposition \ref{PropositionJordan} to $H_1$, there exists a maximal torus $T$ of $G$ defined over $\Q_p$ such that
$[H:AH^+] \le J$ where $A=T(\Q_p)\cap H_1$ and $J$ is a constant depending only on $G$.
Clearly, $A$ can be generated by elements $a_1, \ldots, a_s$, where $s$
is bounded in terms of $G$ only, namely by the $\overline{\Q}$-rank of $G$.

Let $\varepsilon$ be as in Lemma \ref{LemmaLieAlgebraWithA}.
We will show below that there exists a proper, connected algebraic subgroup $Y$ of $G$ defined over $\Q_p$ such that
\begin{equation} \label{eq: existsY}
H^+ \subset (Y (\Q_p) \cap \K_p) \K_p(p^{\lceil \varepsilon n / 2 \rceil})
\text{ and }N_G(Y)(\Q_p)\cap a_i\K_p(p^{\lceil \varepsilon n / 2 \rceil})\ne\emptyset, i=1,\dots,s.
\end{equation}

Let us see how to obtain the theorem from this assertion.
If $Y$ is normal in $G$, we simply take $X=YT$.
Otherwise, let $X=N_G (Y)^\circ\supset Y$ (which holds since $Y$ is connected) so that
$H^+ \subset (X (\Q_p) \cap \K_p) \K_p(p^{\lceil \varepsilon n / 2 \rceil})$. The index of $X$ in $N_G(Y)$ is bounded
in terms of $G$ only by Lemma \ref{LemmaGeometricFiniteness}.
Therefore, choosing $b_i\in N_G(Y)(\Q_p)\cap a_i\K_p(p^{\lceil \varepsilon n / 2 \rceil})$, $i=1,\dots,s$,
there exists an exponent $e$, depending only on $G$, such that $b_i^e \in X (\Q_p) \cap \K_p$, $i = 1, \ldots, s$.
The projection of $A H^+$ to the factor group $\K_p / \K_p (p^{\lceil \varepsilon n / 2 \rceil})$
is generated by the image of the group $H^+$ and the images of the elements $b_i$.
We conclude that the open subgroup $A^e H^+$ of $H$, which has index at most $J' = e^s J$, is contained in the group
$(X (\Q_p) \cap \K_p) \K_p (p^{\lceil \varepsilon n / 2 \rceil})$.

It remains to prove \eqref{eq: existsY}. Assume first that $n \ge \frac{2}{\varepsilon}$.
Consider the open subgroup $H \cap \K_p (p)$ of $\K_p (p)$, which is normalized by the elements $a_1, \ldots, a_s$.
Since we assume that $p \ge \dim G$, the subset $M = \log (H \cap \K_p (p)) \subset p \lie{g}_{\Z_p}$ is by Remark \ref{rem: Ilani} a Lie subalgebra, which
is clearly of level $p^n$ in $\lie{g}_{\Z_p}$ and stable under $\Ad (a_1), \ldots, \Ad (a_s)$.
(Alternatively, we may use part \ref{LemmaUniform} of Lemma \ref{lem: basunif} and imitate the
proof of Theorem \ref{TheoremAlgebraic} above, noting that the uniform subgroup $V$ is characteristic in $H \cap \K_p (p)$.)
From Lemma \ref{LemmaLieAlgebraWithA} we obtain a proper, isolated Lie subalgebra $I \subset \lie{g}_{\Z_p}$  such that
\[
M\subset (I+p^{\lceil \varepsilon n \rceil} \lie{g}_{\Z_p})\cap p \lie{g}_{\Z_p}=pI+p^{\lceil \varepsilon n \rceil} \lie{g}_{\Z_p},
\]
and elements $b_i\in a_i\K_p (p^{\lceil \varepsilon n \rceil})$, $i=1,\dots,s$, such that $\Ad (b_i) I = I$.
We infer that $H \cap \K_p (p) \subset (\exp p I) \K_p (p^{\lceil \varepsilon n \rceil})$.
As in the proof of Theorem \ref{TheoremAlgebraic}, consider the algebraic subgroup $Y = \mathcal{A} (\Q_p I)$ of $G$, which is
a proper, connected algebraic subgroup of $G$ defined over $\Q_p$.
Clearly, $b_1, \ldots, b_s \in N_G (Y) (\Q_p)$.
We can now invoke Lemma \ref{LiftingHPlus} to lift the relation
$H \cap \K_p (p) \subset \exp (pI) \K_p (p^{\lceil \varepsilon n \rceil})
\subset (Y (\Q_p) \cap \K_p (p)) \K_p (p^{\lceil \varepsilon n \rceil})$ to $H^+$ and obtain
$H^+ \subset (Y (\Q_p) \cap \K_p) \K_p(p^{\lceil \varepsilon n \rceil - 1})$.
Thus, \eqref{eq: existsY} holds by the assumption on $n$.

Consider now the case where $1 \le n < \frac{2}{\varepsilon}$.
Let $\bar{H}$ be the image of $H \neq \K_p$ in $\K_p / \K_p (p) \simeq G (\F_p)$.
Since $G$ is simply connected, $\bar H$ is a proper subgroup of $G (\F_p)$ (for $p>\nmbr(G)$) by \eqref{eq: fratinni}
and Nori's algebraic envelope $\bar{Y}$ of $\bar H$ is a proper subgroup of $G_{\F_p}$ (cf.~the proof of Proposition \ref{PropositionApproxLevelp}).
It follows from Theorems \ref{TheoremNoriA} and \ref{TheoremNori} that $\Lie_{\F_p}\bar Y$ is spanned by $\log^{(p)}x$, $x\in\bar H_{\unip}$,
and is a proper Lie subalgebra of $\lie{g}_{\F_p}$.
Let $M$ be the inverse image of $\Lie_{\F_p}\bar Y$ under $\lie{g}_{\Z_p}\rightarrow\lie{g}_{\Z_p}/p\lie{g}_{\Z_p}$,
so that $M$ is a proper $\Z_p$-Lie subalgebra of $\lie{g}_{\Z_p}$.
Since $H_{\resuni}$ is the inverse image of $\bar H_{\unip}$ under the reduction map $H \to \bar H$,
the commutativity of \eqref{eq: logcomm} gives
$M= \sum_{h\in H_{\resuni}} \Z_p \log h + p\lie{g}_{\Z_p}$.

Since $p>\nmbr(G)$, we can use Lemma \ref{LemmaLieAlgebraWithA} to obtain the inclusion
$M \subset I + p \lie{g}_{\Z_p}$ for a proper, isolated subalgebra $I \subset \lie{g}_{\Z_p}$ and elements
$b_i\in a_i\K_p(p)$, $i=1,\dots,s$, such that $\Ad (b_i) I = I$. As above, we can assume that
$I = \Lie_{\Q_p} Y \cap \lie{g}_{\Z_p}$ for a proper, connected algebraic subgroup $Y$ of $G$ defined over $\Q_p$ and retain the property that $I$
is stable under the operators $\Ad (b_i)$.
Since $H^+$ is generated by elements $h$ with $h^p \in \K_p (p)$,
and for such an $h$ we have $h^p \in \exp (pM) \subset \exp (pI) \K_p (p^2) = (Y (\Q_p) \cap \K_p (p)) \K_p(p^2)$, we obtain from Lemma \ref{LiftingHPlus} that
$H^+ \subset (Y (\Q_p) \cap \K_p) \K_p(p)$, so that \eqref{eq: existsY} holds in this case as well. (Alternatively, Nori's proof shows that
$\bar{Y}$ extends to a smooth group scheme over $\Z_p$, and in particular lifts to characteristic zero.)
\end{proof}

If we are willing to use Theorem \ref{thm: corrlev0}, the main result of \S \ref{sectionliealgebra}, then we may simplify the proof above (see Remark \ref{RemarkUnification} below).

\section{Closed subgroups of $G (\Z_p)$ and closed subalgebras of $\lie{g}_{\Z_p}$} \label{sectionliealgebra}

\subsection{Statement of the correspondence}
We continue to use the conventions of \S \ref{subspropsetup}.
In this section we will establish (for $p$ large with respect to $G$) a correspondence
between subgroups of $G(\Z_p)$ which are residually exponentially generated and $\Z_p$-Lie subalgebras of $\lie{g}_{\Z_p}$
which are residually nilpotently generated. On the one hand, this bijection extends
previously known results on pro-$p$ subgroups of $G (\Z_p)$, and on the other hand it is compatible, via reduction modulo $p$, with Nori's correspondence between subgroups
of $G(\F_p)$ and subalgebras of $\lie{g}_{\F_p}$.

We first recall that Nori's two theorems (Theorems \ref{TheoremNoriA} and \ref{TheoremNori}) establish for $p > n (G)$ bijective correspondences between three different collections of objects:
subgroups $\bar{H}$ of $G (\F_p)$ that are generated by their elements of order $p$ (i.e. for which $\bar{H}^+ = \bar{H}$), nilpotently generated subalgebras of $\lie{g}_{\F_p}$,
and exponentially generated algebraic
subgroups $\mathcal{H}$ of $G_{\F_p}$. We will focus on the bijection between the first two collections, which we can describe as follows: to a subgroup $\bar{H}$ of $G (\F_p)$
we associate $\bar{\Liec} (\bar{H})$, the $\F_p$-span of $\log^{(p)} \bar{H}_{\unip}$ (which is a Lie algebra), and to
$\bar{\lie{h}}$ the subgroup $\bar{\grpc} (\bar{\lie{h}})$ of $G (\F_p)$ generated by $\exp^{(p)} \bar{\lie{h}}_{\nilp}$.
Regarding the third collection, to an algebraic group $\mathcal{H} \subset G_{\F_p}$ correspond its Lie algebra $\Lie_{\F_p} \mathcal{H}$ and the finite group $\mathcal{H} (\F_p)^+$.
The opposite bijections were described in \S \ref{SectionApproxLevel0} above.

Next, we recall the following correspondence due to Klopsch \cite{MR2126210} (see also \cite{MR2554763}) generalizing Ilani's correspondence.

\begin{theorem}[Lazard--Ilani--Klopsch] \label{TheoremIlaniKlopsch}
Assume that $p\ge\max(N_0+2,\dim G)$.
Then the exponential and logarithm maps (applied to subsets of $\lie{g}_{\Z_p,\resnilp}$ and $G(\Z_p)_{\resunip}$, respectively) induce mutually inverse bijections between
the set of all closed pro-$p$ subgroups $H \subset G (\Z_p)$ and the set of all closed Lie subalgebras $\lie{h} \subset \lie{g}_{\Z_p}$ consisting of residually nilpotent elements.
\end{theorem}

\begin{remark}
Just like in Ilani's work, the correspondence in \cite{MR2126210} is stated in terms of the intrinsic Lie algebra structure on a saturable pro-$p$ group
in the sense of Lazard.
However, by \cite[Proposition IV.3.2.3]{MR0209286} the intrinsic Lie algebra is isomorphic to the one defined using the logarithm map
(cf. \cite[\S 7.2]{MR1720368}).
\end{remark}

Throughout we assume that $p\ge\max(N_0+2,\dim G)$.
In fact, we also assume that $p > 2 N_0$, so that the exponential and logarithm bijections satisfy the properties of Lemma \ref{lem: explog}.

The basic construction underlying our extension of this theorem to more general subgroups of $G (\Z_p)$ is a generalization of Nori's correspondence.
Let $H$ be a subgroup of $G (\Z_p)$.
We associate to $H$ the $\Z_p$-submodule $\Liec(H)\subset \lie{g}_{\Z_p}$ generated by the set $\log H_{\resunip}$.
We will show in Proposition \ref{PropLieAlgebraLevel0} below that $\Liec(H)$ is in fact always a $\Z_p$-Lie algebra provided that $p>\dim G+1$.
Conversely, to any Lie subalgebra $\lie{h} \subset \lie{g}_{\Z_p}$ we associate the closed subgroup $\grpc(\lie{h})$ of $G (\Z_p)$
topologically generated by the set $\exp \lie{h}_{\resnilp}$.

By Theorem \ref{TheoremIlaniKlopsch}, for pro-$p$ groups $H$ we have $\Liec(H) = \log H$ and for subalgebras $\lie{h}$ contained in
$\lie{g}_{\Z_p,\resnilp}$ we have $\grpc (\lie{h}) = \exp \lie{h}$.

We extend Theorem \ref{TheoremIlaniKlopsch} to a correspondence between the following classes of subgroups and subalgebras.
On one side, we consider the set $\SGRU$ of closed subgroups $H \subset G (\Z_p)$ which
are topologically generated by their residually unipotent elements.
Clearly, $H \in \SGRU$ precisely when $H^+ = H$, where $H^+$ is as in \S \ref{SectionApproxLevel0} above.
It is equivalent to demand that $\bar{H} \subset G (\F_p)$, the reduction modulo $p$ of $H$, is generated by its elements of order $p$.

On the other side, we consider the set $\LSAN$ of all closed Lie subalgebras of $\lie{g}_{\Z_p}$ which are spanned over $\Z_p$ by their
residually nilpotent elements.
Clearly, $\lie{h}\in\LSAN$ if and only if the reduction $\bar{\lie{h}}$ is a nilpotently generated
subalgebra of $\lie{g}_{\F_p}$.

\begin{theorem} \label{thm: corrlev0}
Suppose that $p$ is large with respect to $G$. Then
\[
\SGRU\bjct{\Liec}{\grpc}\LSAN.
\]
These bijections preserve the level of open subgroups and subalgebras, extend those of Theorem \ref{TheoremIlaniKlopsch},
and are compatible with those of Nori under reduction modulo $p$,
that is if $H\in\SGRU$ and $\lie{h}=\Liec(H)$ then the image $\bar{H}$ of $H$ in $G(\F_p)$ corresponds under Nori to the image $\bar{\lie{h}}$
of $\lie{h}$ in $\lie{g}_{\F_p}$, i.e., $\bar{\lie{h}} = \bar{\Liec} (\bar{H})$.
\end{theorem}

The theorem will be proved below. Clearly, the special case of subgroups and subalgebras of level $p$ reduces to Nori's correspondence.

\begin{remark} \label{RemarkUnification}
Using Theorem \ref{thm: corrlev0} we can streamline the proof of Theorem \ref{TheoremAlgebraicLevel0}.
Namely, to prove \eqref{eq: existsY} (with $\varepsilon$ instead of $\varepsilon/2$) we apply Lemma \ref{LemmaLieAlgebraWithA}
to $M=\Liec(H^+)$, which is of level $p^n$ in $\lie{g}_{\Z_p}$ and stable under $\Ad (a_1), \ldots, \Ad (a_s)$.
We obtain a proper, isolated Lie subalgebra $I \subset \lie{g}_{\Z_p}$ such that $M\subset I+p^{\lceil \varepsilon n \rceil} \lie{g}_{\Z_p}$
and elements $b_i\in a_i \K_p (p^{\lceil \varepsilon n \rceil})$, $i=1,\dots,s$, such that $\Ad (b_i) I = I$.
We infer that $H^+\subset \grpc(I) \K_p (p^{\lceil \varepsilon n \rceil})$.
Taking the algebraic subgroup $Y = \mathcal{A} (\Q_p I)$ of $G$, which is
a proper, connected algebraic subgroup of $G$ defined over $\Q_p$, we see that
$b_1, \ldots, b_s \in N_G (Y) (\Q_p)$ and $H^+\subset Y(\Q_p)\K_p (p^{\lceil \varepsilon n \rceil})$, as required.
We opted to include the original proof of Theorem \ref{TheoremAlgebraicLevel0} since it is simpler and more self-contained.
\end{remark}

\subsection{The Lie algebra associated to a subgroup}
We start with some simple considerations (in the spirit of \cite[\S1]{MR880952}) and establish that $\Liec(H)$ is indeed a Lie subalgebra of $\lie{g}_{\Z_p}$ (for $p > \dim G +1$).

\begin{lemma} \label{LemmaSubmoduleResnilp}
Let $V$ be a finitely generated free $\Z_p$-module, $U$ a submodule of $V$ and $s\in\End_{\Z_p}(V)_{\resnilp}$.
Assume that $sU\subset U$. Then $s\rest_U\in\End_{\Z_p}(U)_{\resnilp}$. In particular, $s^{\operatorname{rk} U}(U)\subset pU$.
\end{lemma}

\begin{proof}
We know that $s^n \to 0$ as $n \to \infty$. Clearly, this implies that $s^n\rest_U = (s\rest_U)^n \to 0$ as $n \to \infty$, or
$s\rest_U\in\End_{\Z_p}(U)_{\resnilp}$.
\end{proof}

\begin{corollary} \label{cor: expsinv}
Let $V$ be a free $\Z_p$-module of rank $r<p-1$, $U$ a submodule of $V$ and $s\in\End_{\Z_p}(V)_{\resnilp}$.
Then $U$ is $s$-invariant if and only if it is $\exp(s)$-invariant.
\end{corollary}

\begin{proof}
Suppose that $U$ is $s$-invariant. Then it follows from Lemma \ref{LemmaSubmoduleResnilp} that for all $n\ge1$ we have
$(n!)^{-1}(s\rest_U)^n\in p^{k_n}\End_{\Z_p}(U)$ where $k_n=\lfloor n/\operatorname{rk} U\rfloor-\lfloor(n-1)/(p-1)\rfloor\ge0$.
Since $k_n\rightarrow\infty$ as $n\rightarrow\infty$ and $U$ is closed, we conclude that $U$ is invariant under $\exp s$.
The converse is proved in a similar way using $\log$.
\end{proof}

\begin{proposition} \label{PropLieAlgebraLevel0}
Assume that $p>\dim G+1$ and let $H$ be a subgroup of $G(\Z_p)$.
Then the $\Z_p$-submodule $\lie{h} = \Liec (H)$ of $\lie{g}_{\Z_p}$ generated by the set $\log H_{\resuni}$ is a $\Z_p$-Lie subalgebra.

For $p > n (G)$, the image of $\lie{h}$ in $\lie{g}_{\F_p}$ is the Lie algebra $\Lie_{\F_p} \tilde H$ of
Nori's algebraic envelope $\tilde H \subset G_{\F_p}$ of the image $\bar H$ of $H$ in $G (\F_p)$.
\end{proposition}

We remark that under our standing assumption $p > 2 N_0$ the image of $\lie{h}$ in $\lie{g}_{\F_p}$ is always a Lie algebra (to see this, combine \cite[Lemma 1.6]{MR880952} with \eqref{eq: logcomm}).

\begin{proof}
Clearly, $\lie{h}$ is $\Ad(h)$-invariant for any $h\in H$, and in particular for any $h\in H_{\resuni}$.
We have $\Ad(h)=\exp(\ad(\log h))$ for any $h\in H_{\resuni}$. By Corollary \ref{cor: expsinv} we conclude
that $\lie{h}$ is $\ad(\log h)$-invariant for any $h\in H_{\resuni}$. The first claim follows.

The second assertion follows from the commutativity of \eqref{eq: logcomm} and from Theorems \ref{TheoremNoriA} and \ref{TheoremNori}
as in the proof of Theorem \ref{TheoremAlgebraicLevel0}.
\end{proof}

\begin{remark}
At this stage we can already prove that if $H$ is an open subgroup of $G(\Z_p)$ of level $p^n$, then $\lie{h} = \Liec (H)$ has level $p^n$ or $p^{n-1}$.
(Eventually we will prove that $\lie{h}$ has level $p^n$.)
Indeed, $\log(H\cap \K_p(p))$ is a Lie subalgebra of $\lie{g}_{\Z_p}$ of level $p^n$ by Remark \ref{rem: Ilani},
and $h^p \in H \cap \K_p (p)$ for all $h \in H_{\resuni}$.
Therefore $\log (H \cap \K_p (p))\subset\lie{h} \subset p^{-1}\log (H \cap \K_p (p))$, which shows that $\lie{h}$ has level $p^n$ or $p^{n-1}$.
Moreover, it is also clear at this point that if $H$ has level $p$, then $\lie{h}$ has level $p$.
\end{remark}

The remaining parts of Theorem \ref{thm: corrlev0} are easily deduced from the following two statements which will be proved below.
\begin{subequations}
\begin{gather}
\label{step: fromgroup} \text{For any }H\in\SGRU\text{ we have }\Liec(H)\cap p \lie{g}_{\Z_p} =\log (H \cap \K_p (p)).\\
\label{step: fromalgebra} \text{For any }\lie{h}\in\LSAN\text{ we have }\grpc(\lie{h})\cap\K_p(p)=\exp(\lie{h}\cap p\lie{g}_{\Z_p}).
\end{gather}
\end{subequations}
Indeed, given $H\in\SGRU$ let $\lie{h}=\Liec(H)$ and $H'=\grpc(\lie{h})$. It is clear that $H'\supset H$.
By Nori's theorems, the reductions modulo $p$ of $H'$ and $H$ coincide.
Moreover, by \eqref{step: fromgroup} and \eqref{step: fromalgebra} we have $H' \cap \K_p (p) = \exp (\lie{h} \cap p \lie{g}_{\Z_p}) = H \cap \K_p (p)$.
Therefore, $H' = H$. The other direction can be shown in a similar way.

The following lemma provides a complement to Theorem \ref{TheoremIlaniKlopsch} that will be useful below.

\begin{lemma} \label{lem: untitled}
Let $p > \dim G + 1$.
Let $P = \exp \lie{p}$ be a pro-$p$ subgroup of $G (\Z_p)$ and $h \in G(\Z_p)_{\resunip}$ a residually unipotent element normalizing $P$.
Then we have $[\log h, \lie{p}] \subset \lie{p}$, $h P \subset G(\Z_p)_{\resunip}$ and $\log (P h) = \log (h P) = \log h + \lie{p}$.

Similarly, if $\lie{p}$ is a closed Lie subalgebra of $\lie{g}_{\Z_p}$ contained in
$\lie{g}_{\Z_p,\resnilp}$, and a residually nilpotent element $u \in \lie{g}_{\Z_p,\resnilp}$ satisfies $[u,\lie{p}] \subset \lie{p}$, we have $u+\lie{p} \subset \lie{g}_{\Z_p,\resnilp}$ and
$\exp (u+\lie{p}) = \exp (u) P = P \exp (u)$, where $P = \exp \lie{p}$.
\end{lemma}

\begin{proof}
Since $h$ normalizes $P$, the group $Q$ generated by $h$ and $P$ is a pro-$p$ group. Therefore $hP \subset Q \subset G(\Z_p)_{\resunip}$.
Furthermore, $hPh^{-1} = P$ implies that $\Ad (h) \lie{p} = \lie{p}$, and by Corollary \ref{cor: expsinv} we obtain $[\log h, \lie{p}] \subset \lie{p}$.

Analogously, if we have $\lie{p}$ contained in $\lie{g}_{\Z_p,\resnilp}$, and $u \in \lie{g}_{\Z_p,\resnilp}$ satisfies $[u,\lie{p}] \subset \lie{p}$, then
the Lie algebra generated by $\bar{u}$ and $\bar{\lie{p}}$ inside $\lie{g}_{\F_p}$ clearly consists of nilpotent elements, and therefore
$u+\lie{p} \subset \lie{g}_{\Z_p,\resnilp}$.

Let $\lie{p} \subset \lie{g}_{\Z_p,\resnilp}$ and $u \in \lie{g}_{\Z_p,\resnilp}$ with
$[u, \lie{p}] \subset \lie{p}$.
Denote by $\Phi (x,y) = \sum_{n=1}^\infty \Phi_n (x,y)$ the Hausdorff series \cite[\S II.6.4]{MR1728312},
where $\Phi_n$ is a homogeneous Lie polynomial of degree $n$, and $\Phi_1 (x,y) = x+y$.
Since $\Z_p u + \lie{p}$ is a $\Z_p$-Lie algebra contained in $\gl (N_0, \Z_p)_{\resnilp}$, the Hausdorff series converges on $(\Z_p u + \lie{p})^2$
and we have $\Phi (x,y ) = \log(\exp x\exp y)$ for all $x$, $y \in \Z_p u + \lie{p}$ (\cite[\S II.8.3,II.8.4]{MR1728312}).

It remains to show that $\Phi (u, \lie{p}) \subset u + \lie{p}$ and $\Phi (-u,u+\lie{p}) \subset \lie{p}$.
as these conditions are equivalent to $\exp(u)P\subset\exp(u+\lie{p})$ and $\exp(u+\lie{p})\subset\exp(u)P$, respectively.
(The relation $\log (P h) = \log (h P) = \log h + \lie{p}$ will then follow by taking $\log$.)
For this, it suffices to show that $\Phi_n (u, \lie{p})$, $\Phi_n (-u,u+\lie{p}) \subset \lie{p}$ for all $n \ge 2$.
By \cite[Lemma 4.2]{MR2126210} applied to $K = \lie{p}$, $H = \Z_p u + \lie{p}$ and $j = 1$, and taking into account that $p > \dim G + 1$,
the iterated commutator of an element of $\lie{p}$ and $n-1$ elements of $\Z_p u + \lie{p}$ lies in $p^{\lfloor \frac{n-1}{p-1} \rfloor} \lie{p}$.
On the other hand, by a standard estimate for the denominator of $\Phi_n$ \cite[\S II.8.1, Proposition 1]{MR1728312},
we know that $p^{\lfloor \frac{n-1}{p-1} \rfloor} \Phi_n (x,y)$ is a $\Z_p$-linear combination of iterated commutators of degree $n$ in $x$ and $y$.
This implies that indeed $\Phi_n (u, \lie{p})$, $\Phi_n (-u,u+\lie{p}) \subset \lie{p}$, which finishes the proof.
\end{proof}

\subsection{Application of Nori's theorem}
The proof of \eqref{step: fromgroup} uses crucially Nori's correspondences recalled above. We first summarize some easy complements to Nori's theorems in the following lemma.

\begin{lemma} \label{LemmaComplementNori}
Let $\bar{H}$ be a subgroup of $G (\F_p)$ with $\bar{H}^+=\bar{H}$, $\bar{\lie{h}} = \bar{\Liec} (\bar{H}) \subset \lie{g}_{\F_p}$ the associated Lie algebra, and $\tilde H$ the
associated exponentially generated algebraic subgroup of $G_{\F_p}$. Let $\bar P$ be the maximal normal $p$-subgroup of $\bar H$ and $\tilde{P}$ the unipotent radical of $\tilde H$
(which is also the solvable radical of $\tilde H$). Finally, let $S^{\SC}$ be the simply connected covering group of the semisimple group $\tilde H / \tilde P$. Then
\begin{enumerate}
\item The following four conditions are equivalent: $\bar H$ is a $p$-group; $\tilde H$ is a unipotent algebraic group; $\bar{\lie{h}}$ is contained in $\lie{g}_{\F_p,\nilp}$;
$\bar{\lie{h}}$ is a nilpotent Lie algebra.
\item $\bar H_1$ is a normal subgroup of $\bar H_2$ if and only if $\tilde H_1$ is a normal subgroup of $\tilde H_2$ if and only if
$\bar{\lie{h}}_1$ is an ideal of $\bar{\lie{h}}_2$.
\item $\bar P$ corresponds to $\tilde{P}$ and to the nilradical of $\lie{h}$ (which equals the radical of $\bar{\lie{h}}$ as well as the largest ideal of $\bar{\lie{h}}$ contained in $\lie{g}_{\F_p,\nilp}$).
\item \label{part: cover} The covering map $\kappa:S^{\SC}\rightarrow\tilde H/\tilde P$ induces a surjection $S^{\SC}(\F_p)\rightarrow\bar H/\bar P$
(cf. \cite[Remark 3.6]{MR880952}). The adjoint action of $\bar H$ on $\bar{\lie{h}}$ induces an action of $\bar H/\bar P$ on
$\bar{\lie{h}}/\bar{\lie{p}}$ which is compatible under the above map with the adjoint action of $S^{\SC}$ on its Lie algebra.
\item A Sylow $p$-subgroup of $\bar H$ corresponds to a maximal unipotent subgroup of $\tilde H$ (i.e., to the unipotent radical of a Borel subgroup)
and to a maximal nilpotent subalgebra of $\bar{\lie{h}}$ contained in $\lie{g}_{\F_p,\nilp}$, or equivalently, to the nilradical of a maximal solvable subalgebra of $\bar{\lie{h}}$.
\end{enumerate}
\end{lemma}

\begin{proof}
We only need to observe that if $\lie{h}$ is a nilpotently generated Lie subalgebra of $\gl (N_0,\F_p)$ then $\lie{h}$ is nilpotent if and only if $\lie{h}$ is solvable if and only if
$\lie{h}$ is contained in $\gl (N_0,\F_p)_{\nilp}$. This follows from Engel's and Lie's theorems (the latter is valid over the algebraic closure for $p>N_0$).
\end{proof}

\renewcommand{\arraystretch}{1.5}
\begin{table}[h]
\caption{Nori's correspondence}
\centering
\begin{tabularx}{\textwidth}{ X|X|X }

\hline\hline
subgroups $H$ of $G (\F_p)$ with $H^+=H$ & exponentially generated algebraic subgroups of $G_{\F_p}$ & nilpotently generated Lie subalgebras of $\lie{g}_{\F_p}$ \\
$p$-groups & unipotent groups & Lie subalgebras contained in $\lie{g}_{\F_p,\nilp}$\\
normal subgroups & normal subgroups & ideals\\
maximal normal $p$-subgroup & unipotent (or solvable) radical & nilradical (or radical) \\
a Sylow $p$-subgroup & a maximal unipotent subgroup & the nilradical of a maximal solvable subalgebra
\end{tabularx}
\end{table}

Let now $H\in\SGRU$ and write $\lie{h}=\Liec(H)$. Let $\bar{H}$ be the image of $H$ in $G (\F_p)$. Clearly, $H$ acts on $\lie{h}$ by $\Ad$.
Let $P$ be the largest normal pro-$p$ subgroup of $H$ and $\bar{P}$ the image of $P$ in $G (\F_p)$.
Write $S := H /P$. It is clear that $P$ contains $H \cap \K_p (p)$. Thus, under reduction modulo $p$ we have the isomorphism $S \simeq\bar H/\bar P$.
Also, $\lie{p} = \log P$ is a Lie subalgebra of $\lie{h}$ by Theorem \ref{TheoremIlaniKlopsch}, and
by Corollary \ref{cor: expsinv} in fact an ideal.
It remains to show that $\lie{h} \cap p \lie{g}_{\Z_p}\subset \lie{p}$, since in this case
\[
\lie{h} \cap p \lie{g}_{\Z_p} = \lie{p}\cap p\lie{g}_{\Z_p}=\log(P\cap\K_p(p))=\log (H \cap \K_p (p)).
\]

Let $\lie{s}=\lie{h}/\lie{p}$. Our task is to show that the canonical surjective map $\pi: \lie{s}\rightarrow\bar{\lie{s}} = \bar{\lie{h}} / \bar{\lie{p}}$
is an isomorphism.
Note first that the adjoint action of $H$ on $\lie{h}$ preserves $\lie{p}$ and therefore induces an action of $H$ on $\lie{s}$. Moreover, for $h \in H_{\resunip}$ and $p \in P$
we have $\Ad (p) (\log h) = \log p h p^{-1} = \log [p,h] h \in \log (P h) = \log h + \lie{p}$ by Lemma \ref{lem: untitled}. Therefore, $P$ acts trivially on $\lie{s}$,
and the $H$-action on $\lie{s}$ descends to an action of $S$. Under the map $\pi : \lie{s} \to \bar{\lie{s}}$, this action is compatible with the
canonical action of $\bar H / \bar P \simeq S$ on $\bar{\lie{s}} = \bar{\lie{h}} / \bar{\lie{p}}$.

To study the injectivity of $\pi$, we apply Nori's correspondence (Theorem \ref{TheoremNori}) to the groups $\bar H$ and $\bar P$ and associate to them their algebraic envelopes $\tilde H$ and
$\tilde P$ inside $G_{\F_p}$. By Lemma \ref{LemmaComplementNori}, the group $\tilde P$ is the unipotent radical of $\tilde H$, and the quotient group $\tilde H / \tilde P$ is
therefore a semisimple algebraic group defined over $\F_p$. Let $S^{\SC}$ be the simply connected covering group of the group $\tilde H / \tilde P$.
The quotient $S \simeq \bar{H} / \bar{P}$ is isomorphic to the quotient of $S^{\SC} (\F_p)$ by a central subgroup $K$.
We can therefore regard $\lie{s}$ and $\bar{\lie{s}}$ as representation spaces of $S^{\SC} (\F_p)$.
The action on $\bar{\lie{s}}$ is given by the adjoint representation of $S^{\SC}$ on its Lie algebra $\Lie_{\F_p} S^{\SC}$.

Let $N \supset P$ be a Sylow pro-$p$ subgroup of $H$ and $\lie{n} = \log N \subset \lie{h}$ the associated Lie algebra.
By Lemma \ref{LemmaComplementNori}, the image of $\lie{n}$ in $\bar{\lie{h}}$ is the Lie algebra of a maximal unipotent subgroup of $\tilde{H}$.
Passing to the quotient by $\bar{\lie{p}}$, we obtain that the image of $\lie{n}$ in $\bar{\lie{s}}$ is the Lie algebra of a maximal unipotent subgroup of $S^{\SC}$.
Consider the kernel $\ker \pi = (\lie{p} + p \lie{g}_{\Z_p}) \cap \lie{h} / \lie{p} \subset \lie{s}$ of the map $\pi$.
If an element $\log n + \lie{p}$, where $n \in N$, lies in $\ker \pi$, then $\log n \in \lie{p} + p \lie{g}_{\Z_p}$, which by Lemma \ref{lem: untitled} implies that
$\log (n m) \in p \lie{g}_{\Z_p}$ for a suitable $m \in P$. Consequently, $nm \in H \cap \K_p (p) \subset P$ and therefore $n \in P$.
This means that the restriction of $\pi$ to the subspace $\lie{n} / \lie{p} \subset \lie{s}$ is injective.
Finally, by the very definition of $\lie{h}$, the space $\lie{n} / \lie{p}$ spans $\lie{s}$ under the action of $S^{\SC} (\F_p)$.

The assertion \eqref{step: fromgroup} immediately follows from the following result on characteristic $p$ representations of the group $S^{\SC} (\F_p)$.
Note that $\dim\lie{s}\le \dim G$ in the case at hand.

\begin{proposition} \label{PropAdjoint}
Let $S^{\SC}$ be a simply connected semisimple algebraic group defined over $\F_p$,
$\bar{\lie{s}}=\Lie_{\F_p}S^{\SC}$ with the adjoint representation and
$\pi: \lie{s} \to \bar{\lie{s}}$ a surjection of $\F_p$-representations of $S^{\SC} (\F_p)$. Assume that $p\ge 2\dim\lie{s}$ and that
there exists a subspace $\lie{n} \subset \lie{s}$ such that
\begin{enumerate}
\item The space $\lie{s}$ is spanned as a representation of $S^{\SC} (\F_p)$ by $\lie{n}$.
\item The image $\bar{\lie{n}} = \pi (\lie{n}) \subset \bar{\lie{s}}$ of $\lie{n}$ under $\pi$ is the Lie algebra of a maximal unipotent subgroup of $S^{\SC}$.
\item The induced map $\pi|_{\lie{n}}: \lie{n} \to \bar{\lie{n}}$ is an isomorphism.
\end{enumerate}
Then the map $\pi$ is an isomorphism.
\end{proposition}

\subsection{Proof of Proposition \ref{PropAdjoint}}
To finish the proof of \eqref{step: fromgroup}, it remains to prove Proposition \ref{PropAdjoint}.
For this we need to recall the representation theory of $S^{\SC} (\F_p)$ in characteristic $p$.

Write $S^{\SC} = \prod_{i} \Res_{\F_{q_i} / \F_p} \mathcal{S}_i$, where each $q_i$ is a power of $p$ and $\mathcal{S}_i$ is an absolutely simple, simply connected
algebraic group defined over $\F_{q_i}$.
Set $\mathcal{S} = \prod_i \mathcal{S}_i$, a semisimple, simply connected group defined over $\clos$, and let $\sigma = \prod_i \operatorname{Frob}_{q_i}$,
a Steinberg endomorphism of $\mathcal{S}$ \cite{MR0230728}.
We have then $S^{\SC} (\F_p) = \mathcal{S}^\sigma$.
The $\sigma$-stable Borel subgroups of $\mathcal{S}$ are in one-to-one correspondence with the Borel subgroups of $S^{\SC}$ defined over $\F_p$.
We take the $\sigma$-stable Borel subgroup $\mathcal{B}$ of $\mathcal{S}$ corresponding to the normalizer of the subspace $\bar{\lie{n}} \subset \bar{\lie{s}}$.
We also fix a $\sigma$-stable maximal torus $\mathcal{T}$ of $\mathcal{S}$ contained in $\mathcal{B}$,
and let $\mathcal{U} \subset \mathcal{B}$ be the unipotent radical of $\mathcal{B}$, a maximal unipotent subgroup of $\mathcal{S}$ stable under $\sigma$.

Let $\Phi$ be the root system of $\mathcal{S}$ with respect to $\mathcal{T}$, $\Phi^+ \subset \Phi$ the system of positive roots associated to $\mathcal{B}$,
and $\Delta$ the set of simple roots.
For any root $\alpha\in\Phi$ we set $q(\alpha)=q_i$ if $\alpha$ belongs to the simple factor $\mathcal{S}_i$ in the factorization $\mathcal{S} = \prod_i \mathcal{S}_i$ above.
By \cite[\S11.2, 11.6]{MR0230728}, there exists a permutation $\rho$ of $\Phi$, preserving $\Phi^+$ and $\Delta$, such that
$\sigma^* \rho \alpha = q (\alpha) \alpha$ for all $\alpha \in \Phi$, where $\sigma^*$ denotes the action of $\sigma$ on the weight lattice
$X = X^* (\mathcal{T})$. Let $X^+ \subset X$ the set of dominant weights and
let
\[
X^+_{\sigma} = \{ \lambda \in X : 0 \le \sprod{\lambda}{\alpha^\vee} < q (\alpha), \quad \alpha \in \Delta \} \subset X^+.
\]

For each $\lambda \in X^+$ let $L(\lambda)$ be the irreducible representation of $\mathcal{S}$ of highest weight $\lambda$ with coefficients in $\clos$.
By Steinberg's algebraicity theorem \cite[Theorem 13.3]{MR0230728} the
irreducible representations of $\mathcal{S}^\sigma$ over the field $\clos$ are precisely the restrictions to $\mathcal{S}^\sigma$ of the $\mathcal{S}$-representations $L(\lambda)$ for $\lambda\in X^+_{\sigma}$.

For a $\mathcal{T}$-representation $V$ and a character $\chi \in X$ we write $V^{(\chi)}$ for the $\chi$-eigenspace in $V$, and similarly for $\mathcal{T}^\sigma$-representations and $\clos$-valued characters
of $\mathcal{T}^\sigma$.

\begin{lemma} \label{lem: N}
For any $\lambda \in X^+_{\sigma}$ we have
$L(\lambda)^{\mathcal{U}^\sigma} = L(\lambda)^{\mathcal{U}}= L(\lambda)^{(\lambda)}$, and this space is one-dimensional.
Moreover, any non-trivial $\mathcal{U}^\sigma$-invariant subspace of $L(\lambda)$ contains $L(\lambda)^{\mathcal{U}}$.
\end{lemma}

\begin{proof}
The first equality is contained in \cite[Theorem 4.3]{MR0262383}, and the second equality is standard (cf. \cite[Theorem 5.3]{MR0258838}).
The second part follows from the first one and the well-known fact that any representation of a $p$-group over a field of characteristic $p$
admits a non-trivial vector fixed under the action \cite[\S 8.3, Proposition 26]{MR0450380}.
\end{proof}

\begin{lemma} \label{LemmaTorus}
Let $V$ be a representation of $\mathcal{S}^\sigma$ over $\clos$ with $\dim V < \frac{p+1}{2}$. Then for every $\mu \in X$ with
$V^{(\mu)} \neq 0$ we have $V^{(\mu)} = V^{(\mu|_{\mathcal{T}^\sigma})}$.
\end{lemma}

\begin{proof}
We need to show that for $\mu, \mu' \in X$ with $V^{(\mu)}$, $V^{(\mu')} \neq 0$ the identity $\mu|_{\mathcal{T}^\sigma} = \mu'|_{\mathcal{T}^\sigma}$ implies that $\mu=\mu'$.

For this, we first claim that every $\mu \in X$ with $V^{(\mu)} \neq 0$ satisfies
\[
\abs{\sprod{\mu}{\alpha^\vee}}\le \frac{q(\alpha)-1}2 \text{ for all } \alpha \in \Delta.
\]
Since $V\mapsto V^{(\mu)}$ is an exact functor, we may assume without loss of generality that $V$ is irreducible, say $V=L(\lambda)$.
We may also assume that $\mathcal{S}$ is simple.
Then $q (\alpha) = q = p^r$ for some $r \ge 1$ and we may write
$\lambda=\sum_{i=0}^{r-1}p^i\lambda_i$ with $0 \le \sprod{\lambda_i}{\alpha^\vee} < p$ for all $\alpha \in \Delta$. By Steinberg's tensor product theorem \cite{MR0155937},
the $\mathcal{S}$-representation $L (\lambda)$ is isomorphic to the tensor product of the representations $L (p^i \lambda_i)$
for $i = 0, \ldots, r-1$, and $L (p^i \lambda_i)$ is isomorphic to the $i$-th Frobenius twist of $L (\lambda_i)$.
Therefore $\dim L(\lambda)=\prod_{i=0}^{r-1}\dim L(\lambda_i)$ and in particular $\dim L(\lambda_i)<(p+1)/2$ for all $i$.
By \cite[Lemma 1.2]{MR1635685}, this implies that $\sprod{\lambda_i}{\alpha^\vee}<(p-1)/2$ for all $\alpha\in\Phi$.
Hence, $\abs{\sprod{\mu_i}{\alpha^\vee}}<(p-1)/2$ for any weight $\mu_i$ of $L(\lambda_i)$.
Suppose that $\mu \in X$ with $L (\lambda)^{(\mu)}\simeq\otimes L(p^i \lambda_i)^{(p^i \mu_i)} \neq 0$.
Then $\mu=\sum_{i=0}^{r-1}p^i\mu_i$ with $L(\lambda_i)^{(\mu_i)} \neq 0$ for all $i$ and therefore
\[
\abs{\sprod{\mu}{\alpha^\vee}}\le\sum_{i=0}^{r-1}p^i\abs{\sprod{\mu_i}{\alpha^\vee}}<\frac{p-1}2 \sum_{i=0}^{r-1}p^i=\frac{q-1}2,
\]
which establishes the claim.

For each $\alpha \in \Phi$ let $d (\alpha)$ be the number of elements in the $\rho$-orbit of $\alpha$.
The kernel of the restriction map $\mu\mapsto\mu|_{\mathcal{T}^\sigma}$ is then given by
\[
\{\mu\in X:q (\alpha)^{d (\alpha)}-1\big|\sum_{j=0}^{d (\alpha)-1}\sprod{\mu}{\sigma^j\alpha^\vee}\text{ for all }\alpha\in\Delta\}.
\]

For all $\mu \in X$ with $V^{(\mu)} \neq 0$ we have here
\[
\abs{\sum_{j=0}^{d (\alpha)-1} \sprod{\mu}{\sigma^j\alpha^\vee}} =
\abs{\sum_{j=0}^{d (\alpha)-1} q (\alpha)^j \sprod{\mu}{\rho^j\alpha^\vee}}
<\frac12\sum_{j=0}^{d (\alpha)-1}q (\alpha)^j(q (\alpha)-1)=\frac12(q (\alpha)^{d (\alpha)}-1).
\]
Thus if $\mu,\mu'\in X$ are weights of $V$ with $\mu|_{\mathcal{T}^\sigma} = \mu'|_{\mathcal{T}^\sigma}$, then necessarily
\[
\sum_{j=0}^{d (\alpha)-1} \sprod{\mu}{\sigma^j\alpha^\vee}=\sum_{j=0}^{d (\alpha)-1} \sprod{\mu'}{\sigma^j\alpha^\vee},
\]
which necessitates that $\sprod{\mu}{\sigma^j\alpha^\vee}=\sprod{\mu'}{\sigma^j\alpha^\vee}$ for all $j$.
Thus, $\sprod{\mu}{\alpha^\vee}=\sprod{\mu'}{\alpha^\vee}$ for all $\alpha\in\Delta$, so that $\mu=\mu'$ as required.
\end{proof}

We will now use a rather deep result of Jantzen \cite{MR1635685} (which crucially relies on earlier work by Cline--Parshall--Scott--van den Kallen \cite{MR0439856}).
Namely, every $\clos$-representation of $\mathcal{S}^\sigma$ of dimension less than $p-1$ is completely reducible, and consequently isomorphic to a direct sum of representations
$L (\lambda)$ with $\lambda\in X^+_{\sigma}$.
(See \cite{MR1753813} for an extension of this result. Note that by \cite{MR684821} there exists an $\clos$-representation of $\SL(2,\F_p)$ of dimension $p-1$ that is not completely reducible.)

\begin{lemma} \label{LemmaMultiplicitiesLlambda} Let $V$ be a representation of $\mathcal{S}^\sigma$ over $\clos$ with $\dim V < \frac{p+1}{2}$.
Let $U \subset V$ be a $\mathcal{B}^\sigma$-invariant subspace that spans $V$ as a $\mathcal{S}^\sigma$-representation.
\begin{enumerate}
\item Let $\phi: V \to L (\lambda)$, $\lambda \in X^+_{\sigma}$, be a non-trivial $\mathcal{S}^\sigma$-map. Then
$\phi (U^{(\lambda|_{\mathcal{T}^\sigma})}) = L(\lambda)^{(\lambda|_{\mathcal{T}^\sigma})}$.
\item The multiplicity of $L (\lambda)$, $\lambda \in X^+_{\sigma}$, in $V$ is at most $\dim U^{(\lambda|_{\mathcal{T}^\sigma})}$.
In particular, the restriction to $\mathcal{T}^\sigma$ of the highest weight of any irreducible constituent
of $V$ appears in the $\mathcal{T}^\sigma$-space $U$.
\end{enumerate}
\end{lemma}

\begin{proof}
Consider the first assertion.
Since $\phi (U)$ spans $L(\lambda)$ as an $\mathcal{S}^\sigma$-space, we have $\phi (U) \neq 0$.
Combining Lemmas \ref{lem: N} and \ref{LemmaTorus}, and noting that $\phi(U)$ is $\mathcal{U}^\sigma$-invariant, we have
$\phi(U)^{\mathcal{U}^\sigma}=L(\lambda)^{(\lambda)}=L(\lambda)^{(\lambda|_{\mathcal{T}^\sigma})}$,
which is a one-dimensional space. Thus, $\phi(U)^{(\lambda|_{\mathcal{T}^\sigma})} = L(\lambda)^{(\lambda|_{\mathcal{T}^\sigma})}$.
Since $U$ is $\mathcal{B}^\sigma$-invariant, we also have $\phi(U^{(\lambda|_{\mathcal{T}^\sigma})}) = \phi(U)^{(\lambda|_{\mathcal{T}^\sigma})}$.
The first part follows.

For the second assertion assume on the contrary that the multiplicity $n$ of $L(\lambda)$ in $V$ is strictly greater than $\dim U^{(\lambda|_{\mathcal{T}^\sigma})}$.
By the semisimplicity of $V$ there exists a surjection $\phi: V \to W = L(\lambda)^n$ of $\mathcal{S}^\sigma$-modules. Then $\phi(U)$ spans $W$ as an $\mathcal{S}^\sigma$-space while
$\phi(U)^{(\lambda)}$ is a proper subspace of the $n$-dimensional space $W^{(\lambda)}$.
Now the restriction map from $\Hom_{\mathcal{S}^\sigma} (W, L (\lambda))$ to $\Hom_{\clos} (W^{(\lambda)}, L (\lambda)^{(\lambda)})$ is an isomorphism. Therefore there exists
a surjective $\mathcal{S}^\sigma$-homomorphism $\psi: W \to L (\lambda)$ with $\psi (\phi (U)^{(\lambda)}) = 0$. But then the composition $\psi \circ \phi: V \to L (\lambda)$ contradicts the first assertion,
which finishes the proof.
\end{proof}

For the proof of Proposition \ref{PropAdjoint}
we have to study the adjoint representation of $\mathcal{S}^\sigma = S^{\SC} (\F_p)$ on the Lie algebra $\Lie_{\F_p} (S^{\SC}) \otimes \clos$. It decomposes as the direct sum of the
representations $L(p^i \tilde\alpha)$, where $\tilde\alpha$ ranges over the highest roots of the irreducible components of $\Phi$ and $i$ is such that $p^i | q (\tilde\alpha)$.

We will prove a lemma on the $\mathcal{B}^\sigma$-representations $L(p^i \tilde\alpha)$.
For this we first need the following easy lemma on root systems, a variant of a standard result on sums of roots \cite[\S VI.1.6, Proposition 19]{MR1890629}.

\begin{lemma} \label{LemmaRootSequences}
Let $\Psi$ be a root system, $\Psi^+$ a system of positive roots for $\Psi$, and $\Delta$ the associated set of simple roots.
Let $\beta \in \Psi^+$ and write $\beta = \sum_{\alpha \in \Delta} n_\alpha \alpha$ with non-negative integers $n_\alpha$.
Then the coefficient $n_\alpha$ of a simple root $\alpha \in \Delta$ is positive if and only if there
exists a sequence $\beta_1, \ldots, \beta_k \in \Psi^+$ of positive roots starting with $\beta_1 = \alpha$ and ending with $\beta_k = \beta$ such that
$\beta_{i+1}-\beta_i \in \Psi^+$ for all $i=1,\ldots,k-1$.
\end{lemma}

\begin{proof}
Clearly, the existence of a sequence $\beta_1, \ldots, \beta_k$ as above implies that $n_\alpha > 0$. It remains to show the reverse implication.
For this we proceed by induction on $\sum_{\gamma \in \Delta} n_\gamma$. There has to exist a simple root $\gamma \in \Delta$ with $\sprod{\beta}{\gamma^\vee} > 0$, which
implies that either $\beta = \gamma$ or $\beta - \gamma \in \Psi^+$. In the first case, $\beta = \gamma = \alpha$ and there is nothing to prove.
If in the second case we have $\gamma = \alpha$, then we may simply take $k = 2$. Otherwise, we may apply the induction hypothesis to $\beta - \gamma$ and obtain a sequence
$\beta_1, \ldots, \beta_{k-1} \in \Psi^+$ with $\beta_1 = \alpha$, $\beta_{k-1} = \beta - \gamma$ and $\beta_{i+1}-\beta_i \in \Psi^+$. Setting $\beta_k = \beta$ yields the assertion.
\end{proof}

For any irreducible root system $\Psi$ we denote by $\Psi_{\rm ns}$ the set of all roots which are not contained in the $\Z$-span of the short simple roots.
(By convention, if $\Psi$ is simply laced then all roots are long, so that $\Psi_{\rm ns}=\Psi$.)
If $\Psi^+$ is a system of positive roots, we write $\Psi^+_{\rm ns} = \Psi_{\rm ns} \cap \Psi^+$. A root $\beta = \sum_{\alpha \in \Delta} n_\alpha \alpha \in \Psi^+$ is contained in $\Psi^+_{\rm ns}$
if and only if $n_\alpha > 0$ for at least one long simple root $\alpha \in \Delta \cap \Psi^{\rm long}$.
Note that for any non-zero dominant weight the coefficients with respect to the basis $\Delta$ are all positive,
since this holds for the fundamental weights \cite[\S III.13, exercise 8]{MR499562}.
It follows that the highest short root $\tilde\alpha_{\rm short}$ of $\Psi^+$ belongs to $\Psi^+_{\rm ns}$.

\begin{lemma} \label{LemmaLongRoots}
Assume that $p \ge 2 \dim S^{\SC}$.
Let $\tilde\alpha$ be the highest root of an irreducible component $\Psi$ of $\Phi$ and $p^i$ a divisor of $q(\tilde\alpha)$.
Then in the $\mathcal{S}^\sigma$-representation $L(p^i \tilde\alpha)$,
the $\mathcal{B}^\sigma$-span $D$ of the subspace
\[
\sum_{\alpha \in \Delta \cap \Psi^{\rm long}} L (p^i \tilde\alpha)^{(p^i \alpha)} \subset L (p^i \tilde\alpha)
\]
is the sum of the weight spaces $L (p^i \tilde\alpha)^{(p^i \beta)}$ for $\beta \in \Psi^+_{\rm ns}$.
\end{lemma}

\begin{proof}
By Frobenius twist, we reduce to the case $i=0$. Note that the $\mathcal{S}$-representation
$L (\tilde\alpha)$ is nothing else than the adjoint representation of the simple factor of $\mathcal{S}$ containing $\tilde\alpha$.
For each $\beta \in \Psi$ the corresponding weight space $L (\tilde\alpha)^{(\beta)} = L(\tilde\alpha)^{(\beta|_{\mathcal{T}^\sigma})}$ (where we use Lemma \ref{LemmaTorus}
to identify $\mathcal{T}$- and $\mathcal{T}^\sigma$-eigenspaces) is one-dimensional.

Since it is stable under the action of $\mathcal{T}^\sigma$, the $\mathcal{B}^\sigma$-span of any sum of weight spaces is again a sum of weight spaces. Furthermore, it is clear
that the weights appearing in $D$ are roots in $\Psi^+_{\rm ns}$.
It remains to show that each weight space $L (\tilde\alpha)^{(\beta)}$, $\beta \in \Psi^+_{\rm ns}$ is contained in $D$.

For this consider the following claim: for any $\beta \in \Psi^+$, the $\mathcal{B}^\sigma$-span of $L (\tilde\alpha)^{(\beta)}$
contains $L (\tilde\alpha)^{(\beta')}$ for all $\beta' \in \Psi^+$ for which $\beta' - \beta$ is a positive root.
Granted this claim, we may argue as follows:
by Lemma \ref{LemmaRootSequences}, for every $\beta \in \Psi^+_{\rm ns}$ there exists a sequence $\beta_1, \ldots, \beta_k \in \Psi^+$, where $\beta_1 = \alpha$ is a long simple root,
$\beta_k = \beta$ and $\beta_{i+1}-\beta_i \in \Psi^+$ for all $i=1,\ldots,k-1$. By our claim, for any $i = 1, \ldots, k-1$, the $\mathcal{B}^\sigma$-span of
$L (\tilde\alpha)^{(\beta_i)}$ contains $L (\tilde\alpha)^{(\beta_{i+1})}$. This obviously implies the assertion.

It remains to prove the claim above. For each $\beta \in \Psi$ let $u_\beta \in L (\tilde\alpha)^{(\beta)}$ be a non-trivial element of the one-dimensional weight space.
The action of a root unipotent $x_\gamma (\xi) \in \mathcal{S}$ on $u_\beta$ is given by
\begin{equation} \label{eqn: unipaction}
x_\gamma (\xi) u_\beta = u_\beta + C_{1, \beta\gamma} \xi u_{\beta+\gamma} + \sum_{i \ge 2} C_{i,\beta\gamma} \xi^i u_{\beta+i\gamma}
\end{equation}
with $C_{i,\beta\gamma} \in \clos$ and $C_{1, \beta\gamma} \neq 0$, since $p > 3$ and therefore $p$ does not divide any structure constant of a semisimple Lie algebra.

Let now $\beta, \beta' \in \Psi^+$ with $\gamma = \beta'-\beta \in \Psi^+$ and let $d (\gamma)$ be the cardinality of the $\rho$-orbit of $\gamma$.
It is enough to show that the $\mathcal{U}^\sigma$-span of $u_\beta$ contains an element with non-vanishing projection to the $\beta'$-weight space, since
we can then use the $\mathcal{T}^\sigma$-action to get the desired inclusion.

Note that by \cite[11.2]{MR0230728} the $\sigma$-action on the root unipotents $x_\alpha (\xi)$ is given by
$\sigma x_\alpha (\xi) = x_{\rho\alpha} (c (\alpha) \xi^{q(\alpha)})$ for suitable $c (\alpha) \in \clos^\times$.
For $d (\gamma) = 1$, application of a suitable $x_\gamma (\xi) \in \mathcal{U}^\sigma$ with $\xi \neq 0$ to $u_\beta$ provides an element in the $\mathcal{U}^\sigma$-span of $u_\beta$ with
non-vanishing projection to $L (\tilde\alpha)^{(\beta')}$, as asserted.
In the general case, we can always find a product $x = \prod_{i=1}^n x_{\gamma_i} (\xi_i) \in \mathcal{U}^\sigma$ with $\xi_1 \neq 0$ and positive roots $\gamma_1, \ldots, \gamma_n \in \Psi^+$,
such that $\gamma_1 = \gamma$, and $\sum_{i=1}^n \nu_i \gamma_i = \gamma$ for integers $\nu_i \ge 0$ if and only if $\nu_1 = 1$, $\nu_2 = \dots = \nu_n = 0$.
By repeated application of \eqref{eqn: unipaction} we obtain that $x u_\beta$ has non-vanishing projection to $L (\tilde\alpha)^{(\beta')}$, which finishes the proof.
\end{proof}

\begin{proof}[Proof of Proposition \ref{PropAdjoint}]
By tensoring the spaces $\lie{s}$, $\bar{\lie{s}}$, $\lie{n}$ and $\bar{\lie{n}}$ with $\clos$, we may pass to representations of $\mathcal{S}^\sigma$ and $\mathcal{B}^\sigma$ over $\clos$,
which for the remainder of this proof we will denote by the same letters.
(Note that $\lie{n}$ is $\mathcal{B}^\sigma$-invariant since $\bar{\lie{n}}$ is.)
The same applies to the surjection $\pi: \lie{s} \to \bar{\lie{s}}$.
Obviously, to establish the proposition it is equivalent to show that $\pi$ is an isomorphism in the new setting.

As noted above, $\bar{\lie{s}} \simeq \Lie_{\F_p} (S^{\SC}) \otimes \clos$ decomposes as the direct sum of the
representations $L(p^i \tilde\alpha)$, where $\tilde\alpha$ ranges over the highest roots of the irreducible components of $\Phi$ and $p^i | q (\tilde\alpha)$.
The weights of $\mathcal{T}^\sigma$ on $\lie{n} \simeq \bar{\lie{n}}$ are given by $p^i \alpha$, where $\alpha \in \Phi^+$ and $p^i$ divides $q (\alpha)$,
and they all have multiplicity one.

Applying Lemma \ref{LemmaMultiplicitiesLlambda}
to $U=\lie{n} \subset V=\lie{s}$, we obtain that
the possible highest weights of the irreducible constituents of $\lie{s}$ are $p^i \tilde\alpha$, $p^i | q (\tilde\alpha)$, and $p^i \tilde\alpha_{\rm short}$,
$p^i | q (\tilde\alpha_{\rm short})$, where $\tilde\alpha$ ranges over the highest roots of
the irreducible components of $\Phi$ and $\tilde\alpha_{\rm short}$ over the highest short roots of the components that are not simply laced.
Moreover, all irreducible constituents of $\lie{s}$ appear with multiplicity one.

It only remains to show that no representation $L(p^i \tilde\alpha_{\rm short})$ can be a quotient of $\lie{s}$.
Denote by $\lie{l}$ the (unique) lift of the representation $\bar{\lie{s}}$ to a subspace of $\lie{s}$.
Note that for any long root $\beta$ the corresponding weight space $\lie{s}^{(p^i \beta)}$ is necessarily contained in $\lie{l}$
since the non-zero weights in $L(\tilde\alpha_{\rm short})$ are the short roots.

Let $\tilde\alpha_{\rm short}$ be the highest short root of a component $\Psi$ of $\Phi$ that is not simply laced. For any long simple root $\alpha$ of $\Psi$
the weight space $\lie{n}^{(p^i \alpha)}$ is contained in $\lie{l}$. By Lemma \ref{LemmaLongRoots}, the weight space
$\lie{n}^{(p^i \tilde\alpha_{\rm short})} \simeq \bar{\lie{n}}^{(p^i \tilde\alpha_{\rm short})}$ lies in the $\mathcal{B}^\sigma$-span of the spaces $\lie{n}^{(p^i \alpha)} \simeq \bar{\lie{n}}^{(p^i \alpha)}$,
$\alpha \in \Delta \cap \Psi^{\rm long}$. Therefore
$\lie{n}^{(p^i \tilde\alpha_{\rm short})}$ is contained in $\lie{l}$.
By the first part of Lemma \ref{LemmaMultiplicitiesLlambda}, it follows that the representation $\lie{s} / \lie{l}$ does not admit $L(p^i \tilde\alpha_{\rm short})$ as a quotient.

This shows that $\lie{s} = \lie{l}$, and that $\pi$ is an isomorphism, as asserted.
\end{proof}

\subsection{From subalgebras to subgroups}
For the proof of the remaining identity \eqref{step: fromalgebra}, we need the following easy consequence of the standard presentations
of the groups $S^{\SC} (\F_p)$, where $S^{\SC}$ is a simply connected semisimple group defined over $\F_p$.

\begin{lemma} \label{LemmaPresentation}
Let $S^{\SC}$ be a simply connected semisimple group defined over $\F_p$.
Let $\Gamma = \Gamma (S^{\SC})$ be the group defined by generators $\gamma_u$, $u\in S^{\SC} (\F_p)_{\unip}$, and relations
\begin{enumerate}
\item (restricted multiplication) $\gamma_{u_1} \gamma_{u_2} = \gamma_{u_1u_2}$ for any unipotent subgroup $\tilde U$ of $S^{\SC}$ defined over $\F_p$ and $u_1,u_2\in\tilde U(\F_p)$,
\item (conjugation) $\gamma_{u_1} \gamma_{u_2} \gamma_{u_1}^{-1} = \gamma_{u_1u_2u_1^{-1}}$ for any $u_1,u_2\in S^{\SC} (\F_p)_{\unip}$.
\end{enumerate}
Then the map $\gamma_u \mapsto u$, $u\in S^{\SC} (\F_p)_{\unip}$, extends to an isomorphism of groups
$s:\Gamma\rightarrow S^{\SC} (\F_p)$.
\end{lemma}

\begin{proof}
It is clear that the map $\gamma_u \mapsto u$ extends to a group homomorphism $s$, which is surjective since $S^{\SC} (\F_p)$ is generated by its unipotent elements.
It remains to show that $s$ is injective. For this, we may use any presentation $(X,R)$ of $S^{\SC} (\F_p)$ with generators $X \subset S^{\SC} (\F_p)_{\unip}$ and relations $R$.
We only need to verify that the elements $\gamma_u$, $u\in X$ generate $\Gamma$ and satisfy the relations $R$ in $\Gamma$, i.e. that $R$ is a consequence of the relations above.

We may assume that $S^{\SC}$ is simple over $\F_p$. For $S^{\SC}$ of $\F_p$-rank one the assertion follows (at least for $p > 7$, when the Schur multiplier of $S^{\SC} (\F_p)$ is trivial)
from \cite[Theorem 1.2]{MR0188299}. Alternatively, it follows from
\cite[Th\'{e}or\`{e}me 3.3]{MR0153677} and \cite[Theorem 5.1]{MR630615} without restriction on $p$.
Namely, we consider a maximal unipotent subgroup $U \subset S^{\SC}$ and the opposite unipotent subgroup $\bar U$.
By the Bruhat decomposition, every unipotent element of $S^{\SC} (\F_p)$ is either contained in $U (\F_p)$ or of the form $u \bar{u} u^{-1}$ for $u \in U (\F_p)$ and $\bar u \in \bar U (\F_p)$, and therefore
$\Gamma$ is generated by the $\gamma_u$'s for $u \in U (\F_p) \cup \bar U (\F_p)$. Furthermore, the relations of \cite[Theorem 1.2]{MR0188299} follow from the conjugation relations in $\Gamma$.

In rank $>1$ the assertion follows from the Curtis--Steinberg--Tits presentation [ibid., Theorem 1.4] (at least for $p \ge 5$, when the Schur multiplier of $S^{\SC} (\F_p)$ is trivial,
cf. also \cite{MR0153677, MR630615, MR637803}).
Fix a maximal split torus $T$ of $S^{\SC}$ and a Borel subgroup $B$ containing $T$, and let $U$ be the unipotent radical of $B$, which is a maximal unipotent subgroup of $S^{\SC}$.
Then every element of $U (\F_p)$ is a product of root unipotents $u$ with respect to the roots of $T$ on $U$.
The Bruhat decomposition and the above relations imply that
$\Gamma$ is generated by the elements $\gamma_{u}$, where $u$ ranges over the elements of the unipotent root subgroups with respect to $T$.
Moreover, the Curtis--Steinberg--Tits relations certainly follow from the relations between the $\gamma_u$'s. This finishes the proof of the lemma.
\end{proof}

We now show \eqref{step: fromalgebra}.
Let $\lie{h}\in\LSAN$ and let $\lie{p}$ be the largest ideal of $\lie{h}$ contained in $\lie{h}_{\resnilp}$.
Let $H=\grpc(\lie{h})$ and $P=\exp\lie{p}$, and let as usual $\bar H$ and $\bar P$ be the images of these groups in $G (\F_p)$. It is clear that $P$ is a normal pro-$p$ subgroup of $H$.
Note that $\lie{p}$ is the inverse image under reduction modulo $p$ of the largest ideal $\bar{\lie{p}}$ of $\bar{\lie{h}}$ contained in the set $\bar{\lie{h}}_{\nilp}$.
In particular, $\lie{p}\supset\lie{h}\cap p\lie{g}_{\Z_p}$ and therefore
$P\supset\exp(\lie{h}\cap p\lie{g}_{\Z_p})$.

It remains to prove that $H\cap\K_p(p)\subset P$, since in this case
\[
H\cap\K_p(p)=P\cap\K_p(p)=\exp(\log (P\cap\K_p(p)))=\exp(\lie{p}\cap p\lie{g}_{\Z_p})=
\exp(\lie{h}\cap p\lie{g}_{\Z_p}).
\]

Let $\tilde H \supset \tilde P$ be the algebraic subgroups of $G_{\F_p}$ associated by Nori's correspondence to the nilpotently generated Lie subalgebras $\bar{\lie{h}} \supset \bar{\lie{p}}$
of $\lie{g}_{\F_p}$. The corresponding subgroups of $G (\F_p)$ are $\bar H$ and $\bar P$, respectively.
By Lemma \ref{LemmaComplementNori}, $\tilde P$ is the unipotent radical of $\tilde H$. Let $S^{\SC}$ be the simply connected covering group of
the semisimple group $\tilde H / \tilde P$
and let $K$ be the kernel of the surjective covering map $\kappa: S^{\SC} (\F_p)\rightarrow \tilde H (\F_p)^+ / \tilde P (\F_p) = \bar H/\bar P $. Note that
$\kappa$ restricts to a bijection between $S^{\SC} (\F_p)_{\unip}$ and the image of $\bar H_{\unip}$ in $\bar H / \bar P$.
We may assume that $p$ does not divide the size of the center of $S^{\SC}$, and in particular the size of $K$.

We need to show that the canonical surjective homomorphism $r: H/P \rightarrow \bar H / \bar P$ is an isomorphism of groups.
Using Lemma \ref{LemmaPresentation},
we will show this by constructing a surjective homomorphism $\phi: \Gamma = \Gamma (S^{\SC}) \to H / P$ with $r \circ \phi= \kappa \circ s$.
Note that in any case the kernel of $r$ equals $(H \cap P\K_p(p))/P$, and that it is therefore a $p$-group.

For any residually nilpotent subalgebra $\lie{u} \subset \lie{h}$ containing $\lie{p}$, the group $U = \exp \lie{u}$ is a pro-$p$ subgroup of $H$ containing $P$
and $r$ restricts to an isomorphism $r|_{U/P}: U/P \rightarrow \bar U / \bar P$, where $\bar U \subset \bar H$ is the image of $U$ in $G (\F_p)$. Given any $p$-group $\bar U \subset \bar H$, we can lift
the associated Lie algebra
$\bar{\lie{u}} = \bar{\Liec} (\bar{U}) = \log^{(p)} \bar U \subset \bar{\lie{h}} = \bar{\Liec} (\bar{H})$ to its inverse image $\lie{u}$ under the reduction map $\lie{h} \to \bar{\lie{h}}$ and obtain a
pro-$p$ group $U \subset H$ containing $P$ which projects onto $\bar U$ under reduction modulo $p$ and for which $r|_{U/P}: U/P \rightarrow \bar U / \bar P$ is an isomorphism.

As a consequence, for any unipotent subgroup $\tilde N$ of $S^{\SC}$ defined over $\F_p$, we may consider the preimage
$\tilde U$ of $\kappa(\tilde N)$ in the group $\tilde H$ and lift the $p$-group $\tilde U (\F_p) \subset \bar H$
to a pro-$p$ subgroup $U \subset H$ such that $r$ maps $U/P$ isomorphically onto $\tilde N (\F_p) \simeq \tilde U (\F_p) / \bar P \subset \bar H / \bar P$.

In particular, we may lift any unipotent element $u$ of $S^{\SC} (\F_p)$ to an element $f(u) \in H/P$ with $r(f(u)) = \kappa (u)$, namely to
$f(u) := (\exp v) P$ for some $v \in \lie{h}_{\resnilp}$ such that $\exp v$ maps under reduction modulo $p$ to a representative of $\kappa (u) \in \bar H / \bar P$ in $\bar H$.
This lift does not depend on the choice of $v$: for any $v_1$, $v_2 \in \lie{h}_{\resnilp}$ it follows
from Lemma \ref{lem: untitled} that $\exp v_1$ and $\exp v_2$ map to the same element of $\bar H / \bar P$ if and only if $v_1 - v_2 \in \lie{p} + \lie{h} \cap p \lie{g}_{\Z_p} = \lie{p}$,
which is in turn equivalent to $\exp v_1 \in (\exp v_2) P$. Therefore, we obtain a well-defined map $f: S^{\SC} (\F_p)_{\unip} \rightarrow H / P$ with $r \circ f = \kappa|_{S^{\SC} (\F_p)_{\unip}}$.
Moreover, for any unipotent $\F_p$-subgroup $\tilde N \subset S^{\SC}$ the restriction of $f$ to $\tilde N (\F_p)$ is the inverse of the group isomorphism $r|_{U/P}: U/P \to \tilde N (\F_p)$ constructed above,
and in particular it is a group homomorphism.

We now claim that in the setting of Lemma \ref{LemmaPresentation} the definition
$\phi (\gamma_u) := f (u)$, $u \in S^{\SC} (\F_p)_{\unip}$, extends to a group homomorphism $\phi: \Gamma \rightarrow H/P$, i.e.,
that it respects the relations defining $\Gamma$.

Clearly,
\begin{equation} \label{EqnCommutativityPresentation}
r (\phi (\gamma_u)) = r (f(u)) = \kappa (u) \text{ for all } u \in S^{\SC} (\F_p)_{\unip}.
\end{equation}
For a unipotent subgroup $\tilde N \subset S^{\SC}$ defined over $\F_p$ we know
that the restriction $f|_{{\tilde N} (\F_p)}$ is a group homomorphism, which means that
$\phi(\gamma_{u_1u_2})= f(u_1 u_2) = f (u_1) f (u_2) = \phi(\gamma_{u_1})\phi(\gamma_{u_2})$ for any $u_1,u_2\in\tilde N(\F_p)$.

Let us check the conjugation relations. Suppose that Lie algebra elements $v_i\in\lie{h}_{\resnilp}$ reduce modulo $p$ to $\log^{(p)}u'_i$, where $u'_i \in {\bar H}_{\unip}$, $i = 1,2$,
and that $u'_i$ maps to $\kappa (u_i) \in \bar H / \bar P$ for some $u_i \in S^{\SC} (\F_p)_{\unip}$. Then $\phi (\gamma_{u_i}) = f(u_i) = (\exp v_i) P$.
Now $\Ad(\exp v_1)v_2 \in \lie{h}_{\resnilp}$ maps under reduction modulo $p$ to $\Ad(u'_1)\log^{(p)}(u'_2)=\log^{(p)}(u'_1u'_2 (u'_1)^{-1})$, and
$u'_1 u'_2 (u'_1)^{-1}$ maps to $\kappa (u_1 u_2 u_1^{-1})$ in $\bar H / \bar P$. Hence,
\[
\phi(\gamma_{u_1u_2u_1^{-1}})=\exp(\Ad(\exp v_1)v_2)P=\exp v_1\exp v_2\exp(-v_1)P=\phi(\gamma_{u_1})\phi(\gamma_{u_2})\phi(\gamma_{u_1})^{-1},
\]
which finishes the proof of the claim.

It is clear from the definition of $H$ that the homomorphism $\phi$ is surjective.
Moreover, it is clear from \eqref{EqnCommutativityPresentation} that $r \circ \phi= \kappa \circ s$. Since $s$ is an isomorphism by Lemma \ref{LemmaPresentation},
we conclude that the order of $\Ker r$ divides the order of $K$. Since $\Ker r$ is also a $p$-group, we conclude that $r$ is an isomorphism.

This finishes the proof of \eqref{step: fromalgebra}.

\section{Intersections of conjugacy classes and open compact subgroups} \label{MainApplication}

\subsection{The global bound}
As an application of the approximation theorem, we prove in this section an estimate for the volume of the intersection of a conjugacy class
in the group of $\hat{\Z}$-points of a reductive group defined over $\Q$ with an arbitrary open subgroup.
This is a key technical result in our approach to the limit multiplicity problem for arbitrary congruence subgroups of an arithmetic group,
which is the subject matter of \cite{1504.04795}.
It is convenient to formulate the result in a slightly more general way, namely in terms of the commutator map.
Theorem \ref{thm: mainbnd} below gives the most general formulation, and Corollary \ref{CorConjugacyClass} the main
application. Corollary \ref{CorLattices} is a variant of this result in the language of lattices in semisimple Lie groups.
At the end of this section we also prove some auxiliary results that
will be applied to the limit multiplicity problem.

We will temporarily consider more general groups $G$ than before. (However, in Theorem \ref{thm: mainbnd} and its proof $G$ will be assumed to be semisimple and simply connected.)
For the following definitions let $G$ be a (possibly non-connected) reductive group defined over $\Q$ such that its adjoint group $G^{\ad}$ as well as its derived group $G^{\der}$ are connected.
(The connectedness of $G^{\der}$ is not essential, and we assume it only for simplicity.)
We fix an embedding $\rho_0: G \to \GL (N_0)$.
We denote the canonical action of $G^{\ad}$ on $G$ by $\ad$ and the adjoint representation of $G^{\ad}$ on $\lie{g} = \Lie_{\Q} G$ by $\Ad$.
We have the commutator map
\[
[\cdot,\cdot]: \quad G^{\rm ad} \times G^{\rm ad} \to G^{\der},
\]
which is a morphism of algebraic varieties over $\Q$.

For every prime $p$ we set $\K_p = \rho_0^{-1} (\GL (N_0, \Z_p))$
as before, and let $\K = \prod_p \K_p$, an open compact subgroup of $G (\A_{\rm fin})$.
Let
\[
\K (N) = \{g \in \K \, : \, \rho_0 (g) \equiv 1 \pmod{N} \} = \prod_p \K_p (p^{v_p (N)}), \quad \quad N \ge 1,
\]
be the principal congruence subgroups of $\K$. They are normal open subgroups of $\K$ and form a neighborhood base of the identity element.
For an open subgroup $K \subset \K$ (or, more generally, an open compact subgroup $K \subset G (\A_{\rm fin})$)
let its \emph{level} $\level (K) = \prod_p \level_p (K)$ be the smallest integer $N$ for which $\K (N) \subset K$.
We note that a change of the representation $\rho_0$ changes the level of a given open compact subgroup $K$
only by a factor that is bounded from above and below.

Sometimes we will need a generalization of this notion. If $H \subset G (\A_{\fin})$ is an arbitrary closed subgroup, then we let $\level (K; H) = \prod_p \level_p (K; H)$
be the smallest integer $N$ for which $\K (N) \cap H \subset K$. In particular, this applies to the group $H = G (\A_{\fin})^+$, the
image of $G^{\rm sc} (\A_{\fin})$ under the natural homomorphism $p^{\rm sc}: G^{\rm sc} \to G$, where $G^{\rm sc}$ is the simply connected covering group of $G^{\der}$.

We also fix an embedding $\rho_0^{\ad}: G^{\ad} \to \GL (N_0^{\ad})$ and let $\K_p^{\ad} = (\rho_0^{\ad})^{-1} (\GL (N_0^{\ad}, \Z_p)) \subset G^{\ad} (\Q_p)$ and
$\K^{\ad} = \prod_p \K_p^{\ad} \subset G^{\ad} (\A_{\rm fin})$. Let $\K_p^{\ad} (p^n) \subset \K_p^{\ad}$ be the principal congruence subgroups with respect to $\rho_0^{\ad}$.
We normalize the Haar measures on the groups $\K_p^{\rm ad}$ and on their product $\K^{\rm ad}$ to have volume one.

\begin{definition}
Let $\tilde{\K}$ be a compact subgroup of $G^{\ad} (\A_{\rm fin})$ with its normalized Haar measure.
For an open subgroup $K \subset G (\A_{\rm fin})$ define
\[
\phi_{K, \tilde{\K}}(x) = \vol \left( \{ k \in \tilde{\K} \, : \, [k,x] \in K \} \right), \quad \quad x \in G^{\ad} (\A_{\rm fin}).
\]
For $\tilde{\K} = \K^{\ad}$ we simply write $\phi_K (x) = \phi_{K, \K^{\ad}} (x)$.
Analogously, we write
\[
\phi_{K_p, \tilde{\K}_p} (x_p) = \vol \left( \{ k \in \tilde{\K}_p \, : \, [k,x_p] \in K_p \} \right), \quad \quad x_p \in G^{\ad} (\Q_p),
\]
for open subgroups $K_p \subset G (\Q_p)$ and compact subgroups $\tilde{\K}_p \subset G^{\ad} (\Q_p)$, and set $\phi_{K_p} (x_p) = \phi_{K_p, \K_p^{\ad}} (x)$.
\end{definition}

We will estimate the function $\phi_K (x)$ (indeed $\phi_{K,\tilde{\K}} (x)$ for certain subgroups $\tilde{\K} \subset \K^{\ad}$) for $x \in \K^{\ad}$
and arbitrary open subgroups $K \subset \K$ in terms of the level of $K$.
To state the dependence of the bound on $x$ in a convenient manner, we introduce the following notation.
Fix a $\Z$-lattice $\Lambda$ in $\lie{g}$ such that $\Lambda \otimes \hat{\Z}$ is $\K^{\ad}$-stable.

\begin{definition}
For $x_p \in G^{\ad} (\Q_p)$ set
\[
\lambda^G_p (x_p) = \max \{ n \in \Z \cup \{ \infty \} \, : \, (\Ad (x_p) - 1) \Pro_{\lie{h}} (\Lambda \otimes \Z_p) \subset p^n (\Lambda \otimes \Z_p)
\text{ for some $\lie{h} \neq 0$ } \},
\]
where $\lie{h}$ ranges over the non-trivial semisimple ideals
of the Lie algebra $\lie{g} \otimes \Q_p$, and $\Pro_{\lie{h}}$ denotes the corresponding projection $\lie{g} \otimes \Q_p \to \lie{h} \subset \lie{g} \otimes \Q_p$.
\end{definition}

See Remark \ref{RemarkRestrictionOfScalars} below for an alternative, somewhat more concrete expression for this function.
Note that $\lambda^G_p (x_p) \ge 0$ for $x_p \in \K^{\ad}_p$.
A change of the lattice $\Lambda$ will change $\lambda^G_p (x_p)$ at most by a constant $c_p$, where $c_p = 0$
for almost all $p$. Similarly, for any compact set $\Omega \subset G^{\ad} (\A_{\rm fin})$
there exist constants $A_p$, with $A_p=0$ for almost all $p$, such that
$\abs{\lambda^G_p(y^{-1}xy)-\lambda^G_p(x)}\le A_p$
for all primes $p$, $x\in G^{\ad} (\A_{\rm fin})$, $y\in \Omega$.

\begin{theorem} \label{thm: mainbnd}
Let $G$ be semisimple and simply connected.
Let $\tilde{\K} = \prod_p \tilde{\K}_p$, where for each $p$ the group $\tilde{\K}_p$ is an open
subgroup of $\K_p^{\ad}$ and the indices $[\K_p^{\ad} : \tilde{\K}_p]$ are bounded.
Then there exist constants $\varepsilon, \delta > 0$, depending only on $G$, $\rho_0$, $\rho_0^{\ad}$ and $\tilde{\K}$, such that
for all open subgroups $K$ of $\K$ of level $N = \level (K) = \prod_p p^{n_p}$
and all $x \in \K^{\ad}$ we have
\begin{equation} \label{EqnMainBound}
\phi_{K,\tilde{\K}} (x)\ll
\beta^G (N, x, \delta)^{-\varepsilon} \quad \text{with} \quad \beta^G (N, x, \delta) =
\prod_{p|N, \, \lambda^G_p (x) < \delta n_p} p^{n_p}.
\end{equation}
Here, the implied constant depends on $G$, $\rho_0$, $\rho_0^{\ad}$, $\tilde{\K}$ and $\Lambda$.
\end{theorem}

\begin{remark} \label{RemarkSinglex}
Let
\[
\K^{\ad}_{\reg}:=\{x\in\K^{\ad}:\lambda_p (x) < \infty\text{ for all $p$ and $\lambda_p (x) = 0$ for almost all $p$}\},
\]
which is a dense subset of $\K^{\ad}$.
For $x\in\K^{\ad}_{\reg}$ we obtain from Theorem
\ref{thm: mainbnd} the estimate
\[
\phi_{K,\tilde{\K}} (x) \ll_{\tilde{\K}}\big(\prod_pp^{\lambda_p(x)}\big)^{\varepsilon/\delta}\level (K)^{-\varepsilon}.
\]
By Lemma \ref{LemmaLambdapFiniteness} below, $x\in\K^{\ad}_{\reg}$ whenever $x$ is an element of $G^{\ad} (\Q) \cap \K^{\ad}$ that is not contained in any proper normal subgroup of $G^{\ad}$
(which we may of course assume to be defined over $\Q$).
\end{remark}

\begin{remark} \label{RemarkBetterBound}
In fact, the proof of Theorem \ref{thm: mainbnd} will yield the sharper estimate
\begin{equation} \label{EqnBetterEstimate}
\phi_{K,\tilde{\K}} (x)\le
\prod_{p|N, \, \lambda^G_p (x) = 0} \min (1, \frac{C}{p}) \prod_{p|N} \min (1, p^{\varepsilon' (c + \lambda^G_p (x) - \varepsilon n_p)}),
\end{equation}
where $\varepsilon > 0$ is the constant provided by the approximation theorem (Theorem \ref{TheoremAlgebraic}) for $G$ and $\rho_0$, and $\varepsilon', C >0$ and $c$ are additional
constants depending only on $G$, $\rho_0$, $\rho_0^{\ad}$ and $\tilde{\K}$.
\end{remark}

\begin{remark} \label{rem: slexam}
For square-free level $N$ and $x \in \K^{\ad}_{\reg}$
we obtain from \eqref{EqnBetterEstimate} the estimate $\phi_{K,\tilde{\K}} (x) \ll_{x, \tilde{\K}, \varepsilon} N^{-\varepsilon}$ for any $\varepsilon < 1$.
Up to the determination of the constant $C$, the estimate \eqref{EqnBetterEstimate} is best possible
even for groups of arbitrarily large rank. (It might not be best possible for particular groups $G$.)
Indeed, if $G=\SL (n)$ over $\Q$, $G^{\ad}=\operatorname{PGL}(n)$, $\K_p^{\ad}=\operatorname{PGL}(n,\Z_p)$ and $K_p \subset \K_p = \SL (n, \Z_p)$ is the stabilizer of a point
$\bar{\xi} \in \mathbb{P}^{n-1} (\F_p)$ under the natural action, then
\[
\phi_{K_p} (x_p) = \frac{1}{p^n-1} \sum_{i=1}^r \left( p^{n_i} - 1 \right),
\]
if $x_p \in \GL (n, \Z_p)$ stabilizes $\bar{\xi}$ and the dimensions of the $\F_p$-eigenspaces of the
image $\bar{x}_p$ of $x_p$ in $\GL (n, \F_p)$ are $n_1$, \ldots, $n_r$.
In particular, for $n_1=n-1$ and $n_2=1$ we get $\phi_{K_p} (x_p) \ge \frac1{p}$.
Note that for $\Lambda=\mathfrak{sl}(n,\Z)$, we have $\lambda_p (x_p) = 0$ if $\bar{x}_p$ is not a scalar matrix.
\end{remark}

\begin{remark}
For an arbitrary $N$, it is not possible to take $\varepsilon$ arbitrarily close to $1$ in \eqref{EqnMainBound}.
For example, consider for $G = \SL (2)$ the congruence subgroups $K_p = \Gamma_0 (p^n) \subset \K_p = \SL (2, \Z_p)$, $n \ge 1$. Let $x_p =
\left(\begin{smallmatrix}{a}&{b}\\{0}&{d}\end{smallmatrix}\right) \in \SL (2, \Z_p)$ and $r = \lambda_p (x_p) = \min (v_p (d-a), v_p (b))$.
Then $\phi_{K_p} (x_p)$ is just the number of fixed points of $x_p$ on $\mathbb{P}^1 (\Z / p^n \Z)$ divided by
$\card{\mathbb{P}^1 (\Z / p^n \Z)} = p^n (1 + \frac1{p})$.
Assume that $r < n$. Write $x_p = a + p^r y$ with $y = \left(\begin{smallmatrix}{0}&{b'}\\{0}&{d'}\end{smallmatrix}\right)$ and
$b'$, $d' \in \Z_p$ not both divisible by $p$. Then $\phi_{K_p} (x_p)$ is also
the number of eigenvectors of $y$ in $\mathbb{P}^1 (\Z / p^{n-r} \Z)$ divided by $p^{n-r} (1 + \frac1{p})$. After possibly conjugating by
an upper triangular matrix in $\SL (2, \Z_p)$, we may assume that $b'=1$. Then the number of eigenvectors in question is just the number of solutions
to the quadratic congruence $\xi^2 - d' \xi \equiv 0 \pmod{p^{n-r}}$, where $v_p (d') = v_p (d-a) - r$.
Direct computation yields the result
\begin{equation} \label{eqn: SL2example}
\phi_{K_p} (x_p) = \begin{cases} 2 \left(1+ \frac1{p} \right)^{-1} p^{-(n-v_p (d-a))}, & v_p (d-a) < \frac{n+r}{2}, \\
\left(1+ \frac1{p} \right)^{-1} p^{-\lceil \frac{n-r}2 \rceil}, & v_p (d-a) \ge \frac{n+r}{2}.
\end{cases}
\end{equation}
This implies that $\varepsilon\le\frac12$ for $G = \SL (2)$.
\end{remark}

The following corollary concerning conjugacy classes in $G (\hat{\Z})$ for an arbitrary reductive group $G$
has applications to the limit multiplicity problem.
Essentially the same result has been obtained independently in \cite[\S 5]{ABBGNRS} without using the approximation theorem.
We note that in the case of $G = \SL (2)$ or the group of norm one elements of a quaternion algebra over $\Q$,
and for particular choices of $x$, very explicit estimates of this type have been already obtained in \cite{MR749678}.

\begin{corollary} \label{CorConjugacyClass}
Let $G$ be a (possibly non-connected) reductive group defined over $\Q$ whose derived group $G^{\rm der}$ and adjoint group $G^{\ad}$ are connected.
Let $\rho_0: G \to \GL (N_0)$ be a faithful $\Q$-rational representation and $\K = \rho_0^{-1} (\GL (N_0, \hat{\Z})) \subset G (\A_{\rm fin})$.
Then there exist constants $\varepsilon, \delta>0$ such that
for all open compact subgroups $\K_0 \subset G (\A_{\rm fin})$
we have
\[
\vol \left( \{ k \in \K \, : \, k x k^{-1} \in K \} \right) \ll_{\rho_0, \, \K_0}\beta^G (\level ( K; G (\A_{\fin})^{+} ) , x, \delta)^{-\varepsilon}
\]
for all open subgroups $K \subset \K_0$ and all $x \in G (\A_{\rm fin})$. (Here we pull back $\lambda_p$ to $G$ via the projection $G\rightarrow G^{\ad}$.)
\end{corollary}

The simple proof will be given in \S \ref{SubsectionIntersections} below.

For the convenience of the reader, and
to facilitate comparison with \cite{ABBGNRS}, we also give a variant of this result concerning lattices in semisimple Lie groups.
For an arbitrary group $\Gamma$, a finite index subgroup $\Delta$, and an element $\gamma \in \Gamma$ set
\[
c_\Delta (\gamma) =  \abs{\{ \delta \in \Gamma / \Delta \, : \, \delta^{-1} \gamma \delta \in \Delta\}},
\]
which is also the number of fixed points of $\gamma$ in the permutation representation of $\Gamma$ on the finite set $\Gamma / \Delta$.

\begin{corollary} \label{CorLattices}
Let $G$ be a semisimple and simply connected group defined over $\Q$ such that for no $\Q$-simple factor $H$ of $G$ the group $H (\R)$ is compact,
and let $\K \subset G (\A_{\fin})$ be as above.
Let $\Gamma = G (\Q) \cap \K$, which is a lattice in the Lie group $G (\R)$.
For any open subgroup $K \subset \K$ let $\Delta = G (\Q) \cap K$ be the associated finite index subgroup of $\Gamma$. Then
there exist constants $\varepsilon$, $\delta > 0$, depending only on $G$ and $\rho_0$, such that
for all open subgroups $K$ and all $\gamma \in \Gamma$ we have
\[
\frac{c_\Delta (\gamma)}{[\Gamma : \Delta]} \ll \beta^G (\level (K), \gamma, \delta)^{-\varepsilon}.
\]
In particular, if $\gamma$ is not contained in any proper normal subgroup of $G$ (which we may assume to be defined over $\Q$) then we have
\[
\frac{c_\Delta (\gamma)}{[\Gamma : \Delta]} \ll
\big(\prod_pp^{\lambda_p(\gamma)}\big)^C\,[\Gamma : \Delta]^{-\varepsilon},
\]
where $C>0$ depends only on $G$.
\end{corollary}

\begin{proof}
The assumptions on $G$ imply (and in fact, are equivalent to) that $G$ has the strong approximation property, i.e., that $G (\Q)$ is dense in $G (\A_{\fin})$ \cite[Theorem 7.12]{MR1278263}.
From this we get that $[\Gamma : \Delta]^{-1} c_\Delta (\gamma) =
\vol \left( \{ k \in \K \, : \, k \gamma k^{-1} \in K \} \right)$, and we may apply Corollary \ref{CorConjugacyClass} to deduce the first inequality.
For the second, we invoke Remark \ref{RemarkSinglex} to get the estimate $\ll (\prod_pp^{\lambda_p(\gamma)})^C\,\level (K)^{-\varepsilon}$.
However, we have $[\Gamma : \Delta] = [\K : K ] \le \vol (\K (\level (K)))^{-1} \ll \level (K)^{\dim G}$,
and we may therefore replace $\level (K)$ by $[\Gamma : \Delta]$ (at the price of replacing $\varepsilon$ by $\varepsilon / \dim G$).
\end{proof}

The groups $\Delta = G (\Q) \cap K \subset \Gamma$ are called the congruence subgroups of $\Gamma$. Under the assumptions of the
corollary, they are in one-to-one correspondence with the open subgroups $K$ of $\K$.

\subsection{Reduction to two statements on open subgroups of $\K_p$}
We will now derive Theorem \ref{thm: mainbnd} (and the refined estimate \eqref{EqnBetterEstimate}) from the
following two statements concerning open subgroups of the groups $\K_p$ for arbitrary $p$.
As in the theorem, $G$ will be assumed to be semisimple and simply connected. (The assumption that $G$ is simply connected is relevant only for Proposition \ref{PropositionModp} below.)
The first statement bounds $\phi_{K_p} (x_p)$ for all proper subgroups $K_p$ of $\K_p$, while the second one concerns subgroups of level $p^n$ for large $n$.
The propositions will be proved in \S \ref{SubsectionLocalBounds} below.

\begin{proposition} \label{PropositionModp}
For any prime $p$, any proper subgroup $K_p$ of $\K_p$, and any $x_p \in \K^{\rm ad}_p$ with $\lambda_p (x_p) = 0$, we have $\phi_{K_p} (x_p) \ll_{\rho_0,\rho_0^{\ad}} p^{-1}$.
\end{proposition}

\begin{proposition} \label{PropositionHighLevel}
There exist constants $\varepsilon$, $\varepsilon' > 0$ and $c$, depending only on $G$, $\rho_0$ and $\rho_0^{\ad}$, such that for any prime $p$,
any subgroup $K_p$ of $\K_p$ of level $p^n$ and any $x_p \in \K^{\rm ad}_p$, we have $\phi_{K_p} (x_p) \le p^{\varepsilon' (c + \lambda_p(x_p) - \varepsilon n)}$.
\end{proposition}

As in Remark \ref{RemarkBetterBound}, the constant $\varepsilon$ in Proposition \ref{PropositionHighLevel} is the constant provided by Theorem \ref{TheoremAlgebraic} applied to $G$ and $\rho_0$.

In fact Proposition \ref{PropositionModp} is essentially already known.
More precisely, at least in the case where $x_p$ lies in the image of $K_p$ in $G^{\ad} (\Q_p)$, it is shown in \cite{MR1114511} that for any $G$ one can take the implied constant
(say $C$) to be $2$
for almost all $p$.\footnote{Lemma \ref{LemmaAlmostAllp} below shows that the framework of [loc.~cit.] is applicable for almost all $p$.}
This estimate is optimal
for $G = \SL (2)$, as one sees from \eqref{eqn: SL2example} for $n = 1$ and $r=v_p (d-a) = 0$.
If one excludes the case where $G$ has a factor of type $A_1$, then one may lower the value of $C$ further.
(By Remark \ref{rem: slexam} above, we need to have $C \ge 1$ even if we omit finitely many possibilities for $G$.)
However, the proof in [loc.~cit.] is based on a detailed analysis of a great number of particular cases, and uses explicit information on the maximal subgroups
of the finite simple groups of Lie type, while our proof, which does not give the optimal value of $C$, uses only Nori's theorem.
For more refined recent bounds
see \cite{MR1639620,MR2301233,MR2301234,MR2344583,MR2344584} (concerning classical groups) and \cite{MR1922740} (concerning exceptional groups), as well as the references cited therein.

We will now show how to deduce Theorem \ref{thm: mainbnd} from Propositions \ref{PropositionModp} and \ref{PropositionHighLevel}.
Arguments of this type can be found already in \cite{MR749678} and \cite[\S 6.1]{MR1978431}.
We first state a simple property of the functions $\phi_{K,\tilde{\K}}$ which follows by straightforward calculation.

\begin{lemma} \label{lem: cosets}
Let $L \subset K$ be open subgroups of $\K$.
Then
\[
\phi_{K,\tilde{\K}} (x) = \sum_{\eta \in L \bs K: \, [\tilde{\K}, x] \cap L \eta \neq \emptyset}
\phi_{\ad (k_\eta)^{-1} (L),\tilde{\K}} (x), \quad \quad x \in \K^{\rm ad},
\]
where $k_\eta \in \tilde{\K}$ is an arbitrary element with $[k_\eta,x] \in L \eta$.
\end{lemma}

In the following lemma we collect some standard facts on the behavior of the groups $G$ and $G^{\ad}$ modulo $p$ for almost all primes $p$.

\begin{lemma} \label{LemmaAlmostAllp}
Given $G$, $\rho_0$, $G^{\ad}$, $\rho^{\ad}_0$, for almost all primes $p$ we have:
\begin{enumerate}
\item The group schemes $G$ and $G^{\ad}$ are smooth over $\Z_p$ and thus $\K_p / \K_p (p) \simeq G (\F_p)$ and
$\K^{\ad}_p / \K^{\ad}_p (p) \simeq G^{\ad} (\F_p)$. Moreover, the group schemes $G_{\F_p}$ and $G^{\ad}_{\F_p}$ are semisimple algebraic groups over $\F_p$.
\item The maps $\ad: G^{\ad} \times G \to G$ and $[\cdot,\cdot]: G^{\ad} \times G^{\ad} \to G$ map
$\K^{\ad}_p \times \K_p$ and $\K^{\ad}_p \times \K^{\ad}_p$ to $\K_p$, and moreover descend modulo $p$ to
corresponding maps of the groups of $\F_p$-points.
\item Let $\lie{h}_i$, $i = 1$, \ldots, $r_p$, be the minimal non-zero ideals of $\lie{g} \otimes \Q_p$.
Then the $\Z$-lattice $\Lambda \subset \lie{g}$ used to define $\lambda_p (x)$ satisfies
$\Lambda \otimes \Z_p = \bigoplus_{i=1}^{r_p} (\Lambda \otimes \Z_p) \cap \lie{h}_i$.
The corresponding factorizations $G=G_1\times\dots\times G_{r_p}$
and $G^{\ad} = G_1^{\rm ad} \times \dots\times G_{r_p}^{\ad}$ of $G$ and $G^{\ad}$ as products of
$\Q_p$-simple algebraic groups extend to factorizations of
group schemes that are smooth over $\Z_p$. In particular, we have
corresponding factorizations of $G$ and $G^{\ad}$ over $\F_p$ (cf. \cite[pp. 392--393]{MR1978431}).
\item All proper normal subgroups of the groups $G_i (\F_p)$, $i = 1, \dots, r_p$, are central
\cite[Proposition 7.5]{MR1278263}, and the
center of $G(\F_p)$ is the set of $\F_p$-points of the center of $G$, i.e., the kernel of the adjoint representation of $G (\F_p)$.
\item \label{part: 1645} For any $m \ge 1$ the $m$-th iterated Frattini subgroup of $\K_p$ is equal to $\K_p (p^m)$. In particular,
the Frattini subgroup $\Phi (\K_p)$ is $\K_p (p)$  \cite[Window 9, Lemma 5 and Corollary 6]{MR1978431}.
\end{enumerate}
\end{lemma}

We also need a standard estimate for the number of $\F_p$-points of a linear algebraic group \cite[Lemma 3.5]{MR880952}.
\begin{lemma} \label{LemmaPointCounting}
For almost all $p$, depending on $G^{\ad}$ and $\rho_0^{\ad}$, we have
\[
(p-1)^{\dim G} \le | G^{\ad} (\F_p) | = \left( \vol \K_p^{\ad} (p) \right)^{-1} \le (p+1)^{\dim G}.
\]
\end{lemma}

For convenience, we also isolate a key technical consequence of Proposition \ref{PropositionHighLevel} as a separate lemma.
\begin{lemma} \label{LemmaTechnicalPhi}
Let $\varepsilon$ and $\varepsilon'$ be as in Proposition \ref{PropositionHighLevel}.
Let $\tilde{\K}_p \subset \K_p^{\ad}$ be open subgroups and $B$ a positive integer with $[\K_p^{\ad} : \tilde{\K}_p] \le B$ for all $p$.
There exists a constant
$c$, depending only on $G$, $\rho_0$, and $B$, such that for any prime $p$, for all subgroups $K_p \subset \K_p (p)$
of level $p^n$ and for all $x_p \in \K_p^{\ad}$, we have the estimate
\[
\phi_{\ad (k_p) K_p, \tilde{\K}_p} (x_p) \le \min (1, p^{\varepsilon' (c + \lambda_p(x_p) - \varepsilon n)}) \phi_{\ad (k_p) \K (p), \tilde{\K}_p} (x_p).
\]
\end{lemma}

\begin{proof}
By the second and fifth parts of Lemma \ref{LemmaAlmostAllp},
for all but finitely many primes $p$ the operators $\ad (k_p)$, $k_p \in \K_p^{\ad}$, act on $\K_p$, and
each of the groups $\K_p (p^m)$ is a characteristic subgroup of $\K_p$.
We therefore have $\ad (k_p) \K_p (p^m) = \K_p (p^m)$ as well as
$\level (\ad (k_p) K_p) = \level (K_p) = p^n$ for almost all $p$. Treating the remaining finitely many primes $p$
one at a time, an easy compactness argument shows that $\level (\ad (k_p) K_p) \ge p^{n - c_1}$ for all $p$,
with a constant $c_1$ depending only on $G$ and $\rho_0$.
Again by compactness, for each single $p$ the values $\phi_{\ad (k_p) (\K_p (p)), \tilde{\K}_p} (x_p)$ for $x_p, k_p \in \K_p^{\rm ad}$ are bounded away from zero.
Moreover, by Lemma \ref{LemmaPointCounting}, for almost all $p$ we have
$\phi_{\ad (k_p) \K_p (p), \tilde{\K}_p} (x_p) = \phi_{\K_p (p), \tilde{\K}_p} (x_p) \ge \vol (\K^{\ad}_p (p))) \ge p^{-c_2}$, with a suitable constant $c_2$, and therefore
$\phi_{\ad (k_p) \K_p (p), \tilde{\K}_p} (x_p) \ge p^{-c_3}$ for all $p$ with a suitable $c_3$.
Since $\phi_{\ad (k_p) K_p, \tilde{\K}_p} (x_p) \le B \phi_{\ad (k_p) K_p} (x_p)$,
the lemma follows now from Proposition \ref{PropositionHighLevel}.
\end{proof}

\begin{proof}[Proof of Theorem \ref{thm: mainbnd}]
Let $\tilde{\K} \subset \K^{\ad}$ be as in Theorem \ref{thm: mainbnd}.
Let $K \subset \K$ be an arbitrary open subgroup and write $N = \level (K) = \prod_p p^{n_p}$.
Set $N_1 = \prod_{p | N} p$ and consider the groups
$\tilde{K} = K \K (N_1)$ and $L = K \cap \K (N_1)$. Clearly, $L$ is a normal subgroup of $K$
of level $N$ and
\begin{equation} \label{eq: cosetcorrespondence}
L \bs K \simeq \K (N_1) \bs \tilde{K}.
\end{equation}
Note that we can factor $L = \prod_{p|N} L_p \prod_{p \not| N} \K_p$, where $L_p \subset \K_p (p)$ is a pro-$p$ group.

We can now apply Lemma \ref{lem: cosets} to $L \subset K$ to obtain
\begin{equation} \label{eq: phiK}
\phi_{K,\tilde{\K}} (x)  = \sum_{\eta \in L \bs K: \, [\tilde{\K}, x] \cap L \eta \neq \emptyset}
\phi_{\ad (k_\eta)^{-1} (L), \tilde{\K}} (x).
\end{equation}
By \eqref{eq: cosetcorrespondence}, we may
choose the same representatives $\eta \in K$ and $k_\eta \in \tilde{\K}$ also for the pair of groups
$\K (N_1) \subset \tilde{K}$. We obtain the corresponding equation
\begin{equation} \label{eq: phitildeK}
\phi_{\tilde{K},\tilde{\K}} (x)  = \sum_{\eta}\phi_{\ad (k_\eta)^{-1} (\K (N_1)),\tilde{\K}} (x).
\end{equation}

Consider $\phi_{\ad (k_\eta)^{-1} L, \tilde{\K}} (x) = \prod_{p | N} \phi_{\ad (k_{\eta,p})^{-1} L_p, \tilde{\K}_p} (x_p)$.
Applying Lemma \ref{LemmaTechnicalPhi} yields
\[
\phi_{\ad (k_\eta)^{-1} L, \tilde{\K}} (x) \le \phi_{\ad (k_\eta)^{-1} \K (N_1), \tilde{\K}} (x) \prod_{p|N} \min (1, p^{\varepsilon' (c + \lambda_p(x_p) - \varepsilon n_p)}).
\]
Inserting this into \eqref{eq: phiK} and using \eqref{eq: phitildeK}, we get
\begin{equation} \label{eq: phiKestimate}
\phi_{K,\tilde{\K}} (x) \le \phi_{\tilde{K},\tilde{\K}} (x) \prod_{p|N} \min (1, p^{\varepsilon' (c + \lambda_p(x_p) - \varepsilon n_p)}).
\end{equation}
It remains to estimate $\phi_{\tilde{K}, \tilde{\K}} (x)$.
There exists a constant $C \ge 1$, depending only on $G$ and $\rho_0$, such that for all primes $p | N$ with $p \ge C$ the image of $K$ in
the factor group $\K_p (p) \bs \K_p$ is a proper subgroup of this group \cite[p.~116]{MR1978431}. Therefore,
\[
\tilde{K} \subset \prod_p \tilde{K}_p,
\]
where for the primes $p | N$ with $p \ge C$ the group $\tilde{K}_p$ is a proper subgroup of $\K_p$
and $\tilde{K}_p = \K_p$ for all other primes $p$.
For a suitable value of $C$ we can therefore apply Proposition \ref{PropositionModp} to estimate
\[
\phi_{\tilde{K}, \tilde{\K}} (x) \le \prod_{p | N: \lambda_p (x_p) = 0} \min (1, \frac{C}{p}).
\]
Combining this with \eqref{eq: phiKestimate} yields immediately \eqref{EqnBetterEstimate}.

It is now a routine matter to derive \eqref{EqnMainBound} from \eqref{EqnBetterEstimate}.
For this assume without loss of generality that $c \ge 0$ and let $0 < \delta < (c+1)^{-1} \varepsilon$.
Observe first that we can estimate
\[
\prod_{p | N: \lambda_p (x_p) = 0} \min (1, \frac{C}{p})
\le \prod_{p | N: p \ge C^2, \, \lambda_p (x_p) = 0} p^{-\frac12} \\
\le C_2 \prod_{p | N: \lambda_p (x_p) = 0} p^{-\frac12},
\]
where we set $C_2 = \prod_{p < C^2} p^{\frac{1}{2}}$. Therefore, we obtain
\[
\phi_{\tilde{K}, \tilde{\K}} (x) \le C_2 \prod_{p | N: \lambda_p (x_p) = 0} p^{-\frac12}
\prod_{p|N: \lambda_p (x) < \delta n_p} \min (1, p^{\varepsilon' (c + \lambda_p(x_p) - \varepsilon n_p)}).
\]
Consider now any prime $p|N$ for which $\lambda_p (x) < \delta n_p$.
In the case $n_p \le \delta^{-1}$, the inequality $\lambda_p (x) < \delta n_p$ implies that $\lambda_p (x) = 0$ and the first product contains therefore the factor $p^{-\frac12} \le p^{-\frac{\delta}2 n_p}$.
In case $n_p > \delta^{-1}$, the second product contains the factor
\[
p^{\varepsilon' (c + \lambda_p(x_p) - \varepsilon n_p)} \le p^{\varepsilon' ((c+1) \delta - \varepsilon) n_p} = p^{-\varepsilon'' n_p}
\]
with $\varepsilon'' = \varepsilon' (\varepsilon - (c+1) \delta) > 0$.
This clearly implies \eqref{EqnMainBound}.
\end{proof}

\subsection{Proof of the local bounds} \label{SubsectionLocalBounds}
We now prove Propositions \ref{PropositionModp} and \ref{PropositionHighLevel}.
For this we use the general estimates of Appendix \ref{sec: numbersol} for the number of solutions of polynomial congruences.
The first proposition is easily deduced from Proposition \ref{PropositionApproxLevelp}.

\begin{proof}[Proof of Proposition \ref{PropositionModp}]
We can assume that $p$ is sufficiently large (depending on $G$). In particular, we can assume that we are in the situation
of Lemma \ref{LemmaAlmostAllp}.

Since under this assumption on $p$ we have $\Phi (\K_p) = \K_p (p)$,
the projection $H$ of the proper subgroup $K_p$ of $\K_p$ to $G (\F_p)$
is a proper subgroup. Replacing $K_p$ by the preimage of $H$
and taking into account the first two parts of Lemma \ref{LemmaAlmostAllp}, we are reduced to proving that
\begin{equation} \label{EqnCommutatorModp}
\frac{\card{\{ k \in G^{\rm ad} (\F_p) \, : \, [k,\bar{x}] \in H \}}}{\card{G^{\rm ad} (\F_p)}} \le \frac{C}{p}
\end{equation}
for a suitable constant $C$, where $\bar{x}$ denotes the image of $x_p \in \K^{\ad}_p$ in $G^{\ad} (\F_p)$.

Recall that by the algebraization result of Proposition \ref{PropositionApproxLevelp},
there exist an integer $N$, depending only on $G$, and
a proper, connected algebraic subgroup $X$ of $G_{\F_p}$, defined over $\F_p$,
such that $[H:H\cap X(\F_p)]\le N$. Moreover, the ideal defining the subvariety $X$ of $G_{\F_p}$ is
generated by regular functions of degree $\le N$ (with respect to the affine embedding of $G_{\F_p}$ provided by $\rho_0$).
Given this, it follows easily from Lemma \ref{lem: cosets} that it is enough to establish the estimate \eqref{EqnCommutatorModp} for the group $X (\F_p)$ instead of $H$.

The condition $\lambda_p(x_p)=0$ implies that the group
\[
\bar{\mathcal{C}} (\bar{x}) := \left\langle [G^{\ad} (\F_p), \bar{x}] \right\rangle \subset G (\F_p)
\]
is the full group $G (\F_p)$.
Indeed, by the identity $[yx,z]= (\ad (y) [x,z])[y,z]$ the group $\bar{\mathcal{C}} (\bar{x})$ is a normal subgroup of $G (\F_p)$ for any $\bar{x}$.
It is also clearly the product of its projections $\bar{\mathcal{C}}_i (\bar{x})$ to the factors $G_i (\F_p)$, $i = 1, \ldots, r_p$.
Moreover, by our assumption on $x_p$ the projections of $\bar{x} \in G^{\ad} (\F_p)$
to the factors $G_i^{\ad} (\F_p)$ are all non-trivial.
Therefore the normal subgroups $\bar{\mathcal{C}}_i (\bar{x}) \subset G_i (\F_p)$ are all non-central. By the fourth part of Lemma \ref{LemmaAlmostAllp} they therefore
have to be the full factor groups $G_i (\F_p)$, and we obtain that $\bar{\mathcal{C}} (\bar{x}) = G (\F_p)$.

Thus we have $[G^{\rm ad} (\F_p),\bar{x}]\not\subset X (\F_p)$. Therefore,
in any generating set of the defining ideal of $X$, there exists an element $f$ (a regular function on $G_{\F_p}$)
such that $g = f([\cdot,\bar{x}])$ (a regular function on $G^{\ad}_{\F_p}$) does not vanish on $G^{\rm ad} (\F_p)$.
Since the degree of $f$ can be bounded by $N$, the degree of the function $g$ is
clearly also bounded in terms of $G$ and $N$. Since we have
\[
\card{\{ k \in G^{\rm ad} (\F_p) \, : \, [k,\bar{x}] \in X (\F_p) \}} \le \card{\{ k \in G^{\rm ad} (\F_p) \, : \, g (k) = 0 \}},
\]
and moreover
\[
\card{G^{\rm ad} (\F_p)} \ge c p^{\dim G}
\]
for a suitable constant $c > 0$ by Lemma \ref{LemmaPointCounting},
it remains to invoke Lemma \ref{LemmaVolumeBoundModp} to establish \eqref{EqnCommutatorModp} for $X (\F_p)$ and to finish the proof.
\end{proof}

We now turn to the proof of Proposition \ref{PropositionHighLevel},
which is based on the combination of Theorem \ref{TheoremAlgebraic} with Lemma \ref{VolumeBound}.

We first clarify the group-theoretic meaning of $\lambda_p(x)$ following Larsen and Lubotzky.

\begin{definition} \label{DefLambdapx}
For $x_p \in \K_p^{\ad}$ let
\[
\mathcal{C}_p (x_p) := \overline{\left\langle [\K_p^{\ad} (p), x_p] \right\rangle} \subset G (\Q_p).
\]
\end{definition}
Note that the group $\mathcal{C}_p (x_p)$ is invariant under $\ad (\K_p^{\ad} (p))$. (For almost all $p$ it is by Lemma \ref{LemmaAlmostAllp} therefore a normal subgroup of $\K_p (p)$.)

\begin{lemma} \label{LemmaLambdapx}
There exists a constant $n_0$, depending only on $G$, such that for all $x_p \in \K^{\ad}_p$ with $\lambda_p (x_p) < \infty$ we have
\[
\mathcal{C}_p (x_p) \supset \K_p (p^{\lambda_p (x_p)+n_0}).
\]
\end{lemma}

\begin{proof}
The assertion is essentially proven in \cite[p.~453--454]{MR2094120}, where the case of $\Q_p$-simple groups $G$ is treated.
Since $G_{\Q_p}$ can be factored as a product of $\Q_p$-simple groups, we can easily reduce to this case.
Although it is not explicitly stated there, the proof in [ibid.] also shows that $n_0$ may be chosen independently of $p$.
\end{proof}

In the proof of Proposition \ref{PropositionHighLevel} it turns out to be convenient to consider
a variant of the principal congruence subgroups which is provided by the following definition.

\begin{definition} \label{DefKpRhovm}
Let
\[
\rho: G \to \GL (N)
\]
be a $\Q_p$-rational representation of $G$ with $\rho (\K_p) \subset \GL (N, \Z_p)$,
and assume that no non-zero vector in $\Q_p^N$ is fixed by $\rho(G)$.
For any primitive $v \in \Z_p^N$ (i.e., $v\notin p \Z_p^N$) and any $m \ge 1$ set
\[
\K_p (\rho, v; m) = \{ g \in \K_p \, : \, \exists \lambda \in \Z_p^*: \, \rho (g) v \equiv \lambda v \pmod{p^m} \}.
\]
\end{definition}

Clearly, the groups $\K_p (\rho, v; m)$ are open subgroups of $\K_p$.
We can estimate their levels as follows.

\begin{lemma} \label{LevelBound}
There exist constants $m_0$ and $m_1$, depending only on $G$ and $N$, such that
\begin{equation} \label{EqnLevelKprho}
p^{m+m_1} \ge \level (\K_p (\rho, v; m)) \ge p^{m-m_0}
\end{equation}
for all $\rho$ as in Definition \ref{DefKpRhovm}, all primitive $v \in \Z_p^N$ and all $m \ge 1$.
Moreover, for all $p$ outside a finite set of primes depending only on $G$ and $N$, we have $\level (\K_p (\rho, v; m)) = p^m$.
\end{lemma}

\begin{proof}
We first observe that there exists $e_p \ge 1$ such that
\[
\rho (\K_p (p^{e_p})) \subset \Gamma (N, p^{\epsilon_p}),
\]
where $\Gamma (N, p^m)$, $m\ge1$, are the principal congruence subgroups of $\GL (N, \Z_p)$.
Indeed, the reduction modulo $p$ of the representation $\rho$ induces a homomorphism $\K_p (p^{\epsilon_p}) \to \GL (N, \F_p)$,
whose image is contained in a $p$-Sylow subgroup of $\GL (N, \F_p)$.
Recall that
\begin{equation} \label{EquationKpnPowers}
\K_p (p^n) = \K_p (p^{\epsilon_p})^{\{p^{n-\epsilon_p}\}}
\end{equation}
for all $n \ge \epsilon_p$. Hence we can take $e_p = \epsilon_p + \lceil \log_p N \rceil$ if $p$ is odd,
while for $p=2$ we consider the reduction modulo $4$ and set $e_2 = 2 + \nu$, where $2^\nu$ is the exponent of
the $2$-Sylow subgroup of the finite group $\GL (N, \Z / 4 \Z)$.

Using \eqref{EquationKpnPowers} again, for any $m \ge \epsilon_p$ we have
\[
\rho (\K_p (p^{m+e_p-\epsilon_p}))
= \rho (\K_p (p^{e_p})^{\{p^{m-\epsilon_p}\}}) \subset \Gamma (N, p^{\epsilon_p})^{\{p^{m-\epsilon_p}\}} = \Gamma (N, p^m).
\]
Consequently, $\K_p (p^{m+e_p-\epsilon_p} ) \subset \K_p (\rho, v ; m)$ for all $v$. This shows that
$\level (\K_p (\rho, v; m)) \le p^{m+m_{1,p}}$ for $m_{1,p} = e_p-\epsilon_p$.

Moreover, $\rho (\K_p (p^{e_p}))$ is a uniform subgroup of
the uniform pro-$p$ group $\Gamma (N,p^{\epsilon_p})$, and by Proposition \ref{prop: unif}
we have therefore $\rho (\K_p (p^{e_p})) = \exp (L_{\rho, e_p})$ for
a powerful Lie subalgebra $L_{\rho, e_p} \subset p^{\epsilon_p} \gl (N, \Z_p)$.
We claim that there exists $f_p \ge \epsilon_p$ with the property:
\begin{equation} \label{EqnStatementfp}
\text{For all primitive $v \in \Z_p^N$ there exists $l \in L_{\rho,e_p}$ with $l v \notin \Z_p v + p^{f_p+1} \Z_p^N$.}
\end{equation}
For assume the contrary. Because the set of primitive elements of $\Z_p^N$ is compact, there would then exist
a primitive $v \in \Z_p^N$ with $L_{\rho,e_p} v \subset \Z_p v$, and therefore
$\rho (\K_p (p^{e_p}))$ would stabilize the line $\Q_p v \subset \Q_p^N$.
Since $\K_p (p^{e_p})$ is Zariski dense in $G$ \cite[Lemma 3.2]{MR1278263},
this would imply that the representation $\rho$ stabilizes a line in $\Q_p^N$, and therefore contains the trivial representation (since $G$ is semisimple),
contrary to our assumption. This proves the claim.

Assume now that $n \ge e_p$ and $m \ge \epsilon_p$ are such that
\[
\K_p (p^n) \subset \K_p (\rho,v;m).
\]
Since $\rho (\K_p (p^n)) = \exp (p^{n-e_p} L_{\rho, e_p})$, we can rewrite this relation as
\[
\exp (p^{n-e_p} L_{\rho, e_p}) v \subset \Z_p v + p^m \Z_p^N.
\]
Let $l \in L_{\rho, e_p}$ be arbitrary and let $\xi = \exp (p^{n-e_p} l) \in \Gamma (N, p^{\epsilon_p})$.
Since $\xi v \in \Z_p v + p^m \Z_p^N$ we can write $\xi\in \xi'\Gamma(N,p^m)$ for some $\xi' \in \Gamma (N, p^{\epsilon_p})$ with $\xi' v \in \Z_p v$.
Therefore, by part \ref{part: expcongruence1} of Proposition \ref{prop: unif} we have
$p^{n-e_p} l = \log \xi \in \log \xi' +p^m\gl (N, \Z_p)$ and $(\log \xi') v \in \Z_p v$, which together implies
that $p^{n-e_p} lv \in \Z_p v + p^m \Z_p^N$, and therefore $lv\in\Z_p v+p^{m-n+e_p}\Z_p^N$, since $v$ is primitive.
On the other hand, by \eqref{EqnStatementfp} there exists $l \in L_{\rho,e_p}$ such that $l v \notin \Z_p v + p^{f_p+1} \Z_p^N$.
We infer that $n \ge m + e_p - f_p$. We conclude that $\level (\K_p (\rho,v;m)) \ge p^{m-m_{0,p}}$ for all $m \ge \epsilon_p$ with $m_{0,p} = f_p - e_p$.

It remains to see that for almost all $p$ we can choose $e_p=f_p=1$.
Consider the set $\K_{p,\resuni}$ of residually unipotent elements of $\K_p$.
Note that the proof of \cite[Window 9, Lemma 5]{MR1978431} shows that for almost all $p$, the group $\K_p (p)$ is topologically generated by the elements
$u^p$, $u\in\K_{p,\resuni}$.
Clearly, if $u\in\K_{p,\resuni}$ then its image $\rho (u) \in \GL (N, \Z_p)$ is also residually unipotent,
and for $p \ge N$ this implies that in fact $\rho (u)^p \equiv 1 \pmod{p}$. Therefore, after excluding a
finite set of primes that depends only on $G$ and $N$, we have $\rho (\K_p (p)) \subset \Gamma (N, p)$,
so that we can take $e_p = 1$.

Moreover, for all such $p$ the reduction modulo $p$ of the representation $\rho$ gives a representation $\bar{\rho}: G (\F_p) \to \GL (N, \F_p)$.
On the other hand, the representation $\rho$ decomposes over $\bar{\Q}_p$ into irreducibles, which are parametrized by their highest weights.
For almost all $p$, the root coordinates of these weights are small with respect to $p$, which implies that the reduction modulo $p$ of each irreducible constituent of $\rho$
remains irreducible \cite[Corollary II.5.6]{MR899071}.
By the Brauer--Nesbitt theorem, the semisimplification of the representation $\bar{\rho}$ (considered over $\clos$) is given by the direct sum of the reductions modulo $p$
of the irreducible constituents of $\rho$.
Since by assumption the trivial representation is not a constituent of $\rho$, we conclude that if we exclude a finite set of primes $p$ that depends only on $G$ and $N$,
then the characteristic $p$ representation $\bar{\rho}$ does not contain the trivial representation.

Therefore, for any $0\ne\bar{v}\in\F_p^N$ the line $\F_p \bar{v} \subset \F_p^N$ is not stabilized by the operators
$\overline{\rho (u)}$, $u\in\K_{p,\resuni}$, and hence by the operators $\log^{(p)}\overline{\rho (u)}$, $u\in\K_{p,\resuni}$ \cite[Lemma 1.4]{MR880952}.
By the commutativity of \eqref{eq: logcomm}, we conclude that the line $\F_p \bar{v}$ is not stabilized by
the reduction mod $p$ of the logarithms $\log \rho (u)$, $u\in\K_{p,\resuni}$. However, the Lie algebra $L_{\rho,1}$ contains
the elements $p \log \rho (u)$, $u\in\K_{p,\resuni}$. This means that we may take $f_p = 1$ above and conclude that $\level (\K_p (\rho, v; m)) = p^m$.
This proves the lemma.
\end{proof}

Next, we reformulate Lemma \ref{VolumeBound} in the case at hand.
Recall that $\rho_0^{\ad}$ fixes an affine embedding of the affine variety $G^{\ad}$, which
allows us to speak of the degree of regular function on $G^{\ad}$ (and similarly for $G$).

\begin{corollary} \label{CorollaryVolumeBound}
For any integer $d > 0$ there exists a constant $\varepsilon (d)>0$, depending only on $G^{\ad}$ and $\rho_0^{\ad}$, such that
for any $m,n\in\Z$, and a regular function $f$ on $G^{\ad}$, defined over $\Q_p$, of degree $\le d$ such that $f (\K^{\ad}_p (p)) \not\subset p^m \Z_p$, we have
\[
\vol \left( \{ g \in \K^{\ad}_p (p): f (g) \equiv 0 \pmod{p^n} \} \right) \ll_{d, G^{\ad}} p^{-\varepsilon (d) (n-m+1)}.
\]
\end{corollary}

\begin{proof}
We take $V = G^{\ad}$ in Lemma \ref{VolumeBound} together with the affine embedding provided by $\rho_0^{\ad}$.
Let $\mu_p$ be the normalized Haar measure on $G^{\ad} (\Q_p)$ such that $\vol(\K^{\ad}_p)=1$
and $\tilde{\mu}_p$ the measure on $G^{\ad} (\Q_p)$ induced by the fixed affine embedding of $G^{\ad}$ (cf. \cite[\S 3.3]{MR644559}, \cite[\S 3]{MR656627}).
Since $G^{\ad}$ is a smooth variety over $\Q_p$, for any $p$ we clearly have
\[
c_p \tilde{\mu}_p|_{\K_p (p)} \le \mu_p|_{\K^{\ad}_p (p)} \le c'_p \tilde{\mu}_p|_{\K^{\ad}_p (p)}
\]
for suitable $c'_p \ge c_p > 0$. Moreover, for almost all $p$ the measure
$\tilde{\mu}_p|_{\K^{\ad}_p (p)}$ is a constant multiple of the measure $\mu_p$, and since it gives $\K^{\ad}_p (p)$ total volume
$p^{-\dim G}$, while $\mu_p (\K^{\ad}_p (p)) = | G^{\ad} (\F_p) |^{-1}$ for almost all $p$, we have
\[
\mu_p|_{\K^{\ad}_p (p)} = \frac{p^{\dim G}}{| G^{\ad} (\F_p) |} \tilde{\mu}_p|_{\K^{\ad}_p (p)}
\]
for almost all $p$.
By Lemma \ref{LemmaPointCounting}, the normalizing factor satisfies
\[
\left(1+\frac1p\right)^{-\dim G} \le \frac{p^{\dim G}}{| G^{\ad} (\F_p) |} \le \left(1-\frac1p\right)^{-\dim G},
\]
and is therefore bounded in both directions in terms of $\dim G$ only. The lemma follows therefore from Lemma \ref{VolumeBound}.
\end{proof}

\begin{remark}
Assume that $G$ is split over $\Q$. Then we can realize $\K^{\ad}_p (p)$ as an explicit compact subset of an affine space of dimension $\dim G$ and
by Lemma \ref{LemmaPolynomialCongruence} it is therefore
possible to take any $\varepsilon (d) < \frac1d$ in Corollary \ref{CorollaryVolumeBound}.
In general, the affine variety $G^{\ad}$ does not need to be rational over $\Q$ (or even over $\Q_p$).
\end{remark}

We can now prove a variant of Proposition \ref{PropositionHighLevel} for the groups $K_p = \K_p (\rho, v; m)$.
\begin{lemma} \label{LemmaKprho}
There exist constants $\varepsilon' > 0$ and $c' \ge 0$, depending only on $G$ and $N$, such that
for any prime $p$, any $\Q_p$-rational representation $\rho: G \to \GL (N)$ with $\rho (\K_p) \subset \GL (N, \Z_p)$ without fixed vectors, any
primitive $v \in \Z_p^N$, $m \ge 1$, and any $x_p \in \K^{\ad}_p$ we have
\[
\vol \left( \{ k \in \K^{\ad} (p) \, : \, [k,x_p] \in \K_p (\rho, v; m) \} \right) \le p^{\varepsilon' (c' + \lambda_p(x_p) - m)}.
\]
\end{lemma}

\begin{proof}
Let $\sprod{\cdot}{\cdot}$ be the standard bilinear form on $\Q_p^N$.
Write $v = \sum_{i=1}^N v_i e_i$ and
consider the regular functions
\[
f_{ij,v} (g) = \sprod{\rho ([g,x_p]) v}{v_j e_i - v_i e_j}, \quad i \neq j,
\]
on $G^{\ad}$.
It follows directly from Definition \ref{DefKpRhovm} that for any primitive $v \in \Z_p^N$ , $k \in \K^{\ad}_p$ and $m \ge 1$ we have
\[
f_{ij,v} (k) \in p^m \Z_p \ \ \forall i\neq j \quad \text{ if and only if } \quad [k,x_p] \in \K_p (\rho, v; m),
\]
and therefore (see Definition \ref{DefLambdapx})
\[
f_{ij,v} (\K^{\ad}_p (p)) \subset p^m \Z_p \ \ \forall i\neq j \quad \text{ if and only if } \quad \mathcal{C}_p (x_p) \subset \K_p (\rho, v; m).
\]
Combining Lemma \ref{LemmaLambdapx} with Lemma \ref{LevelBound}, the inclusion
$\mathcal{C}_p (x_p) \subset \K_p (\rho, v; m)$ implies the inequality
$m \le \lambda_p (x_p) + n_0 + m_0$. Therefore, for any $v$ there exist indices $i\neq j$ with
\[
f_{ij,v} (\K^{\ad}_p (p)) \not\subset p^{\lambda_p (x_p) + n_0 + m_0+1} \Z_p.
\]

Over the algebraic closure of $\Q$, there are only finitely many isomorphism classes of representations of $G$ of dimension $N$.
Since $\rho$ is necessarily $\bar{\Q}_p$-isomorphic to such a representation,
we can bound the degrees of the functions $f_{ij,v}$ in terms of $G$ and $N$.
Therefore, we can apply Corollary \ref{CorollaryVolumeBound} to estimate
\begin{multline*}
\vol \left( \{ k \in \K^{\ad}_p (p) \, : \, [k,x_p] \in \K_p (\rho, v; m)  \} \right)
\le \\\min_{i \neq j} \vol \left( \{ k \in \K^{\ad}_p (p) \, :  \, f_{ij,v} (k) \in p^m \Z_p \} \right)
\ll_G p^{-\varepsilon (m-\lambda_p (x_p) - n_0 - m_0)},
\end{multline*}
for suitable $\varepsilon$ as required.
\end{proof}

We are now ready to finish the proof of Proposition \ref{PropositionHighLevel} (and therefore of Theorem \ref{thm: mainbnd}).

\begin{proof}[Proof of Proposition \ref{PropositionHighLevel}]
Let $K_p$ be an open subgroup of $\K_p$ of level $p^n$. Without loss of generality we may assume that $n \ge 2$.
Let $x_p \in \K_p^{\ad}$.

First note that by the continuity of the map $[\cdot, \cdot]$, there exist integers $l_p \ge 1$ with the property that
$[\K^{\ad}_p(p^{l_p}), \K^{\ad}_p] \subset \K_p (p^{\epsilon_p})$.
Moreover, since by Lemma \ref{LemmaAlmostAllp}, for almost all $p$ the map $[\cdot,\cdot]$ descends to a map
$G^{\ad} (\F_p) \times G^{\ad} (\F_p) \to G (\F_p)$, we may take $l_p = 1$ for almost all $p$. We may then write
\begin{equation} \label{EqnPhiKp}
\phi_{K_p} (x_p) = \sum_\xi \vol \left( \{ k \in \K^{\ad}_p (p^{l_p}) \, : \, [k, \xi x_p \xi^{-1}] \in K_p \cap \K_p (p^{\epsilon_p}) \} \right),
\end{equation}
where the summation is over those classes in $\K^{\ad}_p (p^{l_p}) \bs \K^{\ad}_p$ for which there exists a representative $\xi$ with
$[\xi, x_p] \in K_p$.
By Lemma \ref{LemmaPointCounting}, the total number of summands in \eqref{EqnPhiKp} is clearly bounded by $p^{c''}$ for
a constant $c'' \ge 0$ depending only on $G$.

From Theorem \ref{TheoremAlgebraic}, applied to the group $K_p \cap \K_p (p^{\epsilon_p})$ of level $p^n$,
we obtain the existence of a proper, connected algebraic subgroup $X$ of $G$,
defined over $\Q_p$, such that
$K_p \cap \K_p (p^{\epsilon_p}) \subset (X (\Q_p) \cap \K_p) \K_p (p^{\lceil \varepsilon n \rceil})$, where $\varepsilon$ depends only on $G$.
Of course, here we may assume $X$ to be maximal among proper, connected subgroups defined over $\Q_p$, i.e., $X \in \sgrmx_{\Q_p} (G)$ in the notation of \S \ref{subsectionfiniteness}.
We can find a $\Q_p$-rational representation $\rho: G \to \GL (N)$, not containing the trivial representation, such that
$X$ is the stabilizer of a line $\Q_p v \subset \Q_p^N$. Moreover, by Corollary \ref{CorollaryMaximalFiniteness} (taken together with
Lemma \ref{LemmaFiniteness}) we can bound here $N$ in terms of $G$ only.
By conjugating $\rho$ and adjusting $v$, we can also assume that
$\rho (\K_p) \subset \GL (N, \Z_p)$ and $v \in \Z_p^N \setminus p \Z_p^N$.
In this situation we have from Lemma \ref{LevelBound} that
$\level (\K_p (\rho, v; m)) \le p^{m+m_1}$ with $m_1$ depending only on $G$ and $\rho_0$.

Clearly, we have now $X (\Q_p) \cap \K_p \subset \K_p (\rho, v; m)$ for any $m \ge 1$,
and in addition $\K_p (p^{\lceil \varepsilon n \rceil})\subset \K_p (\rho, v; m)$ for
\[
m := \lceil \varepsilon n \rceil - m_1,
\]
if $n > \frac{m_1}{\varepsilon}$, as we may assume.
Therefore,
\[
K_p \cap \K_p (p^{\epsilon_p}) \subset (X (\Q_p) \cap \K_p) \K_p (p^{\lceil \varepsilon n \rceil}) \subset \K_p (\rho, v; m).
\]
We can now apply Lemma \ref{LemmaKprho} to the individual summands in \eqref{EqnPhiKp} to obtain
\[
\phi_{K_p} (x_p) \le p^{c'' + \varepsilon' (c' + \lambda_p (x_p) - m)}.
\]
The proposition follows.
\end{proof}

\subsection{Some supplementary results} \label{SubsectionIntersections}
We turn to the proof of Corollary \ref{CorConjugacyClass}.
Until the end of this section we assume that $G$ and $\K$ are as in Corollary \ref{CorConjugacyClass}, i.e., that $G$ is a (possibly non-connected) reductive
group defined over $\Q$ whose derived group $G^{\der}$ and adjoint group $G^{\ad}$ are connected,
and that $\K = \rho_0^{-1} (\GL (N_0, \hat{\Z})) \subset G (\A_{\rm fin})$ for a faithful $\Q$-rational representation $\rho_0: G \to \GL (N_0)$.

\begin{proof}[Proof of Corollary \ref{CorConjugacyClass}]
Writing
\[
\vol \left( \{ k \in \K \, : \, k x k^{-1} \in K \} \right)\le
\sum_{y \in \K \cap \K_0 \backslash \K}
\vol \left( \{ k \in \K_0 \, : \, k y x y^{-1} k^{-1} \in K \} \right),
\]
we may first reduce to the case where $\K_0 = \K$ and $x \in \K$.

Let $G^{\rm sc}$ be the simply connected covering group of the derived group $G^{\der}$ of $G$
and let $p^{\rm sc}: G^{\rm sc} \to G$ be the associated canonical homomorphism. Let $\rho^{\rm sc}: G^{\rm sc} \to \GL (N^{\rm sc})$ be
a faithful $\Q$-rational representation such that $(p^{\rm sc})^{-1} (\K) \subset \K^{\rm sc} = (\rho^{\rm sc})^{-1} (\GL (N^{\rm sc}, \hat{\Z}))$.

We can assume that there exists $k_0 \in \K$ with $k_0 x k_0^{-1} \in K$, for otherwise the bound is trivial.
Then
\[
\vol \left( \{ k \in \K \, : \, k x k^{-1} \in K \} \right) = \vol \left( \{ k \in \K \, : \, [k, k_0 x k_0^{-1}] \in K \} \right).
\]
To estimate the right-hand side, we can apply Theorem \ref{thm: mainbnd} to the group $G^{\rm sc}$, the open subgroup
$(p^{\rm sc})^{-1} (K)$ of $\K^{\rm sc}$ and the image of $k_0 x k_0^{-1}$ in $G^{\ad} (\A_{\fin})$.
Here we let $\tilde{\K}_p$ be the image of $\K_p$ in $G^{\ad} (\Q_p)$ and take any representation $\rho_0^{\ad}$ such that
$\tilde{\K}_p \subset \K^{\ad}_p$ for all $p$.
In the final estimate we may replace the level of $(p^{\rm sc})^{-1} (K) \subset \K^{\rm sc}$ with respect to $\rho^{\rm sc}$ by $\level (K; G (\A_{\fin})^+)$, since the quotient of these
two quantities is bounded from above and below in terms of $G$, $\rho_0$ and $\rho_0^{\rm sc}$.
\end{proof}

We now give an alternative description of the functions $\lambda_p (x)$, as well as three lemmas on their behavior, which are useful in the application to limit multiplicities.

\begin{remark} \label{RemarkRestrictionOfScalars}
We can write $G^{\rm ad} = \prod_{i=1}^r G_i^{\ad}$ with $G_i^{\ad} = {\rm Res}_{F_i / \Q} \mathcal{G}_i^{\rm ad}$ for finite extensions $F_1, \ldots, F_r$ of $\Q$ and absolutely simple
adjoint $F_i$-groups $\mathcal{G}_i^{\rm ad}$ \cite[\S 6.21 (ii)]{MR0207712}.
The individual factors $G^{\ad}_i$ are then the $\Q$-simple factors of $G^{\ad}$.
The Lie algebra $\lie{g} = \bigoplus_{i=1}^r \lie{g}_i$ acquires naturally the structure of a module over the semisimple algebra $A = \prod_{i=1}^r F_i$.
Let $\mathfrak{o}_A$ be the ring of integers of $A$.
Furthermore, we can take $\K^{\rm ad} = \prod_v \K^{\rm ad}_v$, where $v$ ranges over the prime ideals of $\mathfrak{o}_A$, which implies
the factorizations $\K^{\rm ad}_p = \prod_{v|p} \K^{\rm ad}_v$ for all primes $p$.
Assume finally that $\Lambda$ is an $\mathfrak{o}_A$-lattice in $\lie{g}$. For $v|p$ let $\varpi_v$ be a prime element of
the local field $A_v$ and $e_v$ the corresponding ramification index. For $n \ge 0$ set $\K^{\ad}_v (\varpi_v^n) = \{ x_v \in \K^{\ad}_v \, : \, \Ad (x_v) l \equiv l \pmod{\varpi_v^n \Lambda} \,
\forall l \in \Lambda \}$.
Then for $x_p = (x_v)_{v|p} \in \K^{\rm ad}_p$ we can compute $\lambda_p (x_p)$
as the largest integer $n \ge 0$ for which there exists a place $v$ above $p$ with
$x_v \in \K^{\ad}_v (\varpi_v^{e_v n})$ (and $\lambda_p (x_p) = \infty$ if $x_v = 1$ for some $v|p$).
\end{remark}

\begin{definition}
We say that an element $\gamma\in G(\Q)$ is \nd\ if the smallest normal subgroup of $G$ containing $\gamma$ (which is necessarily defined over $\Q$)
contains the derived group $G^{\der}$.
\end{definition}

\begin{lemma} \label{LemmaLambdapFiniteness}
Let $\gamma \in G (\Q)$ be \nd. Then
$\lambda_p (\gamma) < \infty$ for all $p$ and $\lambda_p (\gamma) = 0$ for almost all $p$ (depending on $\gamma$).
\end{lemma}

\begin{proof}
If $\lambda_p (\gamma) = \infty$ for some $p$, then there exists a semisimple ideal $0 \neq \lie{h}$ of $\lie{g} \otimes \Q_p$ such that
$\Ad (\gamma)|_{\lie{h}} = 1$, i.e., $\gamma$ lies in the kernel of
the corresponding projection $\pi: G \to H$ of reductive algebraic groups defined over $\Q_p$ given by the action of $G$ on $\lie{h}$. This
contradicts the assumption on $\gamma$.

Consider the description of $\lambda_p$ given in Remark \ref{RemarkRestrictionOfScalars}. To have $\lambda_p (\gamma) > 0$ means
that $\Ad (\gamma) l \equiv l \pmod{\varpi_v \Lambda}$ for some place $v$ of $A$ above $p$. If this is the case for infinitely many $p$, then
we may conclude that the linear map $\Ad (\gamma)$ acts as the identity on a $\Q$-simple summand $\lie{g}_i$ for some $i$, which again contradicts the assumption on $\gamma$.
\end{proof}

\begin{lemma}
Let $P$ be a parabolic subgroup of $G$ defined over $\Q$, $U$ its unipotent radical, and $M$ a Levi subgroup of $P$ defined over $\Q$.
Let $\mu \in M (\Q)$ be \nd\ in $G$.
Then for every $p$ we have $\lambda_p(\mu u)\ll_{\mu}1$ for all $u\in U (\A_{\fin}) \cap \K$.
Moreover, for almost all $p$ (depending on $\mu$) we have $\lambda_p (\mu u) = 0$ for all $u \in U (\A_{\fin}) \cap \K$.
\end{lemma}

\begin{proof}
Assume on the contrary that $\lambda_p (\mu u)$ is unbounded for some $p$. Since $U (\A_{\fin}) \cap \K$ is compact and $p^{-\lambda_p (\mu \cdot)}$ is a continuous function, we can then find
$u \in U (\A_{\fin}) \cap \K$ with $\lambda_p (\mu u) = \infty$, or $\Ad (\mu u)|_{\lie{h}} = 1$ for some semisimple ideal $0 \neq \lie{h}$ of
$\lie{g} \otimes \Q_p$. Consider the corresponding projection $\pi: G \to H$ of reductive algebraic groups defined over $\Q_p$.
We have $\pi (\mu u) = 1$, and therefore $\pi (\mu) = \pi (u) = 1$, since $\pi (\mu)$ is semisimple and $\pi (u)$ unipotent.
This means that $\mu$ is contained in the kernel of $\pi$, a normal subgroup of $G$ not containing $G^{\der}$, which contradicts the assumption on $\mu$.

To show the second assertion, we may assume that $p$ is such that we are in the situation of Lemma \ref{LemmaAlmostAllp} (applied to $G^{\ad}$)
and that $\mu \in \K_p$. If $\lambda_p (\mu u) > 0$ for some $u \in U (\A_{\fin}) \cap \K$,
then we can apply the previous argument to the reduction modulo $p$ of $\Ad (\mu u) \in \K_p^{\ad}$ and conclude that we have in fact $\lambda_p (\mu) > 0$.
By Lemma \ref{LemmaLambdapFiniteness}, this is only possible for finitely many $p$ under our assumption on $\mu$.
\end{proof}

We now consider $\lambda_p$ on the unipotent radical $U$ of a parabolic subgroup $P$ of $G$ that is defined over $\Q$.
Recalling the description of Remark \ref{RemarkRestrictionOfScalars},
if $\lie{u}$ is the Lie algebra of $U$, then we can write
$\lie{u} = \bigoplus_{i=1}^r \lie{u}_i$, where each $\lie{u}_i$ is an $F_i$-vector space. Moreover, the spaces $\lie{u}_i$
are non-trivial for all $i = 1, \ldots, r$ if and only if the smallest normal
subgroup of $G$ containing $U$ is the full derived group $G^{\der}$.

\begin{lemma} \label{LemmaLambdaUnipotent}
Let $P$ be a parabolic subgroup of $G$ defined over $\Q$ and $U$ its unipotent radical, and assume that the smallest normal
subgroup of $G$ containing $U$ is the derived group $G^{\rm der}$. Let $\lie{u}$ be the Lie algebra of
$U$ and write $\lie{u} = \bigoplus_{i=1}^r \lie{u}_i$, where $\lie{u}_i$ is a non-trivial $F_i$-vector space.

Let $L_{\lie{u}} \subset \lie{u}$ be an $\mathfrak{o}_A$-lattice and set
\[
\tilde{\lambda}_p (\exp x) := \max \{ k \ge 0 \, : \, \exists v|p : x \in \varpi_v^{e_v k} L_{\lie{u}} \otimes \Z_p \}, \quad x \in L_{\lie{u}} \otimes \hat{\Z},
\]
where $v$ ranges over the places of $A$ lying over $p$, $\varpi_v$ is a prime element of the associated local field $A_v$ and $e_v$ the
corresponding ramification index.

Then we have the following:
\begin{enumerate}
\item For every prime $p$ the difference $\abs{\lambda_p (\exp x) - \tilde{\lambda}_p (\exp x)}$ is bounded on $L_{\lie{u}} \otimes \hat{\Z}$.
\item For almost all $p$ we have
$\lambda_p (\exp x) = \tilde{\lambda}_p (\exp x)$, $x \in L_{\lie{u}} \otimes \hat{\Z}$.
\item For a suitable choice of $L_{\lie{u}}$ we have
$\log (U (\A_{\fin}) \cap \K) \subset L_{\lie{u}} \otimes \hat{\Z}$.
\end{enumerate}
\end{lemma}

\begin{proof}
The third assertion is clear, since $\log (U (\A_{\fin}) \cap \K)$ is a compact subset of $\lie{u} \otimes \A_{\fin}$,
and is therefore contained in a set of the form $L_{\lie{u}} \otimes \hat{\Z}$.

To show the first and second assertions,
it suffices to consider the $\Q$-simple factors of $G^{\ad}$ one at a time. So, we may assume that $G^{\ad}$ is $\Q$-simple
and therefore of the form ${\rm Res}_{F / \Q} \mathcal{G}^{\rm ad}$ for an absolutely simple adjoint group $\mathcal{G}^{\rm ad}$ defined over a finite extension $F$ of $\Q$.
Consider the restriction to $\lie{u}$ of the adjoint representation $\ad: \lie{g} \to \End_F(\lie{g})$ of $F$-Lie algebras.
Clearly $\ad (L_{\lie{u}})$ is a finitely generated $\mathfrak{o}_F$-submodule of $\End_F(\lie{g})$ and therefore
there exists a positive integer $N$ such that $\ad (L_{\lie{u}}) \Lambda \subset N^{-1} \Lambda$.
Let $r$ be a positive integer such that $\ad (u)^i = 0$ for all $u \in \lie{u}$ and $i > r$.

If we now have $x \in \varpi_v^k L_{\lie{u}} \otimes \Z_p$ for some $k \ge 0$, then
we obtain $\ad (x) \Lambda \otimes \Z_p \subset N^{-1} \varpi_v^k \Lambda \otimes \Z_p$, and
therefore $\Ad (\exp x) = \exp \ad (x)$ satisfies
$\Ad (\exp x) l - l \in (r!)^{-1} N^{-r} \varpi_v^k \Lambda \otimes \Z_p$ for all $l \in \Lambda$.
We conclude that $\lambda_p (\exp x) \ge \tilde{\lambda}_p (\exp x) - v_p (r! N^r)$.
Using the logarithm map on $\ad (\lie{u})$, which is again given by polynomials, we obtain also the opposite inequality $\tilde{\lambda}_p (\exp x) \ge \lambda_p (\exp x) - v_p (M)$
for a suitable non-zero integer $M$.
Taken together, these inequalities amount to the first two assertions of the lemma.
\end{proof}

Finally, we show how to bound the unipotent orbital integrals which appear naturally in the limit multiplicity problem.

\begin{corollary}
There exists a constant $\varepsilon > 0$, depending only on $G$ and $\rho_0$, with the following property.
Let $\K_0$ be an open compact subgroup of $G (\A_{\rm fin})$.
Let $P$ be a parabolic subgroup of $G$ defined over $\Q$ and $U$ its unipotent radical, and let $H \subset G^{\der}$ be the smallest normal
subgroup of $G$ containing $U$. Then
\[
\int_{U (\A_{\fin})} \int_{\K} \mathbf{1}_K (k^{-1} u k) \ dk\, du \ll_{P, \, \K_0} \level (K, H (\A_{\fin})^{+})^{-\varepsilon}
\]
for all open subgroups $K \subset \K_0$.
\end{corollary}

\begin{proof}
Estimating the integral on the left-hand side by
\[
[\K_0 : \K \cap \K_0] \int_{U (\A_{\fin})} \int_{\K} \mathbf{1}_{K \cap \K} (k^{-1} u k) \ dk\, du,
\]
we may first reduce to the case where $K$ is contained in $\K$.

We may assume that the Haar measure on $U (\A_{\fin})$ is the product of the measures on $U (\Q_p)$ that
assign the open compact subgroups $U (\Q_p) \cap \K_p$ measure one.
Let $N = \level (K, H (\A_{\fin})^{+}) = \prod_p p^{n_p}$.
We apply Corollary \ref{CorConjugacyClass} to the group $H$ to estimate
\[
\int_{\K} \mathbf{1}_K (k^{-1} u k) \ dk \ll \beta (N, u, \delta)^{-\varepsilon} = \prod_p \beta_p (p^{n_p}, u_p, \delta)^{-\varepsilon}
\]
for all $u \in U (\A_{\fin}) \cap \K$, where
\[
\beta_p (p^{n_p}, u_p, \delta) =
\begin{cases} p^{n_p}, & \lambda_p^H (u_p) < \delta n_p, \\
              1, & \text{otherwise.} \\
\end{cases}
\]
We obtain
\begin{multline*}
\int_{U (\A_{\fin})} \int_{\K} \mathbf{1}_K (k^{-1} u k) \ dk\, du \ll \int_{U (\A_{\fin}) \cap \K}
\beta (N, u, \delta)^{-\varepsilon} du \\= \prod_{p|N} \int_{U (\Q_p) \cap \K_p} \beta_p (p^{n_p}, u_p, \delta)^{-\varepsilon} du_p.
\end{multline*}
To estimate the integral over $U (\Q_p) \cap \K_p$, we write
\begin{equation} \label{EqnIntegralUQp}
\int_{U (\Q_p) \cap \K_p} \beta_p (p^{n_p}, u_p, \delta)^{-\varepsilon} du_p = \mu_p (\lceil \delta n_p \rceil) + p^{- \varepsilon n_p} (1 - \mu_p (\lceil \delta n_p \rceil)),
\end{equation}
where
\[
\mu_p (n) = \vol \{ u_p \in U (\Q_p) \cap \K_p \, : \, \lambda_p^H (u_p) \ge n \}, \quad n \ge 0.
\]
Therefore it only remains to estimate $\mu_p (n)$.
It follows from Lemma \ref{LemmaLambdaUnipotent} that we have the bound $\mu_p (n) \ll_{G,U,\rho_0} p^{-n}$.
This yields a bound for \eqref{EqnIntegralUQp} of the form $C_p p^{-\varepsilon' n_p}$, where $\varepsilon', C_p >0$ and $C_p = 1$ for almost all $p$, which shows the assertion.
\end{proof}

\appendix
\section{Bounds for the number of solutions of polynomial congruences} \label{sec: numbersol}
In this appendix we give some simple general bounds for the number of solutions of polynomial congruences that
are essential for the argument in \S \ref{MainApplication} of the body of the paper. The emphasis lies on general non-trivial bounds which are uniform in all parameters
and are obtained by elementary methods, and not on the quality of the bounds.

Let $V = \Spec B_{\Z}$ be an affine scheme that is flat and of finite type over $\Spec \Z$. Fix once and for all generators $y_1, \ldots, y_r$ of
the affine coordinate ring $B_{\Z}$ of $V$, or equivalently a closed embedding of $V$ into
an affine space $\A^r$ over $\Z$.
\emph{Henceforth, all implied constants will depend on $V$ and on this embedding.}
For any commutative ring $R$ let $B_R=B_{\Z} \otimes R$ be the base change of the ring $B_{\Z}$ to $R$, $V_R = \Spec B_R$ the
base change of $V$, and $V(R)$ the set of $R$-points of $V$.
We assume throughout that the
generic fiber $V_{\Q}$ of $V$ is an absolutely irreducible variety, i.e., that the ring $B_{\bar{\Q}}$ is an integral domain. Let $s$ be the dimension of the variety $V_{\Q}$.

For any commutative ring $R$ and any integer $d \ge 0$ we let $B_{R,\le d}$ be the $R$-submodule of $B_R$ generated by the monomials of degree $\le d$ in the generators
$y_1, \ldots, y_r$.
We define the \emph{degree} $\deg f$ of an element $f \in B_R$ as the smallest integer $d \ge 0$ such that $f \in B_{R,\le d}$.

The main results of this appendix are the following two lemmas. The first lemma concerns polynomial congruences modulo $p$.

\begin{lemma} \label{LemmaVolumeBoundModp}
For all primes $p$ and all non-zero $f \in B_{\F_p}$ we have
\[
\# \{ x \in V (\F_p) \, : \, f (x) = 0 \} \ll_{V,y_i} \deg f\cdot p^{\dim V_{\Q} - 1}.
\]
\end{lemma}

For the next lemma, which treats congruences modulo $p^n$,
note that the given affine embedding $V \hookrightarrow \A^r$ induces for every $p$ a measure on the smooth part of the $p$-adic analytic set
$V (\Q_p) \subset \Q_p^r$ (cf. \cite[\S 3.3]{MR644559}, \cite[\S 3]{MR656627}).

\begin{lemma} \label{VolumeBound}
Assume that the generic fiber $V_{\Q}$ of $V$ is a smooth variety containing the origin of $\A^r$.
Then for each $d > 0$ there exists a constant $\varepsilon (d) >0$, depending only on $V$, such that
for any $f \in B_{\Q_p}$ of degree $\le d$ and
any $m,n\in\Z$ with $f (V (\Z_p))\not\subset p^m \Z_p$ we have
\[
\vol \left( \{ x \in V (\Z_p) \cap p \Z_p^r : f (x) \equiv 0 \pmod{p^n} \} \right) \ll_{V,y_i,d} p^{-\varepsilon (d) (n-m+1)}.
\]
\end{lemma}

\begin{remark}
For any Zariski closed subset $W \subset \A_{\F_p}^r$ defined by polynomials of degree $\le n$, the cardinality of $W(\F_p)$
can be bounded by $C (r,n) p^{\dim W}$ (cf.~\cite[Proposition 3.3]{MR1162433}, which is based on the Lang-Weil estimates).
This implies the statement of Lemma \ref{LemmaVolumeBoundModp} with $\deg f$ replaced by an unspecified function of $\deg f$,
which would be sufficient for our purposes. However, the proof given below is much simpler than the proof of the Lang-Weil estimates.
A proof of the lemma can also be obtained by combining \cite[Proposition 1]{MR656627} (applied to the set $\Omega = \F_p^r$) with basic intersection theory in the affine space $\A_{\F_p}^r$.
\end{remark}

\begin{remark}
In Lemma \ref{VolumeBound}, the smoothness of the variety $V_{\Q}$ is not essential, but it simplifies the argument considerably.
In the main part of the paper we are interested in the case of linear algebraic groups, which are smooth varieties.
\end{remark}

\begin{remark}
For fixed $p$ and $f$, the Poincar\'{e} series associated to the sequence
\[
v_n = \vol \left( \{ x \in V (\Z_p) \cap p \Z_p^r : f (x) \equiv 0 \pmod{p^n} \} \right), \quad n \ge 0,
\]
is given by a rational function \cite[Theorem 1]{MR1020825}. (In fact, we may allow general varieties $V_{\Q_p}$ and more general functions $f$ in this statement.) We will not make use of this fact.
\end{remark}

\begin{remark}
An estimate of the form $\ll_{V, f} p^{-n}$ for the volume of the subset
\[
W(f,n) = \{ x \in V (\Z_p) : \exists y \in V (\Z_p) : f (y) = 0, \, y \equiv x \pmod{p^n} \}
\]
of $V (\Z_p)$ (and for the volume of its intersection with $p \Z_p^r$) can be obtained from \cite[Th\'{e}or\`{e}me 1]{MR656627}. For $f \in B_{\Z_p}$ we obviously have the inclusion
$W (f, n) \subset \{ x \in V (\Z_p)  : f (x) \equiv 0 \pmod{p^n} \}$, but in general no equality.
\end{remark}

The proofs of Lemmas \ref{LemmaVolumeBoundModp} and \ref{VolumeBound} are based on the Noether normalization lemma
(e.g., \cite[Theorem 14.4]{MR0155856}), which gives the existence of algebraically independent elements
$x_1, \ldots, x_s \in B_\Z$ and of a positive integer $D$ such that $B_\Z [1/D]$ is integral over $A_\Z [1/D]$, where
$A_\Z = \Z [x_1, \ldots, x_s]$ is the polynomial ring over $\Z$ generated by $x_1, \ldots, x_s$ inside $B_\Z$. This means
that each generator $y_j$, $j = 1, \ldots, r$, satisfies a monic polynomial equation over $A_\Z [1/D]$.
In fact, one may take $x_1, \ldots, x_s$ to be elements of the $\Z$-module generated by $y_1, \ldots, y_r$ inside $B_\Z$ (cf. [ibid., \S 14, Exercise]).
Let
\[
\xi = (x_1,\ldots,x_s): \quad V \to \A^s
\]
be the associated finite morphism, and for any commutative ring $R$ write $\xi_R: V_R \to \A_R^s$ for its base change to $R$.
If we are given a field extension $K$ of $\Q$ and a smooth point $v$ of $V (K)$, then we may arrange that the morphism $\xi_K$ is smooth at $v$.
(To see this, note that if we write $x_1, \ldots, x_s$ as linear combinations of the generators $y_1, \ldots, y_r$ with coefficients $c_{ij} \in \Z$, then the conclusion of the normalization
lemma holds for all $(c_{ij})$ in an open dense subset of the space of all matrices. The smoothness condition also defines an open dense subset. Since the intersection of both subsets
is still open and dense, it contains a point with coordinates in $\mathbb{Z}$.)
In the following we write $A_R = A_{\Z} \otimes R = R [x_1, \ldots, x_s] \subset B_R$ for a commutative ring $R$.

Denote by $\QF (R)$ the quotient field of an integral domain $R$.
Recall a simple fact: if $A$ is an integrally closed integral domain and $L$ a finite extension of $K = \QF (A)$,
then the integral closure $B$ of $A$ in $L$ is precisely the set of all elements $x \in L$ whose minimal polynomial over $K$
has coefficients in $A$. In particular, we have then $N_{L/K} (x) \in A$ and $N_{L/K} (x) \in x A[x]$ for all $x \in B$.

We apply this to the field extension $\QF (B_{\Z_p}) / \QF (A_{\Z_p})$ for arbitrary $p$ and to the corresponding norm map
\[
N = N_{\QF (B_{\Z_p}) / \QF (A_{\Z_p})}:\QF (B_{\Z_p})\rightarrow \QF (A_{\Z_p}).
\]
We record the following facts.

\begin{lemma} \label{LemmaNoetherNorm}
Let $\tilde{A}_{\Z_p}$ be the integral closure of the polynomial ring $A_{\Z_p}$ inside the ring $B_{\Q_p}$. We have then:
\begin{enumerate}
\item $N (B_{\Q_p}) \subset A_{\Q_p}$.
\item $\Q_p \tilde{A}_{\Z_p} = B_{\Q_p}$.
\item $N (\tilde{A}_{\Z_p}) \subset A_{\Z_p}$.
\item \label{part: EqnNorm} $N (f) \in f \tilde{A}_{\Z_p} \quad \text{for all } f \in \tilde{A}_{\Z_p}$.
\item $\deg N (f) \ll \deg f$ for all $f \in B_{\Q_p}$.
\end{enumerate}
\end{lemma}

\begin{proof}
The first assertion follows immediately from the fact that $B_{\Q_p}$ is integral over $A_{\Q_p}$. The second to fourth assertions are a consequence of the remarks
preceding the lemma.
For the fifth assertion, recall that
each $y_j$, $j = 1, \ldots, r$, satisfies a monic polynomial equation over $A_\Z [1/D]$, of degree $n_j$, say.
Let $z_k$, $k=1,\dots,l$ be an enumeration of the monomials $y_1^{\alpha_1}\dots y_r^{\alpha_r}$ where $0\le\alpha_i<n_i$, $i=1,\dots,r$.
Then the $z_k$'s generate the $A_{\Q}$-module $B_{\Q}$, and therefore
also the $A_{\Q_p}$-module $B_{\Q_p}$ for any $p$. Moreover, an arbitrary monomial
of degree $n$ in $y_1,\dots,y_r$ can be expressed as a linear combination of the $z_k$'s with coefficients in $A_{\Q}$ of degree $\ll n$.
In addition, if $f \in B_{\Q_p}$ is expressed as a linear combination $\sum f_k z_k$ with $f_k \in A_{\Q_p}$, then
$N (f)$ is given by a fixed polynomial in the $f_k$'s with coefficients in $A_{\Q}$. This clearly shows the assertion.
\end{proof}

We now list several additional properties that hold (after fixing $x_1, \ldots, x_s$) for almost all $p$.
For all $p$ prime to the positive integer $D$ introduced above we have $\tilde{A}_{\Z_p} \supset B_{\Z_p}$
and the ring $B_{\F_p}$ is integral over $A_{\F_p} = \F_p [x_1, \ldots, x_s]$.
In addition, for almost all $p$ the scheme $V_{\F_p}$ is an absolutely irreducible variety over $\F_p$, and
in particular the ring $B_{\F_p}$ is an integral domain.
For all such $p$ we can consider the characteristic $p$ norm map
\[
\bar{N} = N_{\QF (B_{\F_p}) / \QF (A_{\F_p})}:\QF (B_{\F_p})^\times \rightarrow  \QF (A_{\F_p})^\times.
\]

\begin{lemma} \label{LemmaNoetherNormAlmostAllp}
We have for almost all $p$:
\begin{enumerate}
\item $N (B_{\Z_p}) \subset A_{\Z_p}$.
\item $\tilde{A}_{\Z_p} = B_{\Z_p}$.
\item $\bar{N} (B_{\F_p}) \subset A_{\F_p}$.
\item $\bar{N} (f) \in f B_{\F_p}$ for all $f \in B_{\F_p}$.
\item The diagram
\[
\begin{CD}
B_{\Z_p} @>N= N_{\QF (B_{\Z_p}) / \QF (A_{\Z_p})}>> A_{\Z_p}\\
@VVV           @VVV\\
B_{\F_p} @>\bar{N} = N_{\QF (B_{\F_p}) / \QF (A_{\F_p})}
>> A_{\F_p}
\end{CD}
\]
commutes, where the vertical lines are the reduction maps modulo $p$.
\end{enumerate}
\end{lemma}

\begin{proof}
The first, third and fourth assertions are an immediate consequence of the remarks preceding Lemmas \ref{LemmaNoetherNorm} and \ref{LemmaNoetherNormAlmostAllp}.
To prove the last assertion, fix a basis $b_1, \ldots, b_t$ of $\QF (B_\Z)$ over $\QF (A_\Z)$ consisting of elements of $B_\Z$.
For any $f \in B_\Z$ let $M (f)$ be the matrix (with entries in the field $\QF (A_\Z) = \Q (x_1, \ldots, x_s)$) representing multiplication
by $f$ in the basis $b_1, \ldots, b_t$. Then $N (f) = \det M (f)$.
Let $S$ be the finite set of all primes $p$ that divide the denominator of an entry of $M (y_j)$ for some $j = 1, \ldots, r$.
Since $B_\Z$ is generated by the elements $y_1, \ldots, y_r$ as a $\Z$-algebra, it follows that only primes in $S$ appear in the denominator of $M(f)$ for any
$f \in B_\Z$.
For almost all $p$, the images $\bar b_1, \ldots, \bar b_t$
of $b_1, \ldots, b_t$ in $B_{\F_p}$ form a basis of $\QF (B_{\F_p})$ over $\QF(A_{\F_p})$, and under the assumption $p \notin S$
we may reduce $M (f)$ modulo $p$ and obtain the matrix representing multiplication by $\bar f$ in the basis $\bar b_1, \ldots, \bar b_t$.
Taking determinants we obtain the fifth assertion.

As for the second assertion, it only remains to show that $\tilde{A}_{\Z_p} \subset B_{\Z_p}$ for almost all $p$.
Assuming the contrary, there exists $f \in B_{\Z_p}$, $f \notin p B_{\Z_p}$, with $p^{-1} f \in \tilde{A}_{\Z_p}$. By Lemma \ref{LemmaNoetherNorm},
we have then $N (p^{-1} f) = p^{-t} N (f) \in A_{\Z_p}$, while on the other hand $\overline{N (f)} = \bar{N} (\bar{f}) \neq 0$ by the fifth assertion of the present lemma,
which is a contradiction. This proves the second assertion.
\end{proof}

We can now prove the first of the two lemmas stated at the beginning.

\begin{proof}[Proof of Lemma \ref{LemmaVolumeBoundModp}]
Our proof is based on the following elementary estimate (see \cite[Ch. IV, Lemma 3A]{MR0429733})
for the cardinality of the zero set in $\F_p^s$ of a non-zero polynomial $g \in \F_p [x_1,\ldots,x_s]$:
\begin{equation} \label{EqnSchmidtBound}
\# \{ x \in \F_p^s \, : \, g (x) = 0 \}  \le (\deg g) p^{s-1}.
\end{equation}

In the situation of Lemma \ref{LemmaVolumeBoundModp}
it is clearly enough to prove the assertion for almost all $p$ (depending on $V$).
We can therefore assume the existence of algebraically independent elements $x_1, \ldots, x_s \in B_{\Z}$ such that the assertions of Lemma \ref{LemmaNoetherNormAlmostAllp} hold.
Let $\bar{\xi} = \xi_{\F_p}: V_{\F_p} \to \A_{\F_p}^s$ be the reduction modulo $p$ of the morphism $\xi$.
Note that since the generators $y_j$ satisfy some fixed monic equations with coefficients in $A_\Z [1/D]$,
the cardinality of the fibers of $\bar{\xi}$ is
bounded by a constant that is independent of $p$ (namely by the product of the degrees of these equations).

By Lemma \ref{LemmaNoetherNormAlmostAllp}, for any $0\ne f \in B_{\F_p}$ we have $g = \bar{N} (f) \in A_{\F_p} = \F_p [x_1,\ldots,x_s]$, $g \neq 0$, and $g \in f B_{\F_p}$. Therefore, the zero set of
$f$ is contained in the zero set of $g \circ \bar{\xi}$.
Invoking \eqref{EqnSchmidtBound}, this implies that
\begin{align*}
\# \{ v \in V (\F_p) \, : \, f (v) = 0 \}  & \le \# \{ v \in V (\F_p) \, : \, g (\bar{\xi}(v)) = 0 \} \\
& \ll \# \{ x \in \F_p^s \, : \, g (x) = 0 \}  \\
& \le (\deg g) p^{s-1}.
\end{align*}
Combining the last assertions of Lemmas \ref{LemmaNoetherNorm} and \ref{LemmaNoetherNormAlmostAllp},
we have the degree bound $\deg g \ll \deg f$ and obtain the assertion.
\end{proof}

We now turn to the proof of Lemma \ref{VolumeBound}. Again, we first consider the case of an affine space
and replace \eqref{EqnSchmidtBound} by the following estimate.

\begin{lemma} \label{LemmaPolynomialCongruence}
Let $f \in \Z_p [X_1, \ldots, X_s]$ be a polynomial of degree $\le d$ and suppose that $f$ is not congruent to zero modulo $p$. Then
we have
\begin{equation} \label{eq: bndsolmodpm}
\vol \left( \{ x \in \Z_p^s: f (x) \equiv 0 \pmod{p^n} \} \right) \le d^s \binom{n+s-1}{s-1} p^{-\frac{n}{d}}
\end{equation}
for any $n \ge 0$.
In other words, the number of solutions to $f(x_1,\dots,x_s)=0$ in $(\Z/p^n\Z)^s$ is bounded by $d^s \binom{n+s-1}{s-1} p^{n(s-\frac1d)}$.
\end{lemma}

\begin{remark}
This result is essentially \cite[Lemma 7.1]{MR932848}. For convenience we include a proof.
Asymptotic estimates for the left-hand side of \eqref{eq: bndsolmodpm}
as $n \to \infty$ have been obtained by Igusa \cite{MR0404215,MR0441933} and Loeser \cite{MR842046}. Here, we are just interested in a non-trivial bound that
depends in a very transparent way on the polynomial $f$.
\end{remark}

\begin{proof}
We may assume that $n>0$.
We prove the statement by induction on the number of variables $s$.
In the case $s=1$ we can factorize $f$ over a suitable algebraic extension $F$ of $\Q_p$ as
\[
f(x)=\prod_{i=1}^d(\alpha_ix+\beta_i),
\]
where $\max(\abs{\alpha_i},\abs{\beta_i})=1$ for all $i = 1, \ldots, d$ and $\abs{\cdot}$ denotes the extension of the standard absolute value of $\Q_p$ to $F$.
If $x\in\Z_p$ and $\abs{f(x)}\le p^{-n}$ then $\abs{\alpha_ix+\beta_i}\le p^{-n/d}$ for some $i$, where necessarily $\abs{\alpha_i}=1$.
Thus
\[
\vol \left( \{ x \in \Z_p: f (x) \equiv 0 \pmod{p^n} \} \right) \le dp^{-\frac{n}{d}},
\]
as claimed.

For the induction step we write $f=\sum_{i=0}^df_i(x_1,\dots,x_{s-1})x_s^i$ where at least one of the polynomials $f_i$, say $f_{i_0}$,
is non-zero modulo $p$.
Fix $x_1,\dots,x_{s-1}$. Let
\[
j=\max\{l\ge0:f_i(x_1,\dots,x_{s-1})\equiv0\pmod{p^l}\text{ for all }i\}.
\]
Applying the one-variable case to $p^{-j}f(x_1,\dots,x_{s-1},\cdot)$ we get
\[
\vol(\{x_s\in\Z_p:f(x_1,\dots,x_{s})\equiv0\pmod{p^n}\})\le\min(dp^{-\frac{n-j}d},1).
\]
Thus, the left-hand side of \eqref{eq: bndsolmodpm} is bounded by
\begin{align*}
&\sum_{j=0}^n dp^{-\frac{n-j}d}\vol\left( \{ x \in \Z_p^{s-1}: f_i(x_1,\dots,x_{s-1}) \equiv 0 \pmod{p^j}\text{ for all } i \} \right) \\
\le d&\sum_{j=0}^n p^{-\frac{n-j}d}\vol\left( \{ x \in \Z_p^{s-1}: f_{i_0}(x_1,\dots,x_{s-1}) \equiv 0 \pmod{p^j}\} \right).
\end{align*}
It remains to apply the induction hypothesis and the binomial identity $\sum_{j=0}^n\binom{j+s-2}{s-2}=\binom{n+s-1}{s-1}$.
\end{proof}

\begin{proof}[Proof of Lemma \ref{VolumeBound}]
Note first that it is enough to establish the following two claims.
\begin{enumerate}
\item For every $p$ and any $v \in V (\Z_p)$ there exists an integer $N(p,v) \ge 0$ such that
for all $d\ge1$ we have
\[
\vol \left( \{ u \in V (\Z_p) \cap ( v + p^{N(p,v)} \Z_p^r ) : f (u) \equiv 0 \pmod{p^n} \} \right) \ll_{p,d,v} p^{-\varepsilon (p, d, v) (n-m+1)}
\]
for any $f \in B_{\Q_p}$ of degree $\le d$ and $m,n\in\Z$ with $f (V (\Z_p)) \not\subset p^m \Z_p$,
where $\varepsilon (p,d,v)>0$ depends on $p$, $d$ and $v$.
\item The assertion of the lemma is true for almost all $p$ (depending only on $V$).
\end{enumerate}

Namely, granted the second claim, it only remains to consider the assertion of the lemma for finitely many primes $p$, which we may treat one at a time.
For each such $p$ we invoke the first claim and cover the compact set
$V (\Z_p) \cap p \Z_p^r$ by finitely many sets of the form $V (\Z_p) \cap ( v + p^{N(p,v)} \Z_p^r )$. It is then enough to sum the resulting
estimates to obtain the assertion.

It remains to prove the two claims. We start with the first one.
Let $p$ be arbitrary and $v \in V (\Z_p)$. As noted above, we can choose $x_1, \ldots, x_s$ (depending on $p$ and $v$) such that the associated morphism $\xi_{\Q_p}: V_{\Q_p} \to \A_{\Q_p}^s$ is smooth at $v$.
This implies the existence of an integer $N(p,v) \ge 0$ such that $\xi_{\Q_p}$ induces an injection of
$V (\Z_p) \cap ( v + p^{N(p,v)} \Z_p^r )$ into $\Z_p^s$ that multiplies volumes by a positive constant $C (p,v)$.

For any $p$ and $f \in B_{\Q_p}$, $f \neq 0$, set
 \[
\nu_p (f) := \max \{ i \in \Z : p^{-i} f \in \tilde{A}_{\Z_p} \}.
\]
Clearly, $v_p (f(x)) \ge \nu_p (f)$ for all $x \in V (\Z_p)$.
In the situation of the first claim we have therefore $\nu_p (f) \le m-1$.
We are reduced to prove the estimate
\begin{equation} \label{EqnVolumeClaim}
\vol \left( \{ u \in V (\Z_p) \cap ( v + p^{N(p,v)} \Z_p^r ) : f (u) \equiv 0 \pmod{p^n} \} \right) \ll_d p^{-\varepsilon (p,d,v) n}
\end{equation}
for all $f \in B_{\Q_p}$ of degree $\le d$ with
$\nu_p (f) = 0$, or equivalently for
all $f \in \tilde{A}_{\Z_p} \setminus p \tilde{A}_{\Z_p}$ with $\deg f \le d$.

Let $\tilde{A}_{\le d} = \tilde{A}_{\Z_p} \cap B_{\Z_p,\le d}$, which is a free $\Z_p$-module of finite rank.
Since $N (f) \neq 0$ for all $f \neq 0$, and the set $\tilde{A}_{\le d} \setminus p \tilde{A}_{\le d}$ is compact, we see that
\[
n (p, d) := \max \{ v_p (N (f)) \, : \, f \in \tilde{A}_{\le d} \setminus p \tilde{A}_{\le d} \} < \infty.
\]
(Here $v_p$ is the usual $p$-valuation on $A_{\Z_p}$.)

It follows from Lemma \ref{LemmaNoetherNorm}, part \ref{part: EqnNorm}, that the set of all $u \in V (\Z_p)$ for which $f (u) \in p^n \Z_p$
is contained in the set of all $u$ with $F (\xi (u)) \in p^n \Z_p$, where
$F = N (f) \in A_{\Z_p}$.
This implies that
\begin{align*}
& \vol \left( \{ u \in V (\Z_p) \cap ( v + p^{N(p,v)} \Z_p^r ) : f (u) \in p^n \Z_p \} \right)  \\
\mbox{} \le &
\vol \left( \{ u \in V (\Z_p) \cap ( v + p^{N(p,v)} \Z_p^r ) : F (\xi(u)) \in p^n \Z_p \} \right) \\
\mbox{} \le & C (p,v)^{-1} \vol \left( \{ x \in \Z_p^s : F (x) \in p^n \Z_p \} \right).
\end{align*}
Invoking Lemma \ref{LemmaPolynomialCongruence} and using that $v_p (F) \le n (p,d)$, we obtain the desired estimate \eqref{EqnVolumeClaim}.
This finishes the proof of the first claim.

We now consider the second claim.
By a suitable choice of $x_1, \ldots, x_s$ we can arrange that the morphism $\xi_{\Q}: V_{\Q} \to \A_{\Q}^s$ is smooth at the origin.
For almost all $p$ the map $\xi_{\Q_p}$ induces then a volume preserving injection of $V (\Z_p) \cap p \Z_p^r$ into $\Z_p^s$.
Moreover, by Lemma \ref{LemmaNoetherNormAlmostAllp} for almost all $p$ we have
$\tilde{A}_{\Z_p} = B_{\Z_p}$ and furthermore $f \in \tilde{A}_{\Z_p} \setminus p \tilde{A}_{\Z_p} = B_{\Z_p} \setminus p B_{\Z_p}$ implies
that $\overline{N(f)} = \bar{N} (\bar{f}) \neq 0$, or $v_p (N(f)) = 0$.
This means that $n (p,d) = 0$ for all $d$.

For all such $p$ we can replace in the previous argument $V (\Z_p) \cap ( v + p^{N(p,v)} \Z_p^r )$
by $V (\Z_p) \cap p \Z_p^r$, and $C(p,v)$ by $1$, and since we have $n (p,d) = 0$, we obtain
\[
\vol \left( \{ u \in V (\Z_p) \cap p \Z_p^r : f (u) \equiv 0 \pmod{p^n} \} \right) \ll_d p^{-\varepsilon (d) n}
\]
for all $f \in B_{\Q_p}$ of degree $\le d$ with
$\nu_p (f) = 0$, which establishes the second claim and finishes the proof.
\end{proof}

\bigskip
\footnotesize
\noindent\textit{Acknowledgments.}
T.F. partially supported by DFG Heisenberg grant \# FI 1795/1-1.
E.L. partially supported by grant \#711733 from the Minerva Foundation.





\newcommand{\etalchar}[1]{$^{#1}$}
\providecommand{\bysame}{\leavevmode\hbox to3em{\hrulefill}\thinspace}
\providecommand{\MR}{\relax\ifhmode\unskip\space\fi MR }
\providecommand{\MRhref}[2]{%
  \href{http://www.ams.org/mathscinet-getitem?mr=#1}{#2}
}
\providecommand{\href}[2]{#2}

\end{document}